\documentclass[11pt]{article}
\usepackage[all,2cell]{xy} \UseAllTwocells \SilentMatrices \SelectTips{cm}{}
\usepackage{latexsym,amsfonts,amssymb}
\usepackage{amsmath,amsthm,amscd}
\usepackage[dvips]{epsfig}
\usepackage{psfrag}

\usepackage{diagcat} 
\usepackage{color}
\usepackage{etoolbox}
\usepackage{hyperref}
\usepackage{graphicx}
\usepackage{a4wide}
\normalfont\upshape

\usepackage{fancyhdr}
\newcommand{\cwbubble}[2]{
\begin{DGCpicture}
\DGCcoupon*(-0.4,-0.4)(.4,.4){ }
\DGCbubble(0,0){0.35}
\DGCdot*<{0.2,L}
\DGCdot*.{0.1,R}[r]{#2}
\DGCcoupon*(-.4,-.4)(.4,.4){\small{#1}}
\end{DGCpicture}
}

\newcommand{\ccwbubble}[2]{
\begin{DGCpicture}
\DGCcoupon*(-0.4,-0.4)(.4,.4){ }
\DGCbubble(0,0){0.35}
\DGCdot*>{0.2,L}
\DGCdot*.{0.1,R}[r]{#2}
\DGCcoupon*(-.4,-.4)(.4,.4){\small{#1}}
\end{DGCpicture}
}

\newcommand{\bigcwbubble}[2]{
\begin{DGCpicture}
\DGCcoupon*(-0.8,-0.8)(.8,.8){ }
\DGCbubble(0,0){0.5}
\DGCdot*<{0.25,L}
\DGCdot*.{0.25,R}[r]{#2}
\DGCcoupon*(-.6,-.6)(.6,.6){\small{#1}}
\end{DGCpicture}
}

\newcommand{\bigccwbubble}[2]{
\begin{DGCpicture}
\DGCcoupon*(-0.8,-0.8)(.8,.8){ }
\DGCbubble(0,0){0.5}
\DGCdot*>{0.25,L}
\DGCdot*.{0.25,R}[r]{#2}
\DGCcoupon*(-.6,-.6)(.6,.6){\small{#1}}
\end{DGCpicture}
}

\newcommand{\cwI}{
\begin{DGCpicture}
\DGCbubble(0,0){.3}
\DGCdot*<{.25,L}
\DGCcoupon*(-.5,-.5)(.5,.5){\small{$1$}}
\end{DGCpicture}}

\newcommand{\cwcapbubcup}[4]{
\begin{DGCpicture}
\DGCcoupon*(-0.3,-0.1)(1.3,2.1){ }
\ifstrequal{#4}{no}{}{
\DGCcoupon*(0,.8)(.3,1.4){#4}}
\ifstrequal{#1}{no}{}{
\DGCstrand(0,0)(1,0)/d/
\DGCdot*>{0.25,2}
\ifstrequal{#1}{$0$}{}{
\ifstrequal{#1}{$1$}{\DGCdot{0.25,1}}{
\DGCdot{0.25,1}[r]{\mbox{\scriptsize #1}}}}}
\ifstrequal{#3}{no}{}{
\DGCstrand/d/(0,2)(1,2)
\DGCdot*<{1.75,2}
\ifstrequal{#3}{$0$}{}{
\ifstrequal{#3}{$1$}{\DGCdot{1.75,1}}{
\DGCdot{1.75,1}[r]{\mbox{\scriptsize #3}}}}}
\ifstrequal{#2}{no}{}{
\DGCbubble(1,1){.3}
\DGCdot*>{1.2,L}
\DGCcoupon*(.65,.65)(1.35,1.35){\small{#2}}}
\end{DGCpicture}
}

\newcommand{\ccwcapbubcup}[4]{
\begin{DGCpicture}
\DGCcoupon*(-0.3,-0.1)(1.3,2.1){ }
\ifstrequal{#4}{no}{}{
\DGCcoupon*(0,.8)(.3,1.4){#4}}
\ifstrequal{#1}{no}{}{
\DGCstrand(0,0)(1,0)/d/
\DGCdot*<{0.25,2}
\ifstrequal{#1}{$0$}{}{
\ifstrequal{#1}{$1$}{\DGCdot{0.25,1}}{
\DGCdot{0.25,1}[r]{\mbox{\scriptsize #1}}}}}
\ifstrequal{#3}{no}{}{
\DGCstrand/d/(0,2)(1,2)
\DGCdot*>{1.75,2}
\ifstrequal{#3}{$0$}{}{
\ifstrequal{#3}{$1$}{\DGCdot{1.75,1}}{
\DGCdot{1.75,1}[r]{\mbox{\scriptsize #3}}}}}
\ifstrequal{#2}{no}{}{
\DGCbubble(1,1){.3}
\DGCdot*<{1.2,L}
\DGCcoupon*(.65,.65)(1.35,1.35){\small{#2}}}
\end{DGCpicture}
}

\newcommand{\RIII}[5]{
\begin{DGCpicture}[scale=0.85]
\DGCcoupon*(-0.3,-0.3)(2.3,2.3){}
\ifstrequal{#5}{no}{}{
\DGCcoupon*(2.1,.95)(2.3,1.15){#5}}
\DGCstrand(0,0)(2,2)
\DGCdot*>{2}
\ifstrequal{#4}{$0$}{}{
\ifstrequal{#4}{$1$}{\DGCdot{1.7}}{
\DGCdot{1.7}[r]{\mbox{\scriptsize #4}}}}
\DGCstrand(2,0)(0,2)
\DGCdot*>{2}
\ifstrequal{#2}{$0$}{}{
\ifstrequal{#2}{$1$}{\DGCdot{1.7}}{
\DGCdot{1.7}[r]{\mbox{\scriptsize #2}}}}
\ifstrequal{#1}{L}{\DGCstrand(1,0)(0,1)(1,2)}{\DGCstrand(1,0)(2,1)(1,2)}
\DGCdot*>{2}
\ifstrequal{#3}{$0$}{}{
\ifstrequal{#3}{$1$}{\DGCdot{1.7}}{
\DGCdot{1.7}[r]{\mbox{\scriptsize #3}}}}
\end{DGCpicture}
}

\newcommand{\twolines}[3]{
\begin{DGCpicture}
\DGCcoupon*(-.3,-.3)(1.3,1.3){}
\ifstrequal{#3}{no}{}{\DGCcoupon*(1.1,.6)(1.4,.9){#3}}
\DGCstrand(0,0)(0,1)
\DGCdot*>{1}
\ifstrequal{#1}{$0$}{}{
\ifstrequal{#1}{$1$}{\DGCdot{.4}}{
\DGCdot{.4}[r]{\mbox{\scriptsize #1}}}}
\DGCstrand(1,0)(1,1)
\DGCdot*>{1}
\ifstrequal{#2}{$0$}{}{
\ifstrequal{#2}{$1$}{\DGCdot{.4}}{
\DGCdot{.4}[r]{\mbox{\scriptsize #2}}}}
\end{DGCpicture}
}

\newcommand{\twolinesD}[3]{
\begin{DGCpicture}
\DGCcoupon*(-.3,-.3)(1.3,1.3){}
\ifstrequal{#3}{no}{}{\DGCcoupon*(-.4,.6)(-.1,.9){#3}}
\DGCstrand(0,0)(0,1)
\DGCdot*<{0}
\ifstrequal{#1}{$0$}{}{
\ifstrequal{#1}{$1$}{\DGCdot{.6}}{
\DGCdot{.6}[r]{\mbox{\scriptsize #1}}}}
\DGCstrand(1,0)(1,1)
\DGCdot*<{0}
\ifstrequal{#2}{$0$}{}{
\ifstrequal{#2}{$1$}{\DGCdot{.6}}{
\DGCdot{.6}[r]{\mbox{\scriptsize #2}}}}
\end{DGCpicture}
}

\newcommand{\crossing}[5]{
\begin{DGCpicture}
\DGCcoupon*(-.3,-.3)(1.3,1.3){}
\ifstrequal{#5}{no}{}{
\DGCcoupon*(1,.4)(1.3,.7){#5}}
\DGCstrand(0,0)(1,1)
\DGCdot*>{1}
\ifstrequal{#2}{$0$}{}{
\ifstrequal{#2}{$1$}{\DGCdot{.7}}{
\DGCdot{.7}[r]{\mbox{\scriptsize #2}}}}
\ifstrequal{#3}{$0$}{}{
\ifstrequal{#3}{$1$}{\DGCdot{.3}}{
\DGCdot{.3}[r]{\mbox{\scriptsize #3}}}}
\DGCstrand(1,0)(0,1)
\DGCdot*>{1}
\ifstrequal{#1}{$0$}{}{
\ifstrequal{#1}{$1$}{\DGCdot{.7}}{
\DGCdot{.7}[r]{\mbox{\scriptsize #1}}}}
\ifstrequal{#4}{$0$}{}{
\ifstrequal{#4}{$1$}{\DGCdot{.3}}{
\DGCdot{.3}[r]{\mbox{\scriptsize #4}}}}
\end{DGCpicture}
}

\newcommand{\crossingD}[5]{
\begin{DGCpicture}
\DGCcoupon*(-.3,-.3)(1.3,1.3){}
\ifstrequal{#5}{no}{}{\DGCcoupon*(-.3,.4)(0,.7){#5}}
\DGCstrand(0,0)(1,1)
\DGCdot*<{0}
\ifstrequal{#2}{$0$}{}{
\ifstrequal{#2}{$1$}{\DGCdot{.3}}{
\DGCdot{.3}[r]{\mbox{\scriptsize #2}}}}
\ifstrequal{#3}{$0$}{}{
\ifstrequal{#3}{$1$}{\DGCdot{.7}}{
\DGCdot{.7}[r]{\mbox{\scriptsize #3}}}}
\DGCstrand(1,0)(0,1)
\DGCdot*<{0}
\ifstrequal{#1}{$0$}{}{
\ifstrequal{#1}{$1$}{\DGCdot{.3}}{
\DGCdot{.3}[r]{\mbox{\scriptsize #1}}}}
\ifstrequal{#4}{$0$}{}{
\ifstrequal{#4}{$1$}{\DGCdot{.7}}{
\DGCdot{.7}[r]{\mbox{\scriptsize #4}}}}
\end{DGCpicture}
}

\newcommand{\oneline}[2]{
\begin{DGCpicture}
\DGCcoupon*(-.3,-.1)(0.3,2.1){}
\DGCstrand(0,0)(0,2)
\DGCdot*>{2}
\ifstrequal{#2}{no}{}{\DGCcoupon*(.1,1.4)(.4,1.7){#2}}
\ifstrequal{#1}{$0$}{}{
\ifstrequal{#1}{$1$}{\DGCdot{1}}{
\DGCdot{1}[r]{\mbox{\scriptsize #1}}}}
\end{DGCpicture}
}

\newcommand{\onelineD}[2]{
\begin{DGCpicture}
\DGCcoupon*(-.3,-.1)(0.3,2.1){}
\DGCstrand(0,0)(0,2)
\DGCdot*<{0}
\ifstrequal{#2}{no}{}{\DGCcoupon*(-.4,1.4)(-.1,1.7){#2}}
\ifstrequal{#1}{$0$}{}{
\ifstrequal{#1}{$1$}{\DGCdot{1}}{
\DGCdot{1}[r]{\mbox{\scriptsize #1}}}}
\end{DGCpicture}
}

\newcommand{\onelineshort}[2]{
\begin{DGCpicture}
\DGCcoupon*(-.3,-.1)(0.3,1.1){}
\DGCstrand(0,0)(0,1)
\DGCdot*>{1}
\ifstrequal{#2}{no}{}{\DGCcoupon*(.1,.7)(.4,1){#2}}
\ifstrequal{#1}{$0$}{}{
\ifstrequal{#1}{$1$}{\DGCdot{.5}}{
\DGCdot{.5}[r]{\mbox{\scriptsize #1}}}}
\end{DGCpicture}
}

\newcommand{\onelineDshort}[2]{
\begin{DGCpicture}
\DGCcoupon*(-.3,-.1)(0.3,1.1){}
\DGCstrand(0,0)(0,1)
\DGCdot*<{0}
\ifstrequal{#2}{no}{}{\DGCcoupon*(-.4,.7)(-.1,1){#2}}
\ifstrequal{#1}{$0$}{}{
\ifstrequal{#1}{$1$}{\DGCdot{.5}}{
\DGCdot{.5}[r]{\mbox{\scriptsize #1}}}}
\end{DGCpicture}
}

\newcommand{\curl}[5]{
\begin{DGCpicture}
\ifstrequal{#1}{L}{
\DGCcoupon*(-2,-.8)(.3,1.8){}
\DGCstrand(0,-.5)(0,.25)/u/(-1.5,.5)/d/(0,.75)/u/(0,1.5)
\ifstrequal{#4}{no}{}{\DGCcoupon*(-1.4,0)(-.4,1){\small{#4}}}
\ifstrequal{#5}{$0$}{}{\ifstrequal{#5}{$1$}{\DGCdot{.5,4}}{\DGCdot{.5,4}[l]{\mbox{\scriptsize #5}}}}
\ifstrequal{#3}{no}{}{\DGCcoupon*(-1.2,1.1)(-.9,1.4){#3}}
}{
\DGCcoupon*(-.3,-.8)(2,1.8){}
\DGCstrand(0,-.5)(0,.25)/u/(1.5,.5)/d/(0,.75)/u/(0,1.5)
\ifstrequal{#4}{no}{}{\DGCcoupon*(.4,0)(1.4,1){\small{#4}}}
\ifstrequal{#5}{$0$}{}{\ifstrequal{#5}{$1$}{\DGCdot{.5,4}}{\DGCdot{.5,4}[r]{\mbox{\scriptsize #5}}}}
\ifstrequal{#3}{no}{}{\DGCcoupon*(.9,1.1)(1.2,1.4){#3}}
}
\ifstrequal{#2}{D}{\DGCdot*<{-0.25} \DGCdot*<{1.25} \DGCdot*<{.75,2}}{\DGCdot*>{-0.25} \DGCdot*>{1.25} \DGCdot*>{.75,2}}
\end{DGCpicture}
}

\newcommand{\cappy}[4]{
\begin{DGCpicture}
\DGCcoupon*(-.3,-.1)(1.3,.8){}
\DGCstrand(0,0)(1,0)/d/
\ifstrequal{#1}{CCW}{\DGCdot*<{0,2}}{\DGCdot*>{0,1}}
\ifstrequal{#2}{$0$}{}{\ifstrequal{#2}{$1$}{\DGCdot{.3}}{\DGCdot{.3}[r]{\mbox{\scriptsize #2}}}}
\ifstrequal{#3}{no}{}{
\DGCbubble(1.7,.4){0.3}
\ifstrequal{#3}{CCW}{\DGCdot*>{.7,L}}{\DGCdot*<{.7,L}}
\DGCcoupon*(1.35,0)(2.05,.8){\small{$1$}}}
\ifstrequal{#4}{no}{}{\DGCcoupon*(.8,.5)(1.2,.8){#4}}
\end{DGCpicture}
}

\newcommand{\cuppy}[4]{
\begin{DGCpicture}
\DGCcoupon*(-.3,0.2)(1.3,1.1){}
\DGCstrand(0,1)/d/(1,1)/u/
\ifstrequal{#1}{CCW}{\DGCdot*>{1,1}}{\DGCdot*<{1,2}}
\ifstrequal{#2}{$0$}{}{\ifstrequal{#2}{$1$}{\DGCdot{.7}}{\DGCdot{.7}[r]{\mbox{\scriptsize #2}}}}
\ifstrequal{#3}{no}{}{
\DGCbubble(1.7,.6){0.3}
\ifstrequal{#3}{CCW}{\DGCdot*>{.9,L}}{\DGCdot*<{.9,L}}
\DGCcoupon*(1.35,.2)(2.05,1){\small{$1$}}}
\ifstrequal{#4}{no}{}{\DGCcoupon*(.8,.2)(1.2,.5){#4}}
\end{DGCpicture}
}

\newcommand{\bottomcurl}[5]{
\begin{DGCpicture}
\DGCcoupon*(-.3,-.1)(1.3,1.6){}
\DGCstrand(0,0)(1,1)/u/(0,1)/d/(1,0)/d/
\ifstrequal{#1}{CW}{\DGCdot*>{1.3}}{\DGCdot*<{1.3}}
\ifstrequal{#2}{no}{}{\DGCcoupon*(0,0.65)(1,1.3){\mbox{\scriptsize #2}}}
\ifstrequal{#3}{yes}{\DGCdot{.3,2}}{}
\ifstrequal{#4}{no}{}{
\DGCbubble(1.5,0.6){.3}
\ifstrequal{#4}{CCW}{\DGCdot*>{.8,L}}{\DGCdot*<{.8,L}}
\DGCcoupon*(1.3,0.4)(1.7,0.8){\small{$1$}}}
\ifstrequal{#5}{no}{}{\DGCcoupon*(1.1,0.9)(1.5,1.5){#5}}
\end{DGCpicture}
}

\newcommand{\crossingR}[4]{
\begin{DGCpicture}
\DGCcoupon*(-.3,-.3)(1.3,1.3){}
\DGCstrand(0,0)(1,1)
\DGCdot*>{1}
\ifstrequal{#1}{no}{}{\DGCdot{.65}}
\DGCstrand(1,0)(0,1)
\DGCdot*<{0}
\ifstrequal{#2}{no}{}{\DGCdot{.65}}
\ifstrequal{#3}{no}{}{
\DGCbubble(1.2,.5){0.3}
\DGCdot*<{.7,R}
\DGCcoupon*(0.9,0.2)(1.5,.8){\small{$1$}}}
\ifstrequal{#4}{no}{}{
\DGCcoupon*(-.3,.2)(.3,.8){#4}}
\end{DGCpicture}
}

\newcommand{\crossingL}[4]{
\begin{DGCpicture}
\DGCcoupon*(-.3,-.3)(1.3,1.3){}
\DGCstrand(0,0)(1,1)
\DGCdot*<{0}
\ifstrequal{#1}{no}{}{\DGCdot{.35}}
\DGCstrand(1,0)(0,1)
\DGCdot*>{1}
\ifstrequal{#2}{no}{}{\DGCdot{.35}}
\ifstrequal{#3}{no}{}{
\DGCbubble(-.2,.5){0.3}
\DGCdot*<{.7,L}
\DGCcoupon*(-.5,0.2)(.1,.8){\small{$1$}}}
\ifstrequal{#4}{no}{}{
\DGCcoupon*(.7,.2)(1.3,.8){#4}}
\end{DGCpicture}
}

\newcommand{\EC}{\mathcal{E}}
\newcommand{\VC}{\mathcal{V}}
\newcommand{\FC}{\mathcal{F}}
\newcommand{\UC}{\mathcal{U}}

\renewcommand{\a}{\alpha}
\renewcommand{\b}{\beta}
\newcommand{\g}{\gamma}

\renewcommand{\l}{\lambda}
\newcommand{\mfsl}{\mathfrak{sl}}

\newcommand{\co}{\colon}
\newcommand{\odd}{\textrm{odd}}
\newcommand{\even}{\textrm{even}}

\newcommand{\oa}{\overline{a}}
\newcommand{\ox}{\overline{x}}
\newcommand{\oy}{\overline{y}}
\newcommand{\od}{\overline{d}}

\theoremstyle{plain}
\newtheorem{thm}{Theorem}[section]
\newtheorem{prop}[thm]{Proposition}
\newtheorem{lemma}[thm]{Lemma}
\newtheorem{cor}[thm]{Corollary}

\theoremstyle{definition}

\newtheorem{defn}[thm]{Definition}
\newtheorem{eg}[thm]{Example}
\newtheorem{rmk}[thm]{Remark}

\numberwithin{equation}{section}

\usepackage{bbm}
\def\1{\mathbbm 1}

\def\Z{\mathbb Z}
\def\N{\mathbb N}
\def\C{\mathbb C}
\def\F{\mathbb F}
\def\o{\otimes}
\def\lra{\longrightarrow}

\def\RHOM{\mathbf{R}\mathrm{HOM}}

\def\mc{\mathcal}
\def\mf{\mathfrak}

\def\End{\mathrm{End}}
\def\END{\mathrm{END}}
 
\newcommand{\dmod}{\!-\!\mathrm{mod}}

\newcommand{\mker}{\mathrm{Ker}}
\newcommand{\mim}{\mathrm{Im}}
\newcommand{\mH}{\mathrm{H}}  
\newcommand{\NH}{\mathrm{NH}} 
\newcommand{\pol}{\mathrm{Pol}}

\newcommand{\sym}{\mathrm{Sym}}

\newcommand{\Hom}{{\rm Hom}}
\newcommand{\HOM}{{\rm HOM}}

\newcommand{\Ind}{{\rm Ind}}

\def\dif{{\partial}}

\def\Ind{{\mathrm{Ind}}}
\def\lra{{\longrightarrow}}
\def\dmod{{\mathrm{-mod}}}   

\def\END{{\mathrm{END}}}

\def\mc{\mathcal}
\def\mf{\mathfrak}
\def\shuffle{\,\raise 1pt\hbox{$\scriptscriptstyle\cup{\mskip
               -4mu}\cup$}\,}

\newcommand{\refequal}[1]{\xy {\ar@{=}^{#1}
(-1,0)*{};(1,0)*{}};
\endxy}

\title{An approach to categorification of some small quantum groups II}
\author{Ben Elias, You Qi}

\date{\today}
\begin{document}
%

\maketitle

\begin{abstract}
We categorify an idempotented form of quantum $\mf{sl}_2$ and some of its simple representations at a prime root of unity.
\end{abstract}

\setcounter{tocdepth}{2}
\tableofcontents

\section{Introduction}

\paragraph{Background.}

In 1994, Crane and Frenkel proposed \cite{CF} that certain 3-dimensional topological field theories (3d TQFTs) could be categorified, and lifted to become 4d TQFTs. This goal has motivated a great deal of mathematics in the recent decades, including the categorification of many structures in quantum representation theory. We hope that this paper will be a stepping stone in an effort to translate the successes in categorified representation theory back into the original context of Crane and Frenkel's conjectures.

The Jones polynomial \cite{Jones} is a polynomial in $q$ associated to a link in the three-sphere. While defined combinatorially, it has a representation-theoretic interpretation involving the semisimple quantum group $U_q(\mf{sl}_2)$ at a generic complex value $q$. Witten \cite{Witten} gave a field-theoretic explanation of the Jones polynomial when $q$ is evaluated at a root of unity. Witten's construction uses Chern-Simons theory \cite{CS}, and yields an invariant of the three-manifold via a path integral. Subsequently, Reshetikhin and Turaev \cite{RT} gave a combinatorial approach to the Jones polynomial and the corresponding three-manifold invariant, using the representation theory of the
quantum group $u_q(\mf{sl}_2)$ at a root of unity. This three-manifold invariant admits generalizations to quantum groups in other types, which are now referred to as Witten-Reshetikhin-Turaev (WRT) invariants. They are also closely related to combinatorially defined three-manifold invariants constructed by Kuperberg \cite{Kup1,Kup2}, which instead use ``one-half'' of the quantum group $u^+_q(\mf{sl}_2)$. It appears that link invariants arising from quantum groups can be defined for generic $q$, but can be extended to three-manifold invariants only when $q$ is a root of unity.

Khovanov's ground-breaking discovery of link homology \cite{KhJones,KhTangleCob}, now known as Khovanov homology, gave strong evidence for the plausibility of Crane and Frenkel's conjecture that WRT invariants could be categorified. Khovanov homology gives a combinatorial categorification of the Jones polynomial at a generic variable $q$, replacing multiplication by $q$ by a categorical grading shift. Moreover, it is functorial with respect to (framed) link cobordisms, so that the whole theory appears to be the restriction of an extended 4d TQFT to the category of links embedded in the three-sphere and embedded link cobordisms in the three sphere times an interval.

The entire representation-theoretic interpretation of the Jones polynomial has since been categorified, describing Khovanov homology in the
framework of 2-representation theory (\cite{Rou2, Lau1, KL3}). Khovanov-Lauda \cite{KL1,KL2} and Rouquier \cite{Rou2} independently constructed monoidal categories that categorify one-half of quantum groups at generic $q$. Lauda \cite{Lau1} categorified the entire quantum group $U_q(\mf{sl}_2)$. Webster \cite{Web,Web2} categorified the tensor products of irreducible representations of quantum groups (adapting the geometric picture of Zheng \cite{Zheng1,Zheng2}), which are used to construct Reshetikhin-Turaev invariants of links. There are also more recent works by Cautis \cite{CautisClasp} and Lauda-Queffelec-Rose \cite{LQR}, in which Khovanov homology is directly reconstructed from $2$-representations of categorified quantum $\mf{sl}_2$.

A fundamental obstacle to transforming Khovanov homology into a categorified TQFT is that it relates to generic $q$ rather than evaluating $q$ at a root of unity. More precisely, Khovanov homology takes values in the category of graded vector spaces; the Grothendieck ring of graded vector spaces is identified with $\Z[q,q^{-1}]$. On the other hand, whatever categorified quantum three-manifold invariants are, they should live in a symmetric monoidal category whose Grothendieck ring is identified with the cyclotomic integers $\mc{O}_n := \Z[\zeta_n]$, where $\zeta_n$ represents a primitive $n$-th root of unity (see, for instance, \cite{BeLe}). One needs to construct the 2-representation theory of $u_q(\mf{sl}_2)$ at a root of unity, taking values in some categorification of the cyclotomic integers. Categorifying $\mc{O}_n$ is the natural first step, and it remains an interesting and fundamental open problem.

Bernstein-Khovanov \cite{Hopforoots} observed, however, that one could categorify $\mc{O}_p$ for $p$ prime: it can be identified with the Grothendieck ring of Mayer's \cite{Mayer1, Mayer2} homotopy category of finite dimensional $p$-complexes over a field $\Bbbk$ of characteristic $p>0$, which we denote by $\mc{C}(\Bbbk, \dif)$:
\begin{equation}
\mc{O}_p\cong K_0(\mc{C}(\Bbbk, \dif)).
\end{equation}
Khovanov suggested that one should find interesting algebra-objects in the homotopy category of $p$-complexes. Such algebra objects are given by $\Bbbk$-algebras with $p$-nilpotent derivations, which we refer to as \emph{$p$-DG algebras}. The Grothendieck group of a $p$-DG algebra will naturally be an $\mc{O}_p$-module. The theory of $p$-DG algebras has
been developed in \cite{QYHopf}, within the framework of \emph{hopfological algebra} \cite{Hopforoots}.

To agree with earlier conventions in representation theory, the derivation or \emph{differential} in a $p$-complex should have degree $2$, not degree $1$. With this convention, the Grothendieck ring of $p$-complexes is now isomorphic to
\begin{equation}
K_0(\mc{C}(\Bbbk, \dif))\cong \Z[q]/(1+q^2+\cdots +q^{2(p-1)}),
\end{equation}
which we will denote as $\mathbb{O}_p$. Inside $\mathbb{O}_p$, $q^2$ is a primitive $p$-th root of unity.

In this paper (and its prequel) we put $p$-DG algebras to work, and take the first steps towards constructing the 2-representation theory of $u_q(\mf{sl}_2)$ at a (prime) root of unity. Essentially, we take the algebras constructed by Khovanov-Lauda, Rouquier, Webster etc.~and find their $p$-DG analogues. In the prequel \cite{KQ}, a $p$-DG structure was placed on the nilHecke algebras to categorify one-half of the quantum group $u_q^+(\mf{sl}_2)$. In this paper, we place a $p$-DG structure on the $2$-category $\UC$ defined by Lauda \cite{Lau1}, and categorify the small form of the whole quantum group $u_q(\mf{sl}_2)$. In a sequel to this paper, the authors will categorify the Lusztig form of $U_q(\mf{sl}_2)$ at a prime root of unity as well, by placing a $p$-DG structure on the ``thick calculus" of \cite{KLMS}. Optimistically, we expect that, in a fully established 4d TQFT as envisioned by Crane and Frenkel, categorified small quantum groups will play an analogous role to gauge groups in Chern-Simons theory.

\paragraph{Outline.} To a $p$-DG algebra one can associate a number of module categories: the (abelian) category of $p$-DG modules, and the (triangulated) homotopy category and derived category. By definition, the Grothendieck group of a $p$-DG algebra is the Grothendieck group of its (triangulated) compact derived category, which could be quite difficult to compute in general. We briefly discuss what needs to be shown in order to prove that a given $p$-DG algebra has the desired Grothendieck group.

In some known (additive or abelian) categorifications of quantum groups \cite{KL1,Lau1} and Hecke algebras \cite{EliasKh}, one has a complete collection $\mathbb{X}$ of indecomposable objects (up to grading shift), such that the endomorphism ring $R=\End(\oplus_{M \in \mathbb{X}} M)$ is positively graded, and consists only of identity maps in degree zero. Equivalently, the indecomposables descend to a basis (such as the canonical basis) with certain positivity properties, relative to an Euler form on the Grothendieck group. In this case, the entire category is Morita-equivalent to $R\dmod$, and therefore it shares a Grothendieck group with the degree zero part $R_0$, which is semisimple. From this one can easily deduce many facts about the Grothendieck group.

Our goal is to use the same trick for $p$-DG algebras. In Chapter \ref{sec-positivepDGalgebras} we develop some general machinery to compute the Grothendieck groups of $p$-DG algebras which are $p$-DG Morita-equivalent to positively graded $p$-DG algebras. In this case one can relate the $p$-DG Grothendieck group to the ordinary Grothendieck group of the underlying algebra. Moreover, a K\"{u}nneth formula holds, allowing one to compare Grothendieck groups of tensor products of $p$-DG algebras. As an important example, we compute the $p$-DG Grothendieck group of the $p$-DG algebra of symmetric functions in Chapter \ref{section-symmetric-functions}. We expect our machinery to be widely (though not universally) applicable in categorification at prime roots of unity.

In order to show that a list of objects $\mathbb{X}$ yields an endomorphism algebra that is Morita-equivalent to the original category, one typically constructs a number of direct
sum decompositions. Algorithmically, one factors the identity morphism of any object as a sum of idempotents, each of which factors through an object in $\mathbb{X}$. In the
$p$-DG context, one needs a structure which is finer than a direct sum decomposition, which we dub a \emph{fantastic filtration}, or a \emph{Fc-filtration} for short. In Chapter
\ref{sec-standardfiltrations} we define this structure, give an algorithmic determination for when a factorization is a Fc-filtration, and prove that all the direct sum
decompositions that Lauda uses in \cite{Lau1} have this property. This feature of our paper should also be widely useful.

Let us describe the structure of this paper. In Chapter \ref{sec-positivepDGalgebras} we provide background on $p$-DG algebras in general, and develop tools to analyze the
Grothendieck rings of positive $p$-DG algebras. In Chapter \ref{section-symmetric-functions} we study the $p$-DG algebra of symmetric functions, which plays an important role in
Lauda's category $\UC$. In Chapter \ref{sec-pdifferentialsonU} we define $\UC$ and a multiparameter family of $p$-differentials, each of which equips $\UC$ with the structure of a
$p$-DG category. This family of $p$-differentials is exhaustive so long as $p \ne 2$. The classification of $p$-differentials is proven in Appendix
\ref{sec-classification-of-dif}, where the exceptional $2$-differentials are also briefly discussed. There are two particular specializations of the parameters which are
important, and which we denote $\dif_1$ and $\dif_{-1}$. In Chapter \ref{sec-standardfiltrations} we verify that the direct sum decompositions of Lauda yield Fc-filtrations if and only if $\dif=\dif_{\pm 1}$. In Chapter \ref{sec-decategorfication} we prove our main results.
We define an $\mathbb{O}_p$-integral form of the small quantum group $\dot{u}_{\mathbb{O}_p}(\mf{sl}_2)$, and establish an isomorphism (Theorem \ref{thm-itsanalghom})
$$K_0(\UC,\dif_{\pm 1}) \cong \dot{u}_{\mathbb{O}_p}(\mf{sl}_2)$$
of $\mathbb{O}_p$-algebras. As an application we also construct a categorification of some simple representations of $\dot{u}_{\mathbb{O}_p}$ of small highest weights (Theorem \ref{thm-cyclo-quot}),
using (universal) cyclotomic quotients. We leave a fuller discussion of categorification of finite dimensional highest weight $\mf{sl}_2$-modules to the sequel \cite{EQ2}.

\paragraph{Acknowledgements.} Both authors would like to express their deep gratitude towards their advisor Mikhail Khovanov for many years of inspirational education, and for supporting this collaboration. They would also like to thank the anonymous referee for helpful suggestions. 

The second author would like to thank Bernhard Keller for helpful discussions about DG categories during the introductory workshop ``Cluster Algebras and Commutative Algebra'', as well as the hospitality of the organizers of the workshop and MSRI. He would also like to thank Mikhail Kapranov for some enlightening discussion about $n$-complexes. The string diagrams in the paper are created via Krzysztof Putyra's LaTeX package \verb"diagcat".

The first author is supported by NSF fellowship DMS-1103862. The second author is partially supported by NSF RTG grant DMS-0739392 at Columbia.


\section{Positive \texorpdfstring{$p$}{p}-DG algebras}\label{sec-positivepDGalgebras}

Fix once and for all a base field $\Bbbk$ of characteristic $p > 0$. The unadorned tensor product $\o$ denotes the tensor product over $\Bbbk$. Throughout we adopt the French convention that $\N:=\{0,1,2,\dots\}$.

%
\subsection{Elements of \texorpdfstring{$p$}{p}-DG algebras}
%

We briefly review the basic notions of $p$-DG algebra. For details, see \cite{Hopforoots, QYHopf}. See also \cite[Chapter 2]{KQ} for an alternative summary of this material.

\begin{defn}\label{def-postive-pdg-algebra}Let $A\cong \oplus_{k\in \Z}A^k$ be a $\Z$-graded algebra over the ground field $\Bbbk$.
\begin{itemize}
\item[(i)] $A$ is called a \emph{$p$-differential graded ($p$-DG) algebra} if $A$ is equipped with an endomorphism $\dif_A$ of degree 2 which is \emph{$p$-nilpotent}, i.e. $\dif_A^p=0$, and satisfies the \emph{Leibniz rule}, i.e. for any $u,v\in A$,
    \[\dif_A(uv)=\dif_A(u)v+u\dif_A(v).\]
    A \emph{(left) $p$-DG module} $(M, \dif_M)$ over the $p$-DG algebra $A$ is a graded $A$-module $M$ with a linear endomorphism $\dif_M$ of degree 2 which is $p$-nilpotent, i.e. $\dif_M^p=0$, and satisfies the Leibniz rule for the $(A,\dif_A)$ action, i.e. for any $u\in A$, $m\in M$
    \[\dif_{M}(um)=\dif_A(u)m+u\dif_M(m).\]
    Analogously one has the notion of right $p$-DG modules. A morphism of $p$-DG modules is a map of $A$-modules that commutes with the differentials.
\item[(ii)] A $p$-DG algebra $A$ is called \emph{positive} if the following three conditions hold:
\begin{itemize}
\item[(ii.1)] $A$ is supported on non-negative degrees: $A\cong \oplus_{k\in \N}A^k$, and it is finite dimensional in each degree.
\item[(ii.2)] The homogeneous degree zero part $A^0$ is semisimple.
\item[(ii.3)] The differential $\dif_A$ acts trivially on $A^0$.
\end{itemize}
\item[(iii)]A positive $p$-DG algebra $A$ is called \emph{strongly positive} over $\Bbbk$ if $A^0$ is isomorphic to a product of matrix algebras over $\Bbbk$.
\end{itemize}
\end{defn}

Let $A$ be a positive $p$-DG algebra, and $\{\epsilon_i| i\in I\}$ be a complete list of pairwise non-isomorphic, indecomposable idempotents in $A^0$. Let $A^\prime:=\oplus_{k>0}A^k$ be the augmentation ideal with respect to the natural projection $A\twoheadrightarrow A^0$. Define for each $i\in I$ the indecomposable projective $A$-module $P_i:=A\cdot \epsilon_i$, and the simple $A$-module $S_i:=P_i/(A^\prime \cdot P_i)\cong A^0\cdot \epsilon_i$. When $A$ is strongly positive, the endomorphism algebra of each $S_i$ is isomorphic to $\Bbbk$.

\begin{eg}We give some easy examples of positive $p$-DG algebras.
\begin{itemize}
\item The ground field $\F_p$ equipped with the zero differential is strongly positive, while the extension filed $\F_{p^r} (r>1)$ with the trivial differential is positive but not strongly positive over $\F_p$.
\item Let $A:=\Bbbk[Q]/(R)$ be the path algebra associated with some oriented quiver $Q$ modulo a set $R$ of homogeneous relations, and let $c$ be a homogenous degree two element such that $c^p\in (R)$. Define a differential $\dif$ on $A$ by taking the commutator with $c$. Then $(A, \dif)$ is a strongly positive $p$-DG algebra.
\item The algebra of symmetric functions (equipped with a certain differential) is a strongly positive $p$-DG algebra, as is the algebra of symmetric polynomials in $n$ variables.  We will consider these algebras in detail in Chapter \ref{section-symmetric-functions}.
\end{itemize}
\end{eg}

The collection of all $p$-DG modules over a $p$-DG algebra $A$ forms an abelian category, which we denote by $A_\dif\dmod$. This category is
equipped with a grading shift endo-functor $\{1\}$, where $M\{1\}^k = M^{k+1}$. Given two $p$-DG modules $M,N$, we write $\Hom_A^i(M,N) =
\Hom^0_A(M,N\{i\})$ for the space of $A$-module maps from $M$ to $N$ of degree $i$. Then we set
\begin{equation}\label{eqn-total-hom}
\HOM_A(M,N) = \oplus_{i \in \Z} \Hom_A^i(M,N),
\end{equation}
which is a graded vector space. Similarly, we write $\Hom^i_{A_\dif}(M,N)$ for those homogeneous degree-$i$ maps which commute with the differentials, while letting $\HOM_{A_\dif}(M,N)$ be the total graded space. The graded space $\HOM_A(M,N)$ is equipped with a $p$-complex structure via
\begin{equation}\label{eqn-dif-on-total-hom}
\dif(f)(m) = \dif_N(f(m)) - f(\dif_M(m))
\end{equation}
for $f \in \HOM_A(M,N)$. Clearly $\HOM_{A_\dif}(M,N)$ is the kernel of $\dif$ inside $\HOM_A(M,N)$. Given $f \co M \lra N$ and $g \co N \lra P$, one can check that $\dif(g \circ f) = \dif(g)\circ f + g \circ \dif(f)$, so that this differential on Hom spaces still satisfies the Leibniz rule. It is easy to compute that $\dif^{p-1}(f) = \sum_{i=0}^{p-1}\dif_N^i\circ f \circ \dif_M^{p-1-i}$, because we work over a field of characteristic $p$.

Two morphisms $f_1,f_2:M\lra N$ of $p$-DG modules are called \emph{homotopic} if there exists an $A$-module map $h:M\lra N$ of degree $2-2p$ such that
\[
f_1-f_2=\dif^{p-1}(h)=\sum_{i=0}^{p-1}\dif_N^i\circ h \circ \dif_M^{p-1-i}.
\]
A morphism is called \emph{null-homotopic} if it is homotopic to zero. Null-homotopic morphisms form an ideal within $A_\dif\dmod$. The \emph{homotopy category} $\mc{C}(A,\dif)$ (or simply $\mc{C}(A)$ if the differential is understood) is the categorical quotient of $A_\dif\dmod$ by all null-homotopic morphisms. The category $\mc{C}(A)$ is triangulated, and we refer the reader to \cite[Section 2.2]{KQ} for more details about the triangulated structure.

The \emph{trivial} $p$-DG algebra is the ground field $\Bbbk$ with zero differential. Its categories of modules $\Bbbk_\dif\dmod$ and $\mc{C}(\Bbbk)$ will be called respectively the abelian and homotopy categories of \emph{$p$-complexes}. For an arbitrary $p$-DG algebra $A$, there is a forgetful functor $A_\dif\dmod\lra\Bbbk_\dif\dmod$ which takes a $p$-DG $A$-module to its underlying $p$-complex. We refer to the homotopy class of the underlying $p$-complex (viewed as an isomorphism class in $\mc{C}(\Bbbk)$) as the \emph{cohomology}\footnote{There is an explicit construction, called the \emph{slash cohomology}, of a minimal $p$-complex $\mH_/(M)$ representing the cohomology of $M$. See \cite[Section 2.1]{KQ} for details.} of $M$. A morphism $f:M\lra N$ of $p$-DG modules over $A$ is called a \emph{quasi-isomorphism} if, after applying the forgetful functor, $f$ descends to an isomorphism in $\mc{C}(\Bbbk)$. Quasi-isomorphisms in $\mc{C}(A)$ constitute a localizing class, and inverting them yields the derived category $\mc{D}(A,\dif)$ (or $\mc{D}(A)$ if the differential is understood). As in the usual DG case, $\mc{D}(A)$ is triangulated and idempotent complete (Karoubian).

\begin{rmk}\label{rmk-smash-product-algebra} Set $H:=\Bbbk[\dif]/(\dif^p)$, a graded Hopf algebra where $\deg(\dif)=2$. A $p$-DG algebra is the same as a graded $H$-module algebra. Thus for such an $A$ one can construct the \emph{smash
product algebra} $A_\dif$, which is isomorphic to $A \o H$ as a vector space. For the trivial $p$-DG algebra we have $\Bbbk_\dif \cong H$, and
in general $A_\dif$ always has $\Bbbk_\dif \cong 1 \o H$ as a subalgebra. A $p$-DG module over $A$ is the same as a module over $A_\dif$. This allows one to study $p$-DG algebras in the more general context of hopfological algebra, see \cite{Hopforoots,QYHopf}.

Because of the Hopf structure, one can tensor a $p$-DG module for $A$ with a $p$-complex to obtain another $p$-DG module for $A$. This gives an action of $\Bbbk_\dif\dmod$ on $A_\dif\dmod$ by exact functors, descending to a categorical action of $\mc{C}(\Bbbk)$ on $\mc{C}(A)$ and $\mc{D}(A)$.
\end{rmk}

As in usual homological algebra, it is not easy to determine the space of morphisms in the derived category, but one can simplify matters by restricting to a well-behaved class of objects.

\begin{defn}\label{def-nice-p-DG-mod}Let $A$ be a $p$-DG algebra and $M$ a $p$-DG module over it.
\begin{itemize}
\item[(i)] $M$ is said to have \emph{property (P)} if there is an exhaustive, possibly infinite, increasing filtration $F^\bullet$ on $M$ such that each subquotient $F^\bullet/F^{\bullet -1}$ is isomorphic to a direct sum of $p$-DG modules of the form $A\{r\}$ for various $r\in \Z$. When $A$ is positive, we extend this definition to allow subquotients isomorphic to $P_i\{r\}$ for $i \in I$ and $r \in \Z$.
\item[(ii)] $M$ is a \emph{cofibrant module} if $M$ is a direct summand of a property (P) module.
\end{itemize}
\end{defn}

Any property (P) module is obviously cofibrant. Notice that cofibrant modules are always projective as $A$-modules. If $M$ is a cofibrant module and $N$ is any $p$-DG module over $A$, then
\begin{equation} \label{eqn-mor-space-homotopy-category}
\Hom_{\mc{D}(A)}(M,N) \cong \Hom_{\mc{C}(A)}(M,N) \cong \dfrac{\mathrm{Ker}(\dif:\Hom_A^0(M,N)\lra \Hom_A^{2}(M,N))}{\mathrm{Im}(\dif^{p-1}:\Hom_A^{2-2p}(M,N)\lra \Hom_A^0(M,N))}.
\end{equation}

The following result says that there are always ``enough'' property (P) modules.
\begin{thm}\label{thm-bar-resolution}Let $A$ be a $p$-DG algebra and $M$ a $p$-DG module. Then there is a surjective quasi-isomorphism of of $p$-DG modules over $A$
\[\mathbf{p}(M)\lra M,\]
where $\mathbf{p}(M)$ has property (P). $\hfill\square$
\end{thm}
The property (P) replacement in the theorem is also known as the \emph{bar resolution} of $M$, and it can be defined functorially.

We recall the definition of \emph{finite cell modules} for positive $p$-DG algebras (c.f. \cite[Definition 2.29]{KQ}).

\begin{defn}\label{def-finite-cell}Let $A$ be a positive $p$-DG algebra. A $p$-DG module $M$ is said to be a \emph{finite cell module} if there is a finite-step increasing filtration $F^\bullet$ on $M$ such that each subquotient
$F^\bullet(M)/F^{\bullet-1}(M)$ is either zero or isomorphic to $P_i\{l_i\}$ for some $i\in I$ and $l_i\in \Z$. The collection of all finite cell modules is denoted by $\mc{F}(A)$.
\end{defn}

\begin{rmk}\label{rmk-convolution-finite-cell}
By definition, if $M$ is a finite cell module, it fits into a convolution diagram in the homotopy category $\mc{C}(A)$ and derived category $\mc{D}(A)$
\begin{equation}\label{eqn-convolution-diagram}
\begin{gathered}
\xymatrix@C=0.95em{0=F_{n-1} \ar[rr] && F_{n} \ar[rr] \ar[dl] && F_{n+1} \ar@{-}[r]\ar[dl] &\cdots\ar[r] &F_{m-1} \ar[rr] && F_m=M, \ar[dl]\\
& Gr_n\ar[ul]_{[1]} && Gr_{n+1} \ar[ul]_{[1]} && && Gr_{m-1}\ar[ul]_{[1]} &
}
\end{gathered}
\end{equation}
where $n$ is the smallest integer that $F_n\neq 0$, $m$ is the smallest integer that $F_m=M$, and  $Gr_k := F^k(M)/F^{k-1}(M)$ is either $0$ or $P_i\{l_i\}$ for some $i\in I$ and $l_i\in \Z$.
\end{rmk}

Any finite cell module $M$ satisfies property (P) and thus is cofibrant. Moreover, considered as an object in the derived category, a finite cell module is \emph{compact} in the sense that, if $(N_i)_{i\in I}$ is an arbitrary set of $p$-DG modules, then the canonical map
\[
\oplus_{i\in I}\Hom_{\mc{D}(A)}(M, N_i)\lra \Hom_{\mc{D}(A)}(M,\oplus_{i\in I}N_i)
\]
is an isomorphism. We have the following characterization of compact modules in $\mc{D}(A)$ (\cite[Corollary 7.15]{QYHopf}).

\begin{thm}\label{thm-compact-mod}Let $A$ be a $p$-DG algebra. An object $N\in \mc{D}(A)$ is compact if and only if it is isomorphic to a direct summand of a finite cell module. \hfill$\square$
\end{thm}

In other words, if we set $\mc{D}^c(A)$ to be the strictly full (triangulated) subcategory of $\mc{D}(A)$ consisting of compact modules, it is equivalent to the idempotent completion of $\mc{F}(A)$ in $\mc{D}(A)$. In particular $\mc{D}^c(A)$ is idempotent complete. The following characterization of $\mc{D}^c(A)$ is a direct consequence of Theorem \ref{thm-compact-mod}.

\begin{cor}\label{cor-compact-derived-category-smallest}Let $A$ be a $p$-DG algebra. Then $\mc{D}^c(A)$ is the smallest strictly full idempotent complete triangulated subcategory in $\mc{D}(A)$ containing $A\{r\}$ for all $r\in \Z$. \hfill$\square$
\end{cor}

We will also need the \emph{finite derived category} $\mc{D}^f(A)$ of $A$, which is the strictly full subcategory in $\mc{D}(A)$ consisting of objects that are quasi-isomorphic to a finite dimensional module in $A_\dif\dmod$. One easily checks that $\mc{D}^f(A)$ is triangulated.

The Grothendieck group of the triangulated category $\mc{D}^c(A)$ will be denoted $K_0(A,\dif)$ (or $K_0(A)$). The Grothendieck group of $\mc{D}^f(A)$ will be denoted
$G_0(A,\dif)$ (or $G_0(A)$). In this paper, these are the most relevant notions of Grothendieck groups for a $p$-DG algebra $A$. They are natural modules over $\Z[q^{\pm}]$, where $q[M] = [M\{1\}]$.

The tensor product action of $\Bbbk_\dif\dmod$ on $A_\dif\dmod$ induces an action of $K_0(\Bbbk)$ on $K_0(A)$, and of $G_0(\Bbbk)$ on $G_0(A)$ (c.~f.~Remark \ref{rmk-smash-product-algebra}). Define
$$\mathbb{O}_p:=\Z[q]/(1+q^2+\cdots + q^{2(p-1)}).$$
It is explained in \cite[Chapter 2]{KQ} that
\[\mathbb{O}_p \cong K_0(\Bbbk)\cong G_0(\Bbbk).\]
Therefore, $K_0(A)$ is naturally a $\mathbb{O}_p$-module.

%
\subsection{Basic Morita theory}
%

If $A$, $B$ are $p$-DG algebras, a \emph{$p$-DG $(A,B)$-bimodule} $_AX_B$ naturally gives rise to the following functors on $p$-DG module categories.
\begin{itemize}
\item[i).]The \emph{derived tensor product} $_AX\o^{\mathbf{L}}_B:
\mc{D}(B)\lra \mc{D}(A)$ is the composition of functors
\begin{equation}\label{def-derived-tensor}
\mc{D}(B)\stackrel{\mathbf{p}}{\lra} \mc{P}(B) \xrightarrow{_A X\o_B(-)}
\mc{C}(A)\stackrel{Q}{\lra}\mc{D}(A),
\end{equation}
where $\mc{P}(B)$ is the full subcategory of $\mc{D}(B)$ or $\mc{C}(B)$ consisting of cofibrant objects, and $Q$ is the natural
localization functor.
\item[ii).] The \emph{derived hom functor} $\mathbf{R}\HOM(_AX_B,-)$
is the composition of functors
\begin{equation}\label{def-derived-hom}
\mc{D}(A) \xrightarrow{\HOM_A(\mathbf{p_A}(X),-)} \mc{C}(B) \stackrel{Q}{\lra}
\mc{D}(B),
\end{equation}
where $\mathbf{p_A}(X)$ denotes the bar resolution of
$X$ as a left $p$-DG $A$-module. Here we have used that in the construction of the
bar resolution, $\mathbf{p_A}(X)$ has a
natural right $p$-DG $B$-module structure.
\end{itemize}

The tensor product and hom functors satisfy an adjunction property as in the usual DG case:
\begin{equation}\label{eqn-tensor-hom-adjunction}
\Hom_{\mc{D}(A)}(X\o^{\mathbf{L}}_BN,{M})\cong \Hom_{\mc{D}(B)}({N},\mathbf{R}\HOM_A(X,M)),
\end{equation}
for any $M\in\mc{D}(A)$ and $N\in \mc{D}(B)$.

A morphism $\mu:X\lra Y$ of $p$-DG $(A,B)$-bimodules descends to a natural
transformation between the derived tensor product functors
$$\mu^{\mathbf{L}}: X\o_B^{\mathbf{L}}(-)\Longrightarrow Y\o_B^{\mathbf{L}}(-):
 \mc{D}(B)\lra \mc{D}(A).$$
The following proposition gives a criterion for deciding when a derived tensor
functor induces an equivalence of derived categories, and when such
a natural transformation is an isomorphism of functors.

\begin{prop}\label{prop-criterion-derived-equivalence}
\begin{itemize}
\item[i)] Let $X$ be a $p$-DG $(A,B)$-bimodule and suppose it is cofibrant
as a $p$-DG $A$-module. Then
$X\o^{\mathbf{L}}_B(-):\mc{D}(B)\lra \mc{D}(A)$ is an
equivalence of triangulated categories if and only if the following
two conditions hold:
\begin{itemize}
\item[1)] The natural map $B\lra \HOM_{A}(X,X)$ is a quasi-isomorphism.
\item[2)] $X$, when regarded as a p-DG $A$-module, is a compact generator of
$\mc{D}(A)$.
\end{itemize}
\item[ii)] Let $\mu:X\lra Y$ be a morphism of p-DG
$(A,B)$-bimodules. The natural transformation
$\mu^{\mathbf{L}}: X\o_B^{\mathbf{L}}(-)\Longrightarrow
Y\o_B^{\mathbf{L}}(-)$ is an isomorphism of functors if and only if
$\mu:X\lra Y$ is a quasi-isomorphism of $p$-DG bimodules.
\end{itemize}
\end{prop}
\begin{proof}Omitted. See Proposition 8.8 of \cite{QYHopf}.
\end{proof}

An important special case comes from a map of $p$-DG algebras
$\phi:B\lra A$, and the bimodule $_AX_B={_AA_B}$. We
will allow maps $\phi: B\lra A$ with $\dif_B\circ \phi= \phi\circ
\dif_A$ which are non-unital, so that $\phi(1_B)$ is only an
idempotent in $A$.
\begin{itemize}
\item[i).] The (derived) induction functor $\phi^*=\mathrm{Ind}_B^A$ is the
derived tensor functor associated with the bimodule $_AA_B$:
\begin{equation}\label{def-induction}
\phi^*=\mathrm{Ind}_B^A=A\o_B^{\mathbf{L}}(-):\mc{D}(B)\lra \mc{D}(A).
\end{equation}
\item[ii).] The (derived) restriction functor $\phi_*=\mathrm{Res}^A_B$ is the
forgetful functor via the map $\phi$,
\begin{equation}\label{eqn-restriction}
\phi_*=\mathrm{Res}^A_B:\mc{D}(A)\lra\mc{D}(B)
\end{equation}
\end{itemize}

We have the following $p$-DG analogue for Theorem 10.12.5.1 of \cite{BL}, which follows readily from Proposition \ref{prop-criterion-derived-equivalence}.

\begin{cor}\label{cor-qis-algebra-equivalence-derived-categories}Let
$\phi:B\lra A$ be a morphism of $p$-DG algebras that is a
quasi-isomorphism. Then the induction and restriction functors
$$\phi^*:\mc{D}(B)\lra \mc{D}(A), \ \ \ \ \phi_*:\mc{D}(A)\lra \mc{D}(B)$$
are mutually inverse equivalences of categories. \hfill $\square$
\end{cor}

%
\subsection{Grothendieck groups of positive \texorpdfstring{$p$}{p}-DG algebras}
%

Our main goal of this section is a numerical understanding of the Grothendieck groups of positive $p$-DG algebras in terms of the classical Grothendieck groups of $A$. If $A$ is a graded ring, let $K_0^\prime(A)$ be the usual Grothendieck group of finitely generated graded projective $A$-modules, and $G_0^\prime(A)$ be the Grothendieck group of graded finite length $A$-modules. Both of these are $\Z[q^{\pm}]$ modules under the grading shift. We will see in Corollary \ref{cor-K-group-positive} that, for a positive $p$-DG algebra $A$, there are isomorphisms
$$K_0(A)\cong K_0^\prime(A) \o_{\Z[q^{\pm}]} \mathbb{O}_p, \ \ \ \ \ G_0(A)\cong G_0^\prime(A) \o_{\Z[q^{\pm}]} \mathbb{O}_p.$$

The main idea is that, for positive $p$-DG algebras, any compact module in $\mc{D}(A)$ is quasi-isomorphic to a finite cell module. The symbol of a finite cell module in $K_0(A)$
is clearly in the span of the symbols of the projective modules $P_i$. It is straightforward to show that $\{[P_i]\}_{i \in I}$ forms an $\mathbb{O}_p$ basis of $K_0(A)$, just as
it forms a $\Z[q^\pm]$ basis of $K_0^\prime(A)$.

\begin{lemma}\label{lemma-simple-property-FA}
Let $A$ be a positive $p$-DG algebra. The following statements hold.
\begin{itemize}
\item[(i)] If $0\lra M_1 \lra M \lra M_2 \lra 0$ is a short exact sequence of $p$-DG modules and $M_1, M_2 \in \mc{F}(A)$, then $M\in \mc{F}(A)$. In particular the category $\mc{F}(A)$ is closed under taking finite direct sums.
\item[(ii)] The category $\mc{F}(A)$ is preserved under tensor multiplication by finite dimensional $p$-complexes.
\end{itemize}
\end{lemma}
\begin{proof}The first part of the lemma is immediate from the definition. Using $(i)$ to prove $(ii)$ it suffices to see that $P_i\o V \in \mc{F}(A)$, where $V$ is an indecomposable $p$-complex. Any such $V$ has a filtration whose associated graded modules are one-dimensional $p$-complexes, and the result follows from $(i)$ again.
\end{proof}

\begin{cor}\label{cor-F-A-closed-under-all-but-summand}The image of $\mc{F}(A)$ in the homotopy category $\mc{C}(A)$ is closed under grading shifts, finite direct sums, cohomological shifts, and taking cones. $\hfill \square$
\end{cor}

\begin{proof} This follows from the previous lemma. In this paper we will not describe the cohomological shift functors; the interested reader should consult \cite[Section
2.2]{KQ} for additional information. \end{proof}

\begin{lemma}\label{lemma-FA-module-closed-under-quotient}Let $A$ be a positive $p$-DG algebra and $M\in \mc{F}(A)$ be a finite cell module. Suppose $M\cong M_1\oplus M_2$ is a decomposition of $M$ in $A_\dif\dmod$. Then $M_1, M_2$ are finite cell modules.
\end{lemma}
\begin{proof} For a  positive $p$-DG algebra $A$, the smash product ring $A_\dif$ is semi-local. Now, realize $M_1\cong M/M_2$, and equip with it the quotient filtration. It satisfies the subquotient requirement since each $P_i~(i\in I)$ is indecomposable in $A_\dif\dmod$.
\end{proof}

\begin{rmk}[Warning] The direct sum of the quotient filtrations on $M_1$ and $M_2$ is in general different from the original filtration on $M$. For
instance, this can happen when $M_1 \cong M_2$ so that the splitting $M \cong M_1 \oplus M_2$ is not canonical. One can then filter $M$ with
subquotients isomorphic to $M_1$ in a way which does not agree with the chosen splitting. \end{rmk}

\begin{prop}\label{prop-postive-idempotent-lifting}
Let $A$ be a positive $p$-DG algebra, and denote by $\overline{\mc{F}}(A)$ the smallest strictly full subcategory in $\mc{C}(A)$ containing all finite cell modules. Then $\overline{\mc{F}}(A)$ is idempotent complete.
\end{prop}

\begin{proof} Suppose $M$ is a finite cell module, and $\xi\in \End_{A_\dif}(M)$ descends to an idempotent in $\End_{\mc{C}(A)}(M)$. We will find a
genuine idempotent which is homotopic to $\xi$; its image in $A_\dif\dmod$ will then serve as the image of $\xi$ in $\mc{C}(A)$. Then the result follows from Lemma \ref{lemma-FA-module-closed-under-quotient}

Let $\eta = \xi^2 - \xi$, which by equation (\ref{eqn-mor-space-homotopy-category}) must be in the image of $\dif^{p-1}$ inside $\End_A(M)$. From
condition (ii.1) of Definition \ref{def-postive-pdg-algebra} it follows that $\End_{A_\dif}^0(M)$ is finite dimensional over $\Bbbk$. The image of $\dif^{p-1}$ inside it is an ideal. By Fitting's lemma\footnote{See, for instance, Benson \cite[Lemma 1.4.4]{Benson1} for the form of
the lemma that is used here.}, there exists $k \gg 0$ such that
\[ M\cong\mim(\eta^k)\oplus \mker(\eta^k), \]
inside $A_\dif\dmod$. The map $\xi$ respects this decomposition since it commutes with $\eta$.

On $\mker(\eta^k)$ the map $\eta$ is nilpotent. Therefore we may use Newton's method\footnote{See \cite[Theorem 1.7.3]{Benson1}.} to find an map
$\tilde{\xi} \in \End_{A_\dif}(M)$ whose restriction to $\mker(\eta^k)$ is an idempotent, and where $\tilde{\xi}-\xi$ is a polynomial in $\eta$
without constant term. Thus $\xi$ and $\tilde{\xi}$ are homotopic.

On the other hand, $\eta$ acts invertibly on $\mim(\eta^k)$. In particular, the identity of $\mim(\eta^k)$ is also in the image of $\dif^{p-1}$, and
thus $\mim(\eta^k)$ is contractible. Thus $\xi$ (resp. $\tilde{\xi}$) is homotopic to its composition with the projection to $\mker(\eta^k)$. In
particular, the projection of $\tilde{\xi}$ to $\mker(\eta^k)$ is a genuine idempotent homotopic to $\xi$.
\end{proof}

Since the localization functor $\mc{C}(A) \to \mc{D}(A)$ does not affect morphism spaces between finite cell modules, we may also use $\overline{\mc{F}}(A)$ to denote the smallest strictly full subcategory of $\mc{D}(A)$ containing all finite cell modules.

\begin{thm}\label{thm-FA-equivalent-compact-derived-category}Let $A$ be a positive $p$-DG algebra. Then $\overline{\mc{F}}(A)\subset \mc{D}^c(A)$ and the inclusion is an equivalence of triangulated categories.
\end{thm}
\begin{proof}Corollary \ref{cor-F-A-closed-under-all-but-summand} shows that $\overline{\mc{F}}(A)$ is triangulated, while the above proposition shows that it is idempotent complete. The claim follows readily from Corollary \ref{cor-compact-derived-category-smallest}.
\end{proof}

Now we can give an upper bound of the Grothendieck group $K_0(A)$ for a positive $p$-DG algebra. By Theorem \ref{thm-FA-equivalent-compact-derived-category}, any compact module in $\mc{D}(A)$ is isomorphic to a finite cell module $M\in\mc{F}(A)$. Using the diagram (\ref{eqn-convolution-diagram}), the symbol of $M$ in the Grothendieck group can be written as an alternating sum of the symbols of the subquotients of the filtration on $M$, which by Definition \ref{def-finite-cell} are $[P_i\{l_i\}]=q^{l_i}[P_i]\in K_0(A)$ for $i\in I$ and $l_i\in \Z$. Since the usual Grothendieck group $K_0^\prime(A)$ of graded projective $A$-modules is freely generated over $\Z[q^{\pm}]$ by $\{[P_i]|i\in I\}$, we have a surjective map of $\mathbb{O}_p$-modules:
\[
K_0^\prime(A)\o_{\Z[q,q^{-1}]} \mathbb{O}_p\lra K_0(A).
\]
Our next goal will be to show that this map is also injective.

When $A$ is a $p$-DG algebra, the $\RHOM$-pairing between derived categories
$$\RHOM_A(-,-):\mc{D}^c(A)\times \mc{D}^f(A)\lra \mc{D}(\Bbbk)$$
descends to a map on Grothendieck groups
$$[\RHOM_A(-,-)]:K_0(A)\times G_0(A)\lra \mathbb{O}_p.$$
The derived functor $\RHOM$ was described in \eqref{def-derived-hom}. Notice that if $V$ is a
finite dimensional $p$-complex, $M\in \mc{D}^c(A)$ and $N\in \mc{D}^f(A)$
there is a canonical isomorphism of $p$-complexes
$$\RHOM_{A}(M \o V,N)\cong \HOM_A(\mathbf{p}(M) \o V, N) \cong \HOM_A(\mathbf{p}(M),N \o V^*) \cong \RHOM_A(M,N)\o V^*.$$
On the Grothendieck group level, this says that the pairing above is
$\mathbb{O}_p$-\emph{sesquilinear}. Let an overbar denote the automorphism of $\mathbb{O}_p$ sending $q \mapsto q^{-1}$. Sesquilinearity of a bilinear form $\{-,-\}$ states that
\[ a \lbrace [M],[N] \rbrace = \lbrace [M],a[N] \rbrace = \lbrace \overline{a}[M],[N] \rbrace \]
for any $a \in \mathbb{O}_p$.

Now we focus on positive $p$-DG algebras. Since the smash product algebra $A_\dif$ is a semi-local graded ring whose degree zero part is isomorphic to $A^0$, the Jacobson radical $J(A_\dif)$ of this algebra consists of everything in positive degrees, and $A_\dif/J(A_\dif)\cong A^0$. It follows that any finite dimensional module over $A_\dif$ admits a finite filtration whose subquotients are graded shifts of $S_i,~(i\in I)$, so that $G_0(A)$ is $\mathbb{O}_p$-generated by the symbols of the simples $[S_i]$. Applying the $\RHOM$-pairing between the systems of modules $\{P_i|i\in I\}$, $\{S_j|j\in I\}$, we have
\[
\RHOM(P_i,S_j)\cong
\left\{
\begin{array}{ll}
D_i & \textrm{if $i=j$,}\\
0 & \textrm{otherwise,}
\end{array}
\right.
\]
where $D_i \cong \End_{A}(S_i)$ is a finite dimensional division algebra over $\Bbbk$ concentrated in degree zero. Necessarily $\dif$ acts trivially on $D_i$. Set $d_i=\mathrm{dim}_\Bbbk D_i$. Now if $\sum_{i\in I}r_i [P_i]=0$ is a linear relation, with $r_i\in \mathbb{O}_p$, we apply $[\RHOM_A(-,S_j)]$ to get
\[
0=\sum_{i\in I} \overline{r_i}[\RHOM_A(P_i,S_j)]=\overline{r_j}[D_j]=\overline{r_j} d_j[\Bbbk] \in K_0(\Bbbk).
\]
Since $K_0(\Bbbk)\cong \mathbb{O}_p$ has no $\Z$-torsion, it follows that $r_j=0$ for each $j$. Thus there could not have been any $\mathbb{O}_p$-linear relation between the symbols $[P_i]\in K_0(A)$ from the start. Likewise one shows that there can be no linear relation among the symbols $[S_i]$ in $G_0(A)$. This discussion gives us the following.

\begin{cor}\label{cor-K-group-positive}Let $A$ be a positive $p$-DG algebra. Then there are isomorphisms of Grothendieck groups
\[
K_0(A)\cong K_0^\prime(A)\o_{\Z[q,q^{-1}]}\mathbb{O}_p, \ \ \ \ G_0(A)\cong G_0^\prime(A)\o_{\Z[q,q^{-1}]}\mathbb{O}_p,
\]
where $K_0^\prime$ (resp. $G_0^\prime$) stands for the usual Grothendieck group of graded projective (resp. graded finite dimensional) $A$-modules.
\hfill$\square$
\end{cor}

%
\subsection{A K\"{u}nneth formula}
%

We specialize to the case when $A$ is strongly positive, as in Definition \ref{def-postive-pdg-algebra}. Recall that in this case $A^0\cong \prod_{i\in I}\mathrm{M}(n_i,\Bbbk)$ is
a product of matrix algebras with coefficients in the ground field. If $A_1$, $A_2$ are two such $p$-DG algebras, then so is their tensor product $A_1\o A_2$. This follows because
for any $n,m\in \N$, $\mathrm{M}(n,\Bbbk)\o \mathrm{M}(m,\Bbbk) \cong \mathrm{M}(nm,\Bbbk)$. By applying Corollary \ref{cor-K-group-positive} to $A_1$, $A_2$ and $A_1\o A_2$, we
obtain the following K\"{u}nneth-type property for their Grothendieck groups.

\begin{cor}\label{cor-Kunneth-strongly-positive}
Let $A_1$, $A_2$ be two strongly positive $p$-DG algebras relative to the ground field $\Bbbk$. Then their tensor product is also strongly positive relative to $\Bbbk$, and there are isomorphisms of Grothendieck groups
\[
K_0(A_1\o A_2)\cong K_0(A_1)\o_{\mathbb{O}_p}K_0(A_2),  \ \ \ \ G_0(A_1\o A_2)\cong G_0(A_1)\o_{\mathbb{O}_p}G_0(A_2),
\]
which are identifications of $\mathbb{O}_p$-modules. \hfill$\square$
\end{cor}

\begin{eg}[A non-example] The above result fails when we remove the ``strongly positive'' hypothesis, because it fails for the ordinary Grothendieck groups $K_0^\prime$ and $G_0^\prime$. Consider $\F_{p^r}$ as a $p$-DG algebra with the zero differential over $\F_p$. It is easy to see that $K_0(\F_{p^r})\cong G_0(\F_{p^r})\cong \mathbb{O}_p$, which is spanned by the symbol $[\F_{p^r}]$. However $\F_{p^r}\o_{\F_p} \F_{p^r}\cong \F_{p^r}^{\oplus r}$, so that
\[K_0(\F_{p^r}\o_{\F_p} \F_{p^r})\cong G_0(\F_{p^r}\o_{\F_p} \F_{p^r})\cong \mathbb{O}_p^{\oplus r},\]
and the K\"{u}nneth property fails.
\end{eg}

\begin{rmk}
The K\"{u}nneth property for ordinary DG algebras that are strongly positive is a direct consequence of the results in Keller-Nicolas \cite{KeNi} and Schn\"{u}rer \cite{SchPos}. The discussion in this chapter is partly motivated by their work.

The method adopted here generalizes immediately, in the context of hopfological algebra, to any strongly positive $H$-module algebra, where $H$ is a graded finite dimensional Hopf (super) algebra. (See Remark \ref{rmk-smash-product-algebra}). Unfortunately, the proof here is essentially ``numerical," only giving an isomorphism of Grothendieck groups rather than a comparison on the level of spectra.

We would like to pose the following general question to the reader: Under what restrictions on a $p$-DG algebra, or more generally, an $H$-module algebra, does the K\"{u}nneth formula hold?
\end{rmk}

The $p$-DG algebras we are most concerned with in this paper are not strongly positive, but they are $p$-DG Morita equivalent to strongly positive $p$-DG algebras (in the sense of Proposition \ref{prop-criterion-derived-equivalence}).

\begin{cor}\label{cor-Moritaly-positive-Kunneth}
Let $A_1$, $A_2$ be two $p$-DG algebras which are Morita equivalent to strongly positive $p$-DG algebras. Then their tensor product is also Morita equivalent to a strongly positive $p$-DG algebra, and the K\"{u}nneth formula holds. \hfill$\square$
\end{cor}
	
This version of the K\"{u}nneth property for the usual Grothendieck groups of certain (DG) algebras has played a significant role in many known examples of categorification, see for instance \cite{KL1,KL2,EliasKh}.

%
\subsection{\texorpdfstring{$p$}{p}-DG categories}\label{subsec-p-DG-cats}
%

There are no references for this section, although it is a straightforward extension of well-known material.

We give a slight generalization of the notion of $p$-DG algebras. This is the analogue in the $p$-DG setting of the corresponding notion in the usual DG theory. See \cite{Ke1, KeICM}.

\begin{defn}\label{def-p-dg-cat}
 A graded $\Bbbk$-linear category $\mc{A}$ is called a $p$-DG category if the morphism spaces between any objects
$X,Y\in \mc{A}$ are equipped with a degree $2$, $p$-nilpotent operator $\dif$
\[\dif:\Hom^{\bullet}_{\mc{A}}(X,Y)\lra \Hom^{\bullet+2}_{\mc{A}}(X,Y),\]
which acts via the Leibnitz rule on the composition of morphisms
\[
\begin{array}{rccl}
\dif: &\HOM_{\mc{A}}(Y,Z)\times \HOM_{\mc{A}}(X,Y) & \lra & \HOM_{\mc{A}}(X,Z), \\
 & (g, f)  & \mapsto & \dif(g\circ f)=\dif(g)\circ f+ g\circ \dif(f).
\end{array}
\]
\end{defn}

\begin{eg}
Given a $p$-DG algebra $A$, one can construct a $p$-DG category by considering an enriched version of $A_\dif\dmod$. The objects are $p$-DG modules over $A$, but morphisms are given by the $p$-complex $\HOM_A(M,N)$ instead of $\Hom^0_{A_\dif}(M,N)$. Taking $A=\Bbbk$ with the zero differential, one recovers the $p$-DG category $\Bbbk_\dif\dmod $ of $p$-complexes.
\end{eg}

\begin{defn}\label{def-mod-over-p-DG-cat}
A left (resp. right) $p$-DG module $\mc{M}$ over a $p$-DG category $\mc{A}$ is a covariant functor
\[\mc{M}:\mc{A}\lra \Bbbk_\dif\dmod \ \ \ \ (\textrm{resp.}~\mc{M}:\mc{A}^{op}\lra \Bbbk_\dif\dmod ),\]
that commutes with the $\dif$-actions on $\mc{A}$ and $\Bbbk_\dif\dmod$.
\end{defn}

For example, given any object $X\in \mc{A}$, the representable functor
$$\HOM_{\mc{A}}(X,-):\mc{A}\lra \Bbbk_\dif\dmod, \ \ \ \ (\textrm{resp.}~\HOM_{\mc{A}}(-,X):\mc{A}^{op}\lra \Bbbk_\dif\dmod)$$
is a left (resp. right) $p$-DG module. Such a $p$-DG module is called \emph{representable}. It is easy to check that the category of left (resp. right) $p$-DG modules over a $p$-DG category is abelian.

\begin{rmk}\label{rmk-shrinking-category-to-algebra}
We will informally treat $p$-DG categories as $p$-DG algebras with many idempotents $1_X$, one for each object $X\in \mc{A}$. Conversely any $p$-DG algebra is just a $p$-DG category with a unique object. Given a $p$-DG category and a finite set of objects $\mathbb{X}=\{X_i|i=1, \dots, n\}$, it will be helpful to consider the $p$-DG endomorphism algebra
\[
\END_{\mc{A}}(\mathbb{X}):=\oplus_{i,j=1}^n\HOM_{\mc{A}}(X_i,X_j),
\]
and modules obtained by induction from this smaller algebra to the whole category $\mc{A}$.
\end{rmk}

The following definition generalizes the corresponding notions over $p$-DG algebras (Definitions \ref{def-nice-p-DG-mod} and \ref{def-finite-cell}).

\begin{defn}\label{def-nice-p-DG-mod-over-category}Let $\mc{A}$ be a $p$-DG category, and $\mc{M}$ be a $p$-DG module over $\mc{A}$.
\begin{itemize}
\item[(i)] $\mc{M}$ is said to satisfy \emph{property (P)} if there exits an exhaustive, possibly infinite, increasing filtration $F^\bullet$, on $\mc{M}$ such that each subquotient $F^\bullet/F^{\bullet -1}$ is isomorphic to a direct sum of representable $p$-DG modules.
\item[(ii)] $\mc{M}$ is a \emph{cofibrant module} if it is a direct summand of a property (P) module;
\item[(iii)] $\mc{M}$ is called a \emph{finite cell module} if it is a property (P) module with a finite filtration each of whose subquotients is isomorphic to a representable $p$-DG module.
\end{itemize}
\end{defn}

The hopfological constructions in the previous sections can be followed verbatim for $p$-DG categories. There is a notion of (simplicial) cofibrant replacements for $p$-DG modules, and of the homotopy category $\mc{C}(\mc{A})$ and derived category $\mc{D}(\mc{A})$. In particular, one can characterize the compact derived category in analogy to Theorem \ref{thm-compact-mod} and Corollary \ref{cor-compact-derived-category-smallest}, using the powerful machinery of Ravenel-Neeman \cite{Ra, Nee} on compactly generated triangulated categories.

\begin{cor}\label{cor-compact-mod-p-dg-cat}A $p$-DG module $\mc{M}$ over $\mc{A}$ is compact if and only if it is a direct summand in the derived category $\mc{D}(\mc{A})$ of a finite cell module. The compact derived category $\mc{D}^c(\mc{A})$ is the smallest strictly full Karoubian triangulated subcategory in $\mc{D}(\mc{A})$ that contains all the grading shifts of all representable modules. \hfill$\square$
\end{cor}

Hence we also have the notion of Grothendieck groups for a $p$-DG category.

Finally, notice that the \emph{$2$-category} of $p$-DG categories is symmetric monoidal, and we will use $\o$ to stand for the tensor product of $p$-DG categories.

\vspace{0.06in}

Now let us discuss the notion of a $p$-DG $2$-category. In this paper, all $2$-categories will be strict and small. When $\mc{U}$ is a $2$-category and $\l,\mu$ are objects, we will typically write ${}_\mu \UC_\l$ for the $1$-category $\Hom_{\mc{U}}(\l,\mu)$, following Lauda \cite[Section 5.1]{Lau1}. We write $\1_\l$ for the identity $1$-morphism of the object $\l$, and $1_M$ for the identity $2$-morphism of the $1$-morphism $M$.We omit any multiplication symbol when composing $1$-morphisms, or using $1$-morphism composition to compose $2$-morphisms \emph{horizontally}. We reserve $\circ$ for the usual, \emph{vertical} composition of $2$-morphisms.

\begin{defn}\label{def-p-DG-2-cat}
A $p$-DG $2$-category $(\mc{U},\dif)$ consists of a usual $2$-category $\UC$, together with a $p$-nilpotent \emph{(2-categorical) derivation}, or a \emph{(2-categorical) differential} on $2$-morphisms, which satisfies the Leibniz rule for both horizontal and vertical multiplication of $2$-morphisms.
\end{defn}

More explicitly, a $p$-DG $2$-category consists of the following data.
 \begin{itemize}
 \item[(i)] A set of objects $I=\{\l ,\mu, \nu \ldots \}$, and for any $\l, \mu\in I$,  ${{}_\mu \UC_\l}=\Hom_{\mc{U}}(\l,\mu)$ forms a $p$-DG category.
 \item[(ii)] For any $1$-morphism ${{}_\mu E_\l}, {{}_\mu E^{\prime}_\l} $ of ${{}_\mu \UC_\l}$, the morphism space
    \[\HOM_{{}_\mu \UC_\l}({{}_\mu E_\l}, {{}_\mu E^{\prime}_\l})\]
    is a $p$-complex.
 \item[(iii)] For any objects $\l,\mu, \nu \in I$, and $1$-morphisms ${{}_\l E_\mu}, {{}_\l E^\prime_\mu} , {{}_\l E^{\prime\prime}_\mu} \in {{}_\l \UC _\mu}$, the (vertical) composition
     \begin{eqnarray*}\HOM_{{}_\mu \UC_\l}({{}_\mu E^\prime _\l}, {{}_\mu E^{\prime \prime}_\l})\times \HOM_{{}_\mu \UC_\l}({{}_\mu E_\l}, {{}_\mu E^{\prime} _\l}) & \lra &\HOM_{{}_\mu \UC _\l}({{}_\mu E_\l}, {{}_\mu E^{\prime \prime} _\l}),\\
     (f\ , \ g)  &\mapsto & f\circ g,
     \end{eqnarray*}
 satisfies the Leibniz rule as in Definition \ref{def-p-dg-cat}.
 \item[(iv)] For any objects $\l,\mu, \nu \in I$, and $1$-morphisms ${_\l E _\mu}, {_\l E^\prime _\mu} \in {_\l \UC _\mu}$, ${_\mu F _\nu}, {_\mu F^\prime _\nu} \in {_\mu \UC _\nu}$, the (horizontal) composition
     \begin{eqnarray*}
     \HOM_{_\l \UC_\mu}({_\l E _\mu}, {_\l E^{\prime}_\mu})\times \HOM_{_\l \UC_\l}({_\mu F_\nu}, {_\mu F^{\prime} _\nu}) & \lra &\HOM_{_\l \UC _\nu}({_\l EF_\nu}, {_\l E^{\prime }F^{\prime} _\nu}),\\
     (f\ , \ h)  &\mapsto & fh,
     \end{eqnarray*}
 satisfies the Leibniz rule
     \[ \dif(fh)=\dif(f)h+f\dif(h).\]
 \end{itemize}

\begin{eg}\label{eg-p-DG-ALG} A key example to keep in mind is the $p$-DG $2$-category of $p$-DG algebras, denoted by $^p\mc{DGA}$. Its objects are $p$-DG algebras $\{A_\l|\l \in
I\}$ (with the usual restrictions coming from categorical hygiene). The $1$-morphisms between two $p$-DG algebras are given by $p$-DG bimodules, and the $2$-morphisms are
homomorphisms of $p$-DG bimodules. Fixing two $p$-DG algebras $A_\l$ and $A_\mu$, the morphism-category between them is identified with the category of left $p$-DG modules over
$A_\mu\o A_\l^{op}$.
\end{eg}

\begin{defn}Let $\mc{U}$ be a $p$-DG $2$-category. We define the homotopy and derived categories of $\mc{U}$ to be
\[
\mc{C}(\UC):=\bigoplus_{\l,\mu \in I}\mc{C}({_\mu\UC_\l}), \quad \quad \mc{D}(\mc{U}):=\bigoplus_{\l,\mu \in I}\mc{D}({_\mu \UC_{\l}}),
\]
and the Grothendieck groups to be
\[
K_0(\mc{U}):=\bigoplus_{\l,\mu \in I}K_0({_\mu\UC_\l}), \quad \quad G_0(\mc{U}):=\bigoplus_{\l,\mu \in I}G_0({_\mu\UC_\l}).
\]
\end{defn}

In nice cases, these Grothendieck groups inherit the natural structure of idempotented rings. This phenomenon will be exemplified in Section \ref{grothring}.


\section{\texorpdfstring{$p$}{p}-DG symmetric functions} \label{section-symmetric-functions}

%
\subsection{Symmetric polynomials}
%

Let $\Bbbk[x]$ be the polynomial ring in one variable, graded with $\deg(x)=2$. If we set $\dif(x) = x^2$, the Leibniz rule implies that $\dif(x^k)
= kx^{k+1}$; this is a valid differential, giving $\Bbbk[x]$ the structure of a $p$-DG algebra. The ideal generated by $x$ is contractible, so that
$\Bbbk[x]$ is quasi-isomorphic to $\Bbbk$. Now let $\Bbbk[x_1,x_2,\ldots,x_n]$ be the $n$-fold tensor product, so that $\dif(x_i) = x_i^2$. Again,
this is a $p$-DG algebra which is quasi-isomorphic to $\Bbbk$. Inside this ring we have $\sym_n$, the symmetric polynomials, which is preserved by
$\dif$ since $\dif$ commutes with the action of $S_n$. It is well-known that $\sym_n$ is isomorphic to a graded polynomial algebra either on the
elementary symmetric polynomials $\{e_1, \dots , e_n\}$, or on the complete symmetric polynomials $\{h_1, \dots ,h_n\}$, where
$\mathrm{deg}(e_k)=\mathrm{deg}(h_k)=2k$.

\begin{lemma}\label{lemma-dif-on-sym-poly} The differential $\dif$ acts on the polynomial generators of $\sym_n$ as follows.
\[
\dif(e_k) = e_1 e_k - (k+1) e_{k+1} \  \textrm{ for $1 \le k \le n-1$}, \ \ \ \dif(e_n) = e_1 e_n,
\]
while for all $k$,
\[
\dif(h_k) = (k+1) h_{k+1} - h_1 h_k.
\]
\end{lemma}
\begin{proof}
Exercise. Alternatively, we prove an analogous statement for symmetric functions in Lemma \ref{lemma-dif-on-sym-function} below. Setting $x_i=0$ for $i>n$ in that proof will yield this result.
\end{proof}

\begin{lemma}\label{lemma-small-sym-qis-to-k}
When $n < p$, $\sym_n$ is quasi-isomorphic to $\Bbbk$.
\end{lemma}

\begin{proof}
There is clearly an injection sending $\Bbbk \to \Bbbk \cdot 1$, with cokernel the augmentation ideal $\sym_n^\prime$. When $n < p$, the representations of $S_n$ over $\Bbbk$ are actually semisimple. Therefore $\sym_n$ is not just a submodule of $\Bbbk[x_1,x_2,\ldots,x_n]$ but actually a summand, under an idempotent which commutes with $\dif$. In particular, the positive degree symmetric polynomials $\sym_n^\prime$ are a summand of the contractible $p$-complex $\Bbbk[x_1,\ldots,x_n]^\prime$, and are therefore contractible.
\end{proof}

This lemma does rely on the fact that $n<p$. When $n \ge p$, $\sym_n^\prime$ is still a submodule of a contractible $p$-complex, but that does not mean it is contractible itself.

\begin{eg}
Suppose that $n=p$. Then $\sym_p$ is not quasi-isomorphic to $\Bbbk$, but is actually quasi-isomorphic to $\Bbbk[e_p^p]$ with the trivial
differential. We will check this soon. Note for now that $\Bbbk[e_1,e_2,\ldots,e_{p-1}]$ is a $p$-DG submodule of $\sym_p$ since $\dif(e_{p-1}) = e_1 e_{p-1}- pe_p = e_1 e_{p-1}$.
\end{eg}

Now we want to understand various $p$-DG module structures on the rank-one free module over $\sym_n$. For an integer $a \in \{0, 1, \dots , p-1\}$ or its residue in $\F_p$, we use the notation $\sym_n(a):=(\sym_n\cdot 1_a, \dif_a)$ to represent the $p$-DG module which, as a $\sym_n$-module, is freely generated by $1_a$, and where $\dif_a(1_a)=ae_11_a$. This is the general form of such a $p$-DG module. After all, every degree two element of $\sym_n$ is equal to $a e_1$ for some $a \in \Bbbk$, and $\dif_a^p=0$ if and only if $a \in \F_p$ (a simple computation). Note that, up to a grading shift, $\sym_n(a)$ can be realized within $\sym_n(0)$ as the ideal generated by $e_n^a$.

\begin{lemma}\label{lemma-small-rank-one-ideal-contractible}
Suppose that $1 \le a \le n < p$. Then the ideal generated by $e_n^a$ inside $\sym_n$ is acyclic.
\end{lemma}

\begin{cor}\label{cor-small-rank-one-mod-contractible}
For $1 \le a \le n < p$, the $p$-DG module $\sym_n(a)$ is acyclic. \hfill$\square$
\end{cor}

\begin{proof} We show this by induction on $n$, and within each fixed $n$, by induction on $a$. For $n=a=1$, the result is already known.

Suppose that $a=1$. There is a short exact sequence of $p$-DG modules
\[
 0 \to (e_n) \to \sym_n \to \sym_{n-1} \to 0.
\]

The first map is the inclusion of the ideal, while the second is a quotient. The quotient is a quasi-isomorphism between two $p$-DG modules which are
quasi-isomorphic to $\Bbbk$. Therefore the ideal is contractible.

Now suppose that $1 < a \le n$. There is a short exact sequence
\[
0 \to (e_n^a) \to (e_n^{a-1}) \to \sym_{n-1}(a-1) \to 0.
\]
The first map is an inclusion of ideals. Here, $\sym_{n-1}(a-1)$ is realized as the classes modulo $(e_n^a)$ of the polynomials $f e_n^{a-1}$ for $f \in \Bbbk[e_1,e_2,\ldots,e_{n-1}] \subset \sym_n$. Now the induction hypothesis implies that the center and right terms are acyclic, and therefore so is the left term.
\end{proof}

The important corollary of this is the special case when $n=p-1$.

\begin{cor}\label{cor-sym-p-1-all-ideal-contractible}
For any $a \in \F_p\backslash \{0\}$, $\sym_{p-1}(a)$ is a contractible $p$-complex. Moreover, $\sym_{p-1}(0)$ is quasi-isomorphic to $\Bbbk$. \hfill$\square$
\end{cor}

An obvious implication of this corollary is that if $N$ is a direct summand of $\sym_{p-1}(a)$ as a $p$-complex, then $N$ is contractible unless $a=0$ and
$N$ contains $1_a$, in which case $N$ is quasi-isomorphic to $\Bbbk$.

%
\subsection{Symmetric functions}
%

Symmetric polynomials form a natural graded inverse system, which allows us to take the inverse limit as $n$ goes to infinity to obtain symmetric functions. This is the ring
\begin{equation}\label{eqn-sym-function}
\Lambda =\underleftarrow{\lim} \sym_n \cong \Bbbk[e_1,e_2,\ldots].
\end{equation}
It is equipped with a natural $p$-differential, arising from its formal inclusion inside $\Bbbk[x_1,x_2, \dots]$, and compatible with any truncation $\Lambda \twoheadrightarrow \sym_n$. The infinite-variable polynomial algebra is not finite-dimensional in each degree so it is not a positive $p$-DG algebra in the sense of Definition \ref{def-postive-pdg-algebra}, but the subalgebra $(\Lambda,\dif)$ is a (strongly) positive $p$-DG algebra.

\begin{lemma} \label{lemma-dif-on-sym-function} The differential acts on the elementary and complete symmetric functions as follows.
\[
\dif(e_k) = e_1 e_k - (k+1) e_{k+1}.\qquad  \dif(h_k) = (k+1) h_{k+1} - h_1 h_k.
\]
\end{lemma}

\begin{proof} This can be shown most easily with generating functions. We prove the second formula. The first follows by
a similar computation.

Consider the following generating function over $\Z[t]$ for the
elementary symmetric functions (we set $h_0:=1$),
$$\prod_{i=1}^{\infty}(1-x_it)^{-1}=\sum_{m=0}^{\infty}h_mt^m, $$
and let $\dif$ act on it as a $\Z[t]$-linear derivation determined by
$\dif(x_i)=x_i^2$. Applying $\dif$ to both sides
gives us
$$\begin{array}{rcl}
\sum_{m=0}^{\infty}\dif(h_m)t^m & = & \sum_{i=1}^{\infty}
\left(\dfrac{x_i^2t}{(1-x_it)^2}\cdot\prod_{j\neq i}(1-x_jt)^{-1}\right) \vspace{0.05in}\\
& = & \sum_{i=1}^{\infty}\left(\dfrac{-x_i+x_i^2t}{(1-x_it)^2}\cdot\prod_{j\neq i}(1-x_jt)^{-1}\right)\vspace{0.05in}\\
& & +\sum_{i=1}^{\infty}\left(\dfrac{x_i}{(1-x_it)^2}\cdot\prod_{j\neq i}(1-x_jt)^{-1}\right)\\
& = & \sum_{i=1}^{\infty}(-x_i) \prod_{j=1}^{\infty}(1-x_jt)^{-1} +
\frac{\dif}{\dif t}\prod_{j=1}^m(1-x_jt)^{-1}\vspace{0.05in}\\
& = &(-h_1)\cdot \prod_{j=1}^{\infty}(1-x_jt)^{-1})+\frac{\dif}{\dif
t}\prod_{j=1}^{\infty}(1-x_jt)^{-1})\\
& = &(-h_1) \cdot \sum_{m=0}^{\infty}h_mt^m
+\sum_{m=-1}^{n-1}(m+1)h_{m+1}t^{m}.
\end{array}
$$
Comparing coefficients of $t$ on both sides gives the claimed
formula. \end{proof}

Now we state the main result of this section. Again, let us denote by $\Lambda(a):=\Lambda\cdot 1_a$ ($a \in \F_p$) the rank-one $p$-DG module over $\Lambda$ with the differential defined by $\dif_a(f1_a)=\dif(f)1_a+ afe_11_a$. When no confusion can be caused, we will drop $1_a$ from the expression.

\begin{prop}\label{prop-cohomology-of-sym-function}
\begin{itemize}
\item[(i)] The inclusion of the $p$-DG subalgebra $\Bbbk[e_p^p,e_{2p}^p,e_{3p}^p,\ldots]$ equipped with the zero differential into $\Lambda$ is a quasi-isomorphism of $p$-DG algebras.
\item[(ii)] When $a \in \F_p\backslash\{0\}$, the $p$-DG $\Lambda$-module $\Lambda(a)$ is acyclic.
\end{itemize}
\end{prop}

\begin{proof} Note that there is an obvious inclusion of $p$-DG algebras from $\Bbbk[e_p^p,e_{2p}^p,e_{3p}^p,\ldots]$ into $\Lambda$, so it suffices
for part (i) to show that this map is a quasi-isomorphism. What remains, for both parts, is merely a question about the underlying $p$-complexes.

The subalgebra $\sym_{p-1} \subset \Lambda$ is $\dif$-stable by Lemma \ref{lemma-dif-on-sym-function}, so we can regard $\Lambda(a)$ as a $p$-DG
module over $\sym_{p-1}$ via restriction. We will construct a filtration on $\Lambda(a)$ whose subquotients are isomorphic to $\sym_{p-1}(a)$ for
various values of $a \in \F_p$. Then the result will follow from Corollary \ref{cor-sym-p-1-all-ideal-contractible}.

\vspace{0.1in}

\emph{Step I.} We begin by organizing the elementary symmetric functions into \emph{blocks}.

Consider an arbitrary monomial $f \in \Lambda(a)$ in terms of $e_i$, and write it as the finite product $$f = g f_1 f_2 \ldots,$$ where $g$ is the
product of all $e_i$ for $1 \le i \le p-1$, $f_1$ is the product of all $e_i$ for $p \le i \le 2p-1$, and in general $f_k$ is the product of all
$e_i$ for $kp \le i < (k+1)p$. We think of this as grouping the monomial $f$ into \emph{blocks} $f_i\in \Bbbk[e_{ip},\dots , e_{(i+1)p-1}]$. The
``zero-th block" $g$ is not thought of as a block, but as a coefficient in $\sym_{p-1}$. For $i \ge 1$ let $d_i$ be the sum of all the exponents in
the block $f_i$. For instance, if $$f = e_1^5 e_2 e_p^3 e_{p+1} e_{2p+1}^9=(e_1^5 e_2)\cdot (e_p^3 e_{p+1})\cdot (e_{2p+1}^9),$$ then $d_1=4$ and
$d_2=9$. We use the term ``exponent" instead of ``degree:" $\Lambda$ is already graded with $\deg(e_i)=2i$, but we are interested in the exponent,
to which each $e_i$ contributes equally. Let $D = \sum_{i \ge 1} d_k$.

Block notation helps us to compute the action of the differential. Let $f_k$ be a monomial in the $k$-th block, with exponent $d_k$. It is not hard to see that $\dif(f_k) = d_k e_1 f_k + f'_k$ for some $f'_k$ in the same block with the same exponent. When $d_k=1$, this follows from the formula $\dif(e_m) = e_1 e_m - (m+1) e_{m+1}$, and the general case is straightforward. For this proof, we will write $\dif '$ for the linear map defined by $\dif '(e_m) = -(m+1)e_{m+1}$ and satisfying the Leibniz rule. A formula for $\dif'$ is given in \eqref{difonNd} below. Clearly $f'_k = \dif '(f_k)$.

Now let $f = g f_1 f_2 \ldots$ be a monomial as above. Then
\begin{eqnarray*}
\dif_a(f) & =  & (ae_1+\dif(g))\prod_{i}f_i+g\sum_i\dif(f_i)\prod_{j\neq i}f_j \\
	& =  & ((D+a)e_1 + \dif(g))\prod_i f_i + g \sum_i \dif '(f_i) \prod_{j \neq i}f_j.
\end{eqnarray*}

Fix $\od = (d_1,d_2,\ldots)$ a sequence of non-negative integers with $d_k=0$ for all $k >> 0$, and set $D = \sum_k d_k$. Define $M_{\od}$ to be the set of all monomials $f = f_1 f_2 \ldots$ where $f_k$ has exponent $d_k$ (and $f_0 = 1$). Let $X_{\od}$ be the $\sym_{p-1}$-span of $M_{\od}$ , consisting of all polynomials of the form $gf$ for $f \in M_{\od}$ and $g \in \sym_{p-1}$. It is now clear from the above observation that $X_{\od}$ is preserved by $\dif$, and that $\Lambda$ splits as a $p$-DG module over $\sym_{p-1}$ into a direct sum of all $X_{\od}$.

\vspace{0.1in}

\emph{Step II.} We study the structure of each block $X_{\od}$ as a \emph{filtered} $p$-DG $\sym_{p-1}$-module. In the notation of the last paragraph, let $f=\prod_k f_k\in M_{\od}$. Then each $f_k$ can be encoded as a composition of $d_k$ with $p$ rows, based on the exponents of each $e_i$ for $kp \le i < (k+1)p$. Place a partial order on $M_{\od}$ similar to the usual dominance order on compositions, so that the operation which sends $e_i^b e_{i+1}^c \mapsto e_i^{b-1} e_{i+1}^{c+1}$ within a given block will increase the order. If $f \in M_{\od}$ then $\dif '(f)$ is a sum of monomials in $M_{\od}$ which are strictly greater in the partial order. From the previous step, if $g\in \sym_{p-1}$ and $f \in M_{\od}$ we have
\begin{equation}
\label{difonXd}
\dif_a(g f) = g \dif '(f) + (\dif(g) + (a+D) e_1 g) f.
\end{equation}
In particular, the $\sym_{p-1}$ spans of ideals in $M_{\od}$ are $p$-DG submodules.

Each $M_{\od}$ is finite, so that the partial order has a filtration by ideals whose subquotients are singletons $\{f\}$. Clearly this descends to a filtration on $X_{\od}$, whose subquotients are spanned by $g f$ for fixed $f$. This subquotient is isomorphic as a $p$-DG module over $\sym_{p-1}$ to $\sym_{p-1}(a+D)$. Therefore, by Corollary \ref{cor-sym-p-1-all-ideal-contractible}, it is acyclic if $a+D \not\equiv 0$ (mod $p$), and does not contribute to the cohomology. On the other hand, whenever $a+D \equiv 0$ (mod $p$), each subquotient of the above filtration is quasi-isomorphic to the one-dimensional space $\Bbbk f$.

\vspace{0.1in}

\emph{Step III.} We now examine the $p$-DG $\sym_{p-1}$-modules $X_{\od}(a)$ when $a+D\equiv 0$ (mod $p$). Let $N_{\od}$ denote the $\Bbbk$-span of $M_{\od}$. According to equation \eqref{difonXd}, $N_{\od}$ is preserved by $\dif_a$, as is $\sym_{p-1}^\prime \cdot M_{\od}$. This gives a splitting of $X_{\od}(a)$, compatible with the filtration of the previous step. In particular, $\sym_{p-1}^\prime \cdot M_{\od}$ is contractible as a $p$-complex, and $X_{\od}(a)$ is quasi-isomorphic to $N_{\od}$. The differential on $N_{\od}$ is given by $\dif '$, or more explicitly, by
\begin{equation} \label{difonNd}
\dif_{N_{\od}}(\prod e_i^{b_i}) = \sum_i \left(- (i+1) b_i e_{i+1} e_i^{b_i-1} \prod_{j \ne i} e_j^{b_j}\right).
\end{equation}

Thus it remains to examine the $p$-complex $N_{\od}$. Note that each block is acted on independently by the differential. Moreover, shifting the indices in $e_i$ by $p$ will not affect the formula, because the only coefficient which depends on the index is $(i+1)$. Therefore, as a $p$-complex $N_{\od}$ is actually a tensor product $\bigotimes_i N_{d_i}$, where $N_d$ is spanned by monomials $y_0^{b_0} y_1^{b_1} \cdots y_{p-1}^{b_{p-1}}$ with $\sum_i b_i = d$, and with a formula for $\dif_N$ identical to \eqref{difonNd} above.

Let us observe that the differential on $N_{\od}$ is independent of our original parameter $a$. Also, as a tensor product, $N_{\od}$ will be contractible if any $N_{d_i}$ is contractible. Suppose that we can prove that $N_d$ is contractible unless $d = 0$ mod $p$, when it is quasi-isomorphic to $\Bbbk \cdot y_0^d$ with zero differential. If this is the case, then $N_{\od}$ is contractible unless each $d_i=0$ mod $p$, meaning that $D = 0$ (mod $p$) and therefore $a = 0$. This will prove that $\Lambda(a)$ is contractible when $a \ne 0$. Moreover, when $a=0$, $\Lambda(0)$ is quasi-isomorphic to the sum over all $\od = (c_1 p,c_2 p,\ldots)$ of $\Bbbk \cdot e_p^{c_1 p} e_{2p}^{c_2 p} \cdots$, which is exactly the image of $\Bbbk[e_p^p,e_{2p}^p,\ldots]$.

Thus we have reduced the problem to the following lemma. \end{proof}

\begin{lemma} The $p$-complex $N_d$ is contractible unless $d\equiv 0$ (mod $p$). If $d \equiv 0$ (mod $p$) then $N_d$ is quasi-isomorphic to $\Bbbk \cdot y_0^d$ with the trivial differential.
\end{lemma}

\begin{proof} Before we examined a module by filtering it with subquotients isomorphic to $\sym_{p-1}$. Now we take the opposite route, embedding $N_d$ inside $\sym_{p-1}$.
	
Consider the map from $N_d$ to $\sym_{p-1}(-d)$ which sends $y_0^{b_0} y_1^{b_1} \cdots y_{p-1}^{b_{p-1}} \mapsto e_1^{b_1} \cdots e_{p-1}^{b_{p-1}}$. It is an exercise to check that this map intertwines the differentials on both sides. In this exercise, the interesting term in $\dif_N(\prod_{i \ge 0} y_i^{b_i})$ is $-b_0 y_0^{b_0-1} y_1^{b_1+1} \cdots$, while the interesting term in $\dif_{-d}(\prod_{i > 0} e_i^{b_i})$ is $(-d+\sum_{i > 0} b_i) e_1^{b_1+1} \cdots$. These terms intertwine, because $d = \sum_{i \ge 0} b_i$.

The image of $N_d \to \sym_{p-1}(-d)$ is the span of polynomials in $\sym_{p-1}(-d)$ with exponent $\le d$. This is actually a summand of $\sym_{p-1}(-d)$. For any monomial $f$ with exponent $q$, $\dif_{-d}(f)$ will be a sum of other monomials with exponent $q$ and the term $(q-d)e_1f$, which has higher exponent. When $q\equiv d$ (mod $p$), this higher exponent term vanishes. Now the comment after Corollary \ref{cor-sym-p-1-all-ideal-contractible}, gives the desired result.
\end{proof}

We write $\mH(\Lambda)$ for the cohomology $p$-DG algebra of $\Lambda$, i.e.~$\mH(\Lambda)=\Bbbk[e_p^p, e_{2p}^p, e_{3p}^p,\dots]$ with zero differential. It is isomorphic to the tensor product of one-variable polynomial $p$-DG algebras
\begin{equation}\label{eqn-tensor-Lambdap}
\mH(\Lambda) \cong \bigotimes_{k\in \N\backslash\{0\}}\Bbbk[e_{kp}^p].
\end{equation}

\begin{cor}\label{cor-derived-category-of-lambda}
Let $\iota:\mH(\Lambda)\lra \Lambda$ be the natural inclusion of $p$-DG algebras.
\begin{itemize}
\item[(i)] The derived induction functor along the natural inclusion $\iota: \mH(\Lambda) \lra \Lambda$
\[\iota^*:\mc{D}(\mH(\Lambda)) \lra \mc{D}(\Lambda) \]
induces an equivalence of triangulated categories. A quasi-inverse is given by the restriction functor $\iota_*$.
\item[(ii)] There are isomorphism of Grothendieck groups as $\mathbb{O}_p$-modules
\[K_0(\Lambda)\cong \mathbb{O}_p , \ \ \ \  G_0(\Lambda) \cong \mathbb{O}_p.\]
\end{itemize}
\end{cor}
\begin{proof}The derived equivalence follows from Proposition \ref{prop-cohomology-of-sym-function} and Corollary \ref{cor-qis-algebra-equivalence-derived-categories}.
The second part follows from Corollary \ref{cor-K-group-positive}.
\end{proof}


\section{\texorpdfstring{$p$}{p}-Differentials on \texorpdfstring{$\mathcal{U}$}{U}} \label{sec-pdifferentialsonU}
In this chapter we review Lauda's category $\mc{U}$ and define a $p$-DG $2$-category structure on it. We assume the reader is familiar with string diagrams for $2$-categories. An introduction to the topic can be found in \cite[Section 4]{Lau1}.

%
\subsection{The category \texorpdfstring{$\mc{U}$}{U}}\label{defn-u}
%

\begin{defn}\label{def-u-dot} The $2$-category $\UC$ is an additive graded $\Bbbk$-linear category. It has one object for each $\lambda \in \Z$, where $\Z$ is the weight lattice of $\mfsl_2$. The $1$-morphisms are (direct sums of grading shifts of) composites of the generating $1$-morphisms $\1_{\lambda}$, $\1_{\l+2} \EC \1_\l$ and $\1_\l \FC \1_{\l+2}$, for each $\l \in \Z$. Each $\1_{\l+2} \EC \1_\l$ will be drawn the same, regardless of the object $\l$. One imagines that there is a single $1$-morphism $\EC$ which increases the weight by $2$ (from right to left). We draw the identity $1$-morphism by empty diagrams labeled by $\lambda$, and the other $1$-morphisms as oriented strands. 

\begin{align*}
\begin{tabular}{|c|c|c|}
	\hline
	$1$-\textrm{morphism Generator} &
	\begin{DGCpicture}
	\DGCstrand(0,0)(0,1)
	\DGCdot*>{0.5}
	\DGCcoupon*(0.1,0.25)(1,0.75){$^\l$}
	\DGCcoupon*(-1,0.25)(-0.1,0.75){$^{\l+2}$}
    \DGCcoupon*(-0.25,1)(0.25,1.15){}
    \DGCcoupon*(-0.25,-0.15)(0.25,0){}
	\end{DGCpicture}&
	\begin{DGCpicture}
	\DGCstrand(0,0)(0,1)
	\DGCdot*<{0.5}
	\DGCcoupon*(0.1,0.25)(1,0.75){$^{\l+2}$}
	\DGCcoupon*(-1,0.25)(-0.1,0.75){$^\l$}
    \DGCcoupon*(-0.25,1)(0.25,1.15){}
    \DGCcoupon*(-0.25,-0.15)(0.25,0){}
	\end{DGCpicture} \\ \hline
	\textrm{Name} & $\EC$ & $\FC$ \\
	\hline
\end{tabular}
\end{align*}

The weight of any region in a diagram is determined by the weight of any single region. When no region is labelled, the ambient weight is irrelevant.

The $2$-morphisms will be generated by the following pictures.
	
\begin{align*}
\begin{tabular}{|c|c|c|c|c|}
  \hline
  \textrm{Generator} &
  \begin{DGCpicture}
  \DGCstrand(0,0)(0,1)
  \DGCdot*>{0.75}
  \DGCdot{0.45}
  \DGCcoupon*(0.1,0.25)(1,0.75){$^\l$}
  \DGCcoupon*(-1,0.25)(-0.1,0.75){$^{\l+2}$}
  \DGCcoupon*(-0.25,1)(0.25,1.15){}
  \DGCcoupon*(-0.25,-0.15)(0.25,0){}
  \end{DGCpicture}&
  \begin{DGCpicture}
  \DGCstrand(0,0)(0,1)
  \DGCdot*<{0.25}
  \DGCdot{0.65}
  \DGCcoupon*(0.1,0.25)(1,0.75){$^{\l}$}
  \DGCcoupon*(-1,0.25)(-0.1,0.75){$^{\l-2}$}
  \DGCcoupon*(-0.25,1)(0.25,1.15){}
  \DGCcoupon*(-0.25,-0.15)(0.25,0){}
  \end{DGCpicture} &
  \begin{DGCpicture}
  \DGCstrand(0,0)(1,1)
  \DGCdot*>{0.75}
  \DGCstrand(1,0)(0,1)
  \DGCdot*>{0.75}
  \DGCcoupon*(1.1,0.25)(2,0.75){$^\l$}
  \DGCcoupon*(-1,0.25)(-0.1,0.75){$^{\l+4}$}
  \DGCcoupon*(-0.25,1)(0.25,1.15){}
  \DGCcoupon*(-0.25,-0.15)(0.25,0){}
  \end{DGCpicture} &
  \begin{DGCpicture}
  \DGCstrand(0,0)(1,1)
  \DGCdot*<{0.25}
  \DGCstrand(1,0)(0,1)
  \DGCdot*<{0.25}
  \DGCcoupon*(1.1,0.25)(2,0.75){$^\l$}
  \DGCcoupon*(-1,0.25)(-0.1,0.75){$^{\l-4}$}
  \DGCcoupon*(-0.25,1)(0.25,1.15){}
  \DGCcoupon*(-0.25,-0.15)(0.25,0){}
  \end{DGCpicture} \\ \hline
  \textrm{Degree}  & 2   & 2 & -2 & -2 \\
  \hline
\end{tabular}
\end{align*}

\begin{align*}
\begin{tabular}{|c|c|c|c|c|}
  \hline
  \textrm{Generator} &
  \begin{DGCpicture}
  \DGCstrand/d/(0,0)(1,0)
  \DGCdot*>{-0.25,1}
  \DGCcoupon*(1,-0.5)(1.5,0){$^\l$}
  \DGCcoupon*(-0.25,0)(1.25,0.15){}
  \DGCcoupon*(-0.25,-0.65)(1.25,-0.5){}
  \end{DGCpicture} &
  \begin{DGCpicture}
  \DGCstrand/d/(0,0)(1,0)
  \DGCdot*<{-0.25,2}
  \DGCcoupon*(1,-0.5)(1.5,0){$^\l$}
  \DGCcoupon*(-0.25,0)(1.25,0.15){}
  \DGCcoupon*(-0.25,-0.65)(1.25,-0.5){}
  \end{DGCpicture}&
  \begin{DGCpicture}
  \DGCstrand(0,0)(1,0)/d/
  \DGCdot*<{0.25,1}
  \DGCcoupon*(1,0)(1.5,0.5){$^\l$}
  \DGCcoupon*(-0.25,0.5)(1.25,0.65){}
  \DGCcoupon*(-0.25,-0.15)(1.25,0){}
  \end{DGCpicture}&
  \begin{DGCpicture}
  \DGCstrand(0,0)(1,0)/d/
  \DGCdot*>{0.25,2}
  \DGCcoupon*(1,0)(1.5,0.5){$^\l$}
  \DGCcoupon*(-0.25,0.5)(1.25,0.65){}
  \DGCcoupon*(-0.25,-0.15)(1.25,0){}
  \end{DGCpicture}  \\ \hline
  \textrm{Degree} & $1+\l$ & $1-\l$ & $1+\l$ & $1-\l$ \\
  \hline
\end{tabular}
\end{align*}
\end{defn}

We will give a list of relations shortly, after we discuss some notation. For a product of $m$ dots on a single strand, we draw a single dot labeled by
$m$. Here is the case $m=3$.

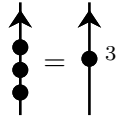
\begin{align*}
\begin{DGCpicture}[scale=.75]
\DGCstrand(0,0)(0,2)
\DGCdot{.4}
\DGCdot{.8}
\DGCdot{1.2}
\DGCdot*>{1.75}
\end{DGCpicture}
~=~
\begin{DGCpicture}[scale=.75]
\DGCstrand(0,0)(0,2)
\DGCdot{1}[r]{$^3$}
\DGCdot*>{1.75}
\end{DGCpicture}
\end{align*}

A \emph{closed diagram} is a diagram without boundary, constructed from the generators above. The simplest non-trivial closed diagram is a \emph{bubble}, which can be
oriented clockwise or counter-clockwise.
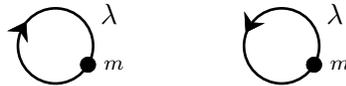
\begin{align*}
\begin{DGCpicture}
\DGCbubble(0,0){0.5}
\DGCdot*<{0.25,L}
\DGCdot{-0.25,R}[r]{$_m$}
\DGCdot*.{0.25,R}[r]{$\l$}
\end{DGCpicture}
\qquad \qquad
\begin{DGCpicture}
\DGCbubble(0,0){0.5}
\DGCdot*>{0.25,L}
\DGCdot{-0.25,R}[r]{$_m$}
\DGCdot*.{0.25,R}[r]{$\l$}
\end{DGCpicture}
\end{align*}
In an arbitrary closed diagram, the interior of an oriented circle might contain another closed diagram. We use the term bubble only when the interior is empty.

A simple calculation shows that the degree of a bubble with $m$ dots in a region labeled $\l$ is $2(m+1-\l)$ if the bubble is clockwise, and $2(m+1+\l)$ if the bubble is counter-clockwise. Instead of keeping track of the number $m$ of dots a bubble has, it will be much more illustrative to keep track of the degree of the bubble, which is in $2\Z$. We will use the following shorthand to refer to a bubble of degree $2k$.

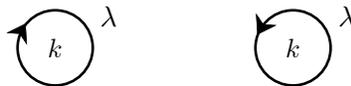
\begin{align*}
\bigcwbubble{$k$}{$\l$}
\qquad \qquad
\bigccwbubble{$k$}{$\l$}
\end{align*}

We will never label the region on the inside of a bubble with a weight; an integer inside a bubble always refers to (half) the degree. Analogous notation is used in \cite{KLMS}, where a spade is used to designate the
number of dots which would yield degree $0$. Bubbles have a life of their own, independent of their presentation in terms of caps, cups, and dots, and the
notation emphasizes this fact.

Note that $\l$ can be any integer, but $m\ge 0$ because it counts dots. Therefore, we can only construct a clockwise (resp. counter-clockwise) bubble of degree $k$ when $k \ge 1-\l$ (resp. $k \ge 1+\l$). These are called \emph{real bubbles}. Following Lauda, we also allow bubbles drawn as above with arbitrary $k \in \Z$. Bubbles with $k$ outside of the appropriate range are not yet defined in terms of the generating maps; we call these \emph{fake bubbles}. Using the relations below, one can express any fake bubble in terms of real bubbles.

Now we list the relations. Whenever the region label is omitted, the relation applies to all ambient weights.

\begin{itemize}
\item[(i)] {\bf Biadjointness relations.}
\begin{subequations} \label{biadjoint}
\begin{align} \label{biadjoint1}
\begin{DGCpicture}[scale=0.85]
\DGCstrand(0,0)(0,1)(1,1)(2,1)(2,2)
\DGCdot*>{0.75,1}
\DGCdot*>{1.75,1}
\end{DGCpicture}
~=~
\begin{DGCpicture}[scale=0.85]
\DGCstrand(0,0)(0,2)
\DGCdot*>{0.5}
\end{DGCpicture}
~=~
\begin{DGCpicture}[scale=0.85]
\DGCstrand(2,0)(2,1)(1,1)(0,1)(0,2)
\DGCdot*>{0.75}
\DGCdot*>{1.75}
\end{DGCpicture}
\qquad \qquad
\begin{DGCpicture}[scale=0.85]
\DGCstrand(0,0)(0,1)(1,1)(2,1)(2,2)
\DGCdot*<{0.75,1}
\DGCdot*<{1.75,1}
\end{DGCpicture}
~=~
\begin{DGCpicture}[scale=0.85]
\DGCstrand(0,0)(0,2)
\DGCdot*<{0.5}
\end{DGCpicture}
~=~
\begin{DGCpicture}[scale=0.85]
\DGCstrand(2,0)(2,1)(1,1)(0,1)(0,2)
\DGCdot*<{0.75}
\DGCdot*<{1.75}
\end{DGCpicture}
\end{align}

\begin{align} \label{biadjointdot}
\begin{DGCpicture}
\DGCstrand(0,0)(0,.5)(1,.5)/d/(1,0)/d/
\DGCdot*<{1}
\DGCdot{.3,1}
\end{DGCpicture}
~=~
\begin{DGCpicture}
\DGCstrand(0,0)(0,.5)(1,.5)/d/(1,0)/d/
\DGCdot*<{1}
\DGCdot{.3,2}
\end{DGCpicture}
\qquad \qquad \qquad
\begin{DGCpicture}
\DGCstrand(0,0)(0,.5)(1,.5)/d/(1,0)/d/
\DGCdot*>{1}
\DGCdot{.3,1}
\end{DGCpicture}
~=~
\begin{DGCpicture}
\DGCstrand(0,0)(0,.5)(1,.5)/d/(1,0)/d/
\DGCdot*>{1}
\DGCdot{.3,2}
\end{DGCpicture}
\end{align}

\begin{align} \label{biadjointcrossing}
\begin{DGCpicture}[scale=0.75]
\DGCstrand(0,0)(1,1)/u/(2,1)/d/(2,0)/d/
\DGCdot*<{1.5}
\DGCstrand(1,0)(0,1)/u/(3,1)/d/(3,0)/d/
\DGCdot*<{2.5}
\end{DGCpicture}
~=~
\begin{DGCpicture}[scale=0.75]
\DGCstrand(0,0)(0,1)(3,1)/d/(2,0)/d/
\DGCdot*<{2.5}
\DGCstrand(1,0)(1,1)(2,1)/d/(3,0)/d/
\DGCdot*<{1.5}
\end{DGCpicture}
\qquad \qquad
\begin{DGCpicture}[scale=0.75]
\DGCstrand(0,0)(1,1)/u/(2,1)/d/(2,0)/d/
\DGCdot*>{1.5}
\DGCstrand(1,0)(0,1)/u/(3,1)/d/(3,0)/d/
\DGCdot*>{2.5}
\end{DGCpicture}
~=~
\begin{DGCpicture}[scale=0.75]
\DGCstrand(0,0)(0,1)(3,1)/d/(2,0)/d/
\DGCdot*>{2.5}
\DGCstrand(1,0)(1,1)(2,1)/d/(3,0)/d/
\DGCdot*>{1.5}
\end{DGCpicture}
\end{align}
\end{subequations}

\item[(ii)] {\bf Positivity and Normalization of bubbles.} Positivity states that all bubbles (real or fake) of negative degree should be zero.
\begin{subequations} \label{negzerobubble}
\begin{align} \label{negbubble}
\bigcwbubble{$k$}{}~=~0~=~\bigccwbubble{$k$}{}
\qquad
\textrm{if}~k<0,
\end{align}

Normalization states that degree $0$ bubbles are equal to the empty diagram (i.e. the identity $2$-morphism of the identity $1$-morphism).

\begin{align} \label{zerobubble}
\bigcwbubble{$0$}{}~=~1~=~\bigccwbubble{$0$}{}
\end{align}
\end{subequations}

\item[(iii)] {\bf Infinite Grassmannian relations.} This family of relations can be expressed most succinctly in terms of generating functions.

\begin{equation}\label{eqn-infinite-Grassmannian}
\left( \cwbubble{$0$}{}+t~\cwbubble{$1$}{}+t^2~\cwbubble{$2$}{}+\ldots \right) \cdot
\left( \ccwbubble{$0$}{}+t~\ccwbubble{$1$}{}+t^2~\ccwbubble{$2$}{}+\ldots \right)  =  1~.
\end{equation}

The cohomology ring of the ``infinite dimensional Grassmannian" is the ring $\Lambda$ of symmetric functions. Inside this ring, there is an analogous relation $e(t)h(t)=1$, where $e(t) = \sum_{i \ge 0} (-1)^i e_i t^i$ is the total Chern class of the tautological bundle, and $h(t) = \sum_{i \ge 0} h_i t^i$ is the total Chern class of the dual bundle. Lauda has proved that the bubbles in a single region generate an algebra inside $\UC$ isomorphic to $\Lambda$.

Looking at the homogeneous component of degree $m$, we have the following equation.
\begin{align}
\sum_{a + b = m} \bigcwbubble{$a$}{} \bigccwbubble{$b$}{} = \delta_{m,0} \label{infgrass}
\end{align}
Because of the positivity of bubbles relation, this equation holds true for any $m \in \Z$, and the sum can be taken over all $a,b \in \Z$.

Using these equations one can express all (positive degree) counter-clockwise bubbles in terms of clockwise bubbles, and vice versa.
Consequentially, all fake bubbles can be expressed in terms of real bubbles.

\begin{rmk} This relation is actually redundant, but we include it for pedagogical reasons. \end{rmk}
\item[(iv)] {\bf NilHecke relations.} The upward pointing strands satisfy nilHecke relations
\begin{subequations} \label{NHrels}
\begin{align}
\begin{DGCpicture}
\DGCstrand(0,0)(1,1)(0,2)
\DGCdot*>{2}
\DGCstrand(1,0)(0,1)(1,2)
\DGCdot*>{2}
\end{DGCpicture}
=0, \label{NHrelR2}
\end{align}
\begin{align}
\RIII{L}{$0$}{$0$}{$0$}{no} = \RIII{R}{$0$}{$0$}{$0$}{no},
\label{NHrelR3}
\end{align}
\begin{align}
\crossing{$0$}{$0$}{$1$}{$0$}{no} - \crossing{$0$}{$1$}{$0$}{$0$}{no} = \twolines{$0$}{$0$}{no} = \crossing{$1$}{$0$}{$0$}{$0$}{no} -
\crossing{$0$}{$0$}{$0$}{$1$}{no}. \label{NHreldotforce}
\end{align}
\end{subequations}
\item[(v)] {\bf Reduction to bubbles.} The following equalities hold for all $\l \in \Z$.
\begin{subequations} \label{bubblereduction}
\begin{align}
\curl{R}{U}{$\l$}{no}{$0$} = -\sum_{a+b=-\l} \oneline{$b$}{$\l$} \bigcwbubble{$a$}{},
\end{align}
\begin{align}
\curl{L}{U}{$\l$}{no}{$0$} = \sum_{a+b=\l} \bigccwbubble{$a$}{$\l$} \oneline{$b$}{no}
\end{align}
\end{subequations}
These sums only take values for $a,b \ge 0$. Therefore, when $\l \ne 0$, either the right curl or the left curl is zero.
\item[(vi)] {\bf Identity decomposition.} The following equations hold for all $\l \in \Z$.
\begin{subequations} \label{IdentityDecomp}
\begin{align}
\begin{DGCpicture}
\DGCstrand(0,0)(0,2)
\DGCdot*>{1}
\DGCdot*.{1.25}[l]{$\l$}
\DGCstrand(1,0)(1,2)
\DGCdot*<{1}
\DGCdot*.{1.25}[r]{$\l$}
\end{DGCpicture}
~=~-~
\begin{DGCpicture}
\DGCstrand(0,0)(1,1)(0,2)
\DGCdot*>{0.25}
\DGCdot*>{1}
\DGCdot*>{1.75}
\DGCstrand(1,0)(0,1)(1,2)
\DGCdot*.{1.25}[l]{$\l$}
\DGCdot*<{0.25}
\DGCdot*<{1}
\DGCdot*<{1.75}
\end{DGCpicture}
~+~
\sum_{a+b+c=\l-1}~
\cwcapbubcup{$a$}{$b$}{$c$}{$\l$} \label{IdentityDecompPos}\end{align}
\begin{align}
\begin{DGCpicture}
\DGCstrand(0,0)(0,2)
\DGCdot*<{1}
\DGCdot*.{1.25}[l]{$\l$}
\DGCstrand(1,0)(1,2)
\DGCdot*>{1}
\DGCdot*.{1.25}[r]{$\l$}
\end{DGCpicture}
~=~-~
\begin{DGCpicture}
\DGCstrand(0,0)(1,1)(0,2)
\DGCdot*<{0.25}
\DGCdot*<{1}
\DGCdot*<{1.75}
\DGCstrand(1,0)(0,1)(1,2)
\DGCdot*.{1.25}[l]{$\l$}
\DGCdot*>{0.25}
\DGCdot*>{1}
\DGCdot*>{1.75}
\end{DGCpicture}
~+~
\sum_{a+b+c=-\l-1}~
\ccwcapbubcup{$a$}{$b$}{$c$}{$\l$} \label{IdentityDecompNeg}\end{align}
\end{subequations}
The sum in the first equality vanishes for $\l \le 0$, and the sum in the second equality vanishes for $\l \ge 0$.

The terms on the right hand side form a collection of orthogonal idempotents, which we shall explore in detail in the next chapter.
\end{itemize}

%
\subsection{Some properties of \texorpdfstring{$\mc{U}$}{U}}
\label{remarksandequalities}
%

In this section we gather together some known properties of $\mc{U}$ that will be useful later.

\begin{rmk}\label{rmk-splitting-U-into-two} There are no $1$-morphisms which change the weight by an odd number. Therefore $\UC \cong \UC_{\even} \oplus \UC_{\odd}$, where $\UC_{\even}$ only has even weights $\l \in 2\Z$, and $\UC_{\odd}$ only has odd weights.  \end{rmk}

\begin{rmk} In \cite{KL1}, Khovanov and Lauda define a monoidal category $\UC^+$ which categorifies the positive half of $U_q(\mathfrak{sl}_2)$. One can think of this category as spanned by diagrams with only upward-oriented strands. The objects are given by $\EC^m$, $m \in \N$ with the obvious monoidal structure; the morphisms are generated by the dot and the crossing only; the relations are given by the nilHecke relations \eqref{NHrels} only. In $\UC^+$, $\END(\EC^m) = \NH_m$ is the nilHecke algebra on $m$ strands. Regions are not labeled in diagrams for $\UC^+$, but for any $\l \in \Z$ there is a well-defined functor $\UC^+ \to {}_\l \UC$ which labels the leftmost region by $\l$, sending $\EC^m \mapsto \1_\l \EC^m \1_{\l-2m}$.
\end{rmk}

\begin{rmk} \label{rmk-morphisms-in-u} In Theorem 8.3 of \cite{Lau1}, Lauda has classified the space of $2$-morphisms in $\UC$. Using biadjunction and the isomorphisms implied by \eqref{IdentityDecomp}, any nonzero $2$-morphism space is related to the endomorphism spaces $\END_{{_{\l}\UC_{\l-2m}}}(\1_\l \EC^m \1_{\l-2m})$ for various $m \ge 0$ and $\l \in \Z$. Let $\Lambda$ represent the ring of symmetric functions, identified with the clockwise bubbles in the region labeled $\l$. The map $\Lambda \o \NH_m \to \END_{_\l\UC_{\l-2m}}(\1_\l \EC^m\1_{\l-2m})$ is an isomorphism.
\end{rmk}

Whenever we assert that certain diagrams form a basis for the $2$-morphisms in a certain degree, we are implicitly using Lauda's classification of morphisms.
This classification implies that all complicated closed diagrams will reduce to clockwise bubbles.

\paragraph{Symmetries of $\UC$.}The $2$-category $\UC$ enjoys various symmetries (see \cite[Section 5.6]{Lau1}). We list them below. To keep track of the covariance of these
functors, we introduce some auxiliary $2$-categories, each of which shares the same objects $\l \in \Z$ as $\UC$. In $\UC^{op}$ the $1$-morphisms are reversed, so that horizontal concatenation is reversed but vertical concatenation is unchanged. In $\UC^{co}$, the $2$-morphisms are reversed. In $\UC^{coop}$ both the $1$-morphisms and the $2$-morphisms are reversed.

\vspace{0.1in}

($I$) The $2$-functor $\widetilde{\omega}:\UC\lra \UC$ rescales the crossing by $-1$, inverts the orientation of each strand, and negates each region label. More precisely, it is given on the generators of $\UC$ as follows.
  \[\begin{array}{ccccc}
   \widetilde{\omega}
  \left(~
  \begin{DGCpicture}
  \DGCstrand(0,0)(1,1)
  \DGCdot*>{1}
  \DGCstrand(1,0)(0,1)
  \DGCdot*>{1}
  \DGCcoupon*(0.9,0.25)(1.25,0.75){$^\l$}
  \end{DGCpicture}
  \right) = -
  \begin{DGCpicture}
  \DGCstrand(0,0)(1,1)
  \DGCdot*<{0}
  \DGCstrand(1,0)(0,1)
  \DGCdot*<{0}
  \DGCcoupon*(0.9,0.25)(1.25,0.75){$^{-\l}$}
  \end{DGCpicture}, &&
  \widetilde{\omega}\left(~
  \begin{DGCpicture}
  \DGCstrand(0,0)(1,0)/d/
  \DGCdot*>{0.25,2}
  \DGCcoupon*(1,0)(1.25,0.5){$^\l$}
  \end{DGCpicture}
  \right)=
  \begin{DGCpicture}
  \DGCstrand(0,0)(1,0)/d/
  \DGCdot*<{0.25,2}
  \DGCcoupon*(1,0)(1.35,0.5){$^{-\l}$}
  \end{DGCpicture}, &&
  \widetilde{\omega}\left(
  \begin{DGCpicture}
  \DGCstrand(0,0)(1,0)/d/
  \DGCdot*<{0.25,2}
  \DGCcoupon*(1,0)(1.25,0.5){$^\l$}
  \end{DGCpicture}
  \right)=
  \begin{DGCpicture}
  \DGCstrand(0,0)(1,0)/d/
  \DGCdot*>{0.25,2}
  \DGCcoupon*(1,0)(1.35,0.5){$^{-\l}$}
  \end{DGCpicture},\\
  &&&&\\
  \widetilde{\omega}
  \left(
  \begin{DGCpicture}
  \DGCstrand(0,0)(1,1)
  \DGCdot*<{0}
  \DGCstrand(1,0)(0,1)
  \DGCdot*<{0}
  \DGCcoupon*(0.9,0.25)(1.25,0.75){$^{\l}$}
  \end{DGCpicture}
  ~\right)= -
  \begin{DGCpicture}
  \DGCstrand(0,0)(1,1)
  \DGCdot*>{1}
  \DGCstrand(1,0)(0,1)
  \DGCdot*>{1}
  \DGCcoupon*(0.9,0.25)(1.25,0.75){$^{-\l}$}
  \end{DGCpicture},&&
  \widetilde{\omega}\left(~
  \begin{DGCpicture}
  \DGCstrand/d/(0,0)(1,0)
  \DGCdot*<{-0.25,2}
  \DGCcoupon*(1,-0.5)(1.25,0){$^\l$}
  \end{DGCpicture}
  \right)=
  \begin{DGCpicture}
  \DGCstrand/d/(0,0)(1,0)
  \DGCdot*>{-0.25,2}
  \DGCcoupon*(1,-0.5)(1.5,0){$^{-\l}$}
  \end{DGCpicture},
  &&
    \widetilde{\omega}\left(~
  \begin{DGCpicture}
  \DGCstrand/d/(0,0)(1,0)
  \DGCdot*>{-0.25,2}
  \DGCcoupon*(1,-0.5)(1.25,0){$^\l$}
  \end{DGCpicture}
  \right)=
  \begin{DGCpicture}
  \DGCstrand/d/(0,0)(1,0)
  \DGCdot*<{-0.25,2}
  \DGCcoupon*(1,-0.5)(1.5,0){$^{-\l}$}
  \end{DGCpicture}.
  \end{array}
   \]

\vspace{0.1in}

($II$) The $2$-functor $\widetilde{\sigma}:\UC\lra \UC^{op}$ rescales a crossing by $-1$, reflects any diagram along a vertical axis, and negates each region label. More precisely:
  \[
  \begin{array}{ccccc}
  \widetilde{\sigma}
  \left(~
  \begin{DGCpicture}
  \DGCstrand(0,0)(1,1)
  \DGCdot*>{1}
  \DGCstrand(1,0)(0,1)
  \DGCdot*>{1}
  \DGCcoupon*(0.9,0.25)(1.25,0.75){$^\l$}
  \end{DGCpicture}
  \right) = -
  \begin{DGCpicture}
  \DGCstrand(0,0)(1,1)
  \DGCdot*>{1}
  \DGCstrand(1,0)(0,1)
  \DGCdot*>{1}
  \DGCcoupon*(-0.35,0.25)(0.1,0.75){$^{-\l}$}
  \end{DGCpicture}, &&
  \widetilde{\sigma}\left(~
  \begin{DGCpicture}
  \DGCstrand(0,0)(1,0)/d/
  \DGCdot*>{0.25,2}
  \DGCcoupon*(1,0)(1.25,0.5){$^\l$}
  \end{DGCpicture}
  \right)=
  \begin{DGCpicture}
  \DGCstrand(0,0)(1,0)/d/
  \DGCdot*<{0.25,2}
  \DGCcoupon*(1,0)(1.35,0.5){$^{-\l}$}
  \end{DGCpicture}, &&
  \widetilde{\sigma}\left(
  \begin{DGCpicture}
  \DGCstrand(0,0)(1,0)/d/
  \DGCdot*<{0.25,2}
  \DGCcoupon*(1,0)(1.25,0.5){$^\l$}
  \end{DGCpicture}
  \right)=
  \begin{DGCpicture}
  \DGCstrand(0,0)(1,0)/d/
  \DGCdot*>{0.25,2}
  \DGCcoupon*(1,0)(1.35,0.5){$^{-\l}$}
  \end{DGCpicture},\\
  &&&&\\
  \widetilde{\sigma}
  \left(
  \begin{DGCpicture}
  \DGCstrand(0,0)(1,1)
  \DGCdot*<{0}
  \DGCstrand(1,0)(0,1)
  \DGCdot*<{0}
  \DGCcoupon*(-0.35,0.25)(0.1,0.75){$^{\l}$}
  \end{DGCpicture}
  ~\right)=-
  \begin{DGCpicture}
  \DGCstrand(0,0)(1,1)
  \DGCdot*<{0}
  \DGCstrand(1,0)(0,1)
  \DGCdot*<{0}
  \DGCcoupon*(0.9,0.25)(1.25,0.75){$^{-\l}$}
  \end{DGCpicture},&&
  \widetilde{\sigma}\left(~
  \begin{DGCpicture}
  \DGCstrand/d/(0,0)(1,0)
  \DGCdot*<{-0.25,2}
  \DGCcoupon*(1,-0.5)(1.25,0){$^\l$}
  \end{DGCpicture}
  \right)=
  \begin{DGCpicture}
  \DGCstrand/d/(0,0)(1,0)
  \DGCdot*>{-0.25,2}
  \DGCcoupon*(1,-0.5)(1.5,0){$^{-\l}$}
  \end{DGCpicture},
  &&
  \widetilde{\sigma}\left(~
  \begin{DGCpicture}
  \DGCstrand/d/(0,0)(1,0)
  \DGCdot*>{-0.25,2}
  \DGCcoupon*(1,-0.5)(1.25,0){$^\l$}
  \end{DGCpicture}
  \right)=
  \begin{DGCpicture}
  \DGCstrand/d/(0,0)(1,0)
  \DGCdot*<{-0.25,2}
  \DGCcoupon*(1,-0.5)(1.5,0){$^{-\l}$}
  \end{DGCpicture}.
  \end{array}
   \]

\vspace{0.1in}

($III$) The $2$-functor $\widetilde{\psi}:\UC\lra \UC^{co}$ reflects any diagram by a horizontal axis, and inverts the orientation of all strands. More precisely:
  \[
  \begin{array}{ccccc}
  \widetilde{\psi}
  \left(~
  \begin{DGCpicture}
  \DGCstrand(0,0)(1,1)
  \DGCdot*>{1}
  \DGCstrand(1,0)(0,1)
  \DGCdot*>{1}
  \DGCcoupon*(0.9,0.25)(1.25,0.75){$^\l$}
  \end{DGCpicture}
  \right) =
  \begin{DGCpicture}
  \DGCstrand(0,0)(1,1)
  \DGCdot*>{1}
  \DGCstrand(1,0)(0,1)
  \DGCdot*>{1}
  \DGCcoupon*(0.9,0.25)(1.25,0.75){$^{\l}$}
  \end{DGCpicture}, &&
  \widetilde{\psi}\left(~
  \begin{DGCpicture}
  \DGCstrand(0,0)(1,0)/d/
  \DGCdot*>{0.25,2}
  \DGCcoupon*(1,0)(1.25,0.5){$^\l$}
  \end{DGCpicture}
  \right)=
  \begin{DGCpicture}
  \DGCstrand/d/(0,0)(1,0)
  \DGCdot*<{-0.25,2}
  \DGCcoupon*(1,-0.5)(1.25,0){$^\l$}
  \end{DGCpicture}, &&
  \widetilde{\psi}\left(
  \begin{DGCpicture}
  \DGCstrand(0,0)(1,0)/d/
  \DGCdot*<{0.25,2}
  \DGCcoupon*(1,0)(1.25,0.5){$^\l$}
  \end{DGCpicture}
  \right)=
  \begin{DGCpicture}
  \DGCstrand/d/(0,0)(1,0)
  \DGCdot*>{-0.25,2}
  \DGCcoupon*(1,-0.5)(1.25,0){$^\l$}
  \end{DGCpicture},\\
  &&&&\\
  \widetilde{\psi}
  \left(
  \begin{DGCpicture}
  \DGCstrand(0,0)(1,1)
  \DGCdot*<{0}
  \DGCstrand(1,0)(0,1)
  \DGCdot*<{0}
  \DGCcoupon*(-0.35,0.25)(0.1,0.75){$^{\l}$}
  \end{DGCpicture}
  ~\right)=
  \begin{DGCpicture}
  \DGCstrand(0,0)(1,1)
  \DGCdot*<{0}
  \DGCstrand(1,0)(0,1)
  \DGCdot*<{0}
  \DGCcoupon*(-0.35,0.25)(0.1,0.75){$^{\l}$}
  \end{DGCpicture},&&
  \widetilde{\psi}\left(~
  \begin{DGCpicture}
  \DGCstrand/d/(0,0)(1,0)
  \DGCdot*<{-0.25,2}
  \DGCcoupon*(1,-0.5)(1.25,0){$^\l$}
  \end{DGCpicture}
  \right)=
  \begin{DGCpicture}
  \DGCstrand(0,0)(1,0)/d/
  \DGCdot*>{0.25,2}
  \DGCcoupon*(1,0)(1.25,0.5){$^\l$}
  \end{DGCpicture},
  &&
  \widetilde{\psi}\left(~
  \begin{DGCpicture}
  \DGCstrand/d/(0,0)(1,0)
  \DGCdot*>{-0.25,2}
  \DGCcoupon*(1,-0.5)(1.25,0){$^\l$}
  \end{DGCpicture}
  \right)=
  \begin{DGCpicture}
  \DGCstrand(0,0)(1,0)/d/
  \DGCdot*<{0.25,2}
  \DGCcoupon*(1,0)(1.25,0.5){$^\l$}
  \end{DGCpicture}.
  \end{array}
   \]

\vspace{0.1in}
($IV$) The $2$-functor $\widetilde{\tau}:\UC\lra \UC^{coop}$ rotates any diagram by $180^\circ$. More precisely:
  \[
  \begin{array}{ccccc}
  \widetilde{\tau}
  \left(~
  \begin{DGCpicture}
  \DGCstrand(0,0)(1,1)
  \DGCdot*>{1}
  \DGCstrand(1,0)(0,1)
  \DGCdot*>{1}
  \DGCcoupon*(0.9,0.25)(1.25,0.75){$^\l$}
  \end{DGCpicture}
  \right) =
  \begin{DGCpicture}
  \DGCstrand(0,0)(1,1)
  \DGCdot*<{0}
  \DGCstrand(1,0)(0,1)
  \DGCdot*<{0}
  \DGCcoupon*(-0.35,0.25)(0.1,0.75){$^{\l}$}
  \end{DGCpicture}, &&
  \widetilde{\tau}\left(~
  \begin{DGCpicture}
  \DGCstrand(0,0)(1,0)/d/
  \DGCdot*>{0.25,2}
  \DGCcoupon*(1,0)(1.25,0.5){$^\l$}
  \end{DGCpicture}
  \right)=
  \begin{DGCpicture}
  \DGCstrand/d/(0,0)(1,0)
  \DGCdot*<{-0.25,2}
  \DGCcoupon*(1,-0.5)(1.25,0){$^\l$}
  \end{DGCpicture}, &&
  \widetilde{\tau}\left(
  \begin{DGCpicture}
  \DGCstrand(0,0)(1,0)/d/
  \DGCdot*<{0.25,2}
  \DGCcoupon*(1,0)(1.25,0.5){$^\l$}
  \end{DGCpicture}
  \right)=
  \begin{DGCpicture}
  \DGCstrand/d/(0,0)(1,0)
  \DGCdot*>{-0.25,2}
  \DGCcoupon*(1,-0.5)(1.25,0){$^\l$}
  \end{DGCpicture},\\
  &&&&\\
  \widetilde{\tau}
  \left(
  \begin{DGCpicture}
  \DGCstrand(0,0)(1,1)
  \DGCdot*<{0}
  \DGCstrand(1,0)(0,1)
  \DGCdot*<{0}
  \DGCcoupon*(-0.35,0.25)(0.1,0.75){$^{\l}$}
  \end{DGCpicture}
  ~\right)=
  \begin{DGCpicture}
  \DGCstrand(0,0)(1,1)
  \DGCdot*>{1}
  \DGCstrand(1,0)(0,1)
  \DGCdot*>{1}
  \DGCcoupon*(0.9,0.25)(1.25,0.75){$^{\l}$}
  \end{DGCpicture},&&
  \widetilde{\tau}\left(~
  \begin{DGCpicture}
  \DGCstrand/d/(0,0)(1,0)
  \DGCdot*<{-0.25,2}
  \DGCcoupon*(1,-0.5)(1.25,0){$^\l$}
  \end{DGCpicture}
  \right)=
  \begin{DGCpicture}
  \DGCstrand(0,0)(1,0)/d/
  \DGCdot*>{0.25,2}
  \DGCcoupon*(1,0)(1.25,0.5){$^\l$}
  \end{DGCpicture},
  &&
  \widetilde{\tau}\left(~
  \begin{DGCpicture}
  \DGCstrand/d/(0,0)(1,0)
  \DGCdot*>{-0.25,2}
  \DGCcoupon*(1,-0.5)(1.25,0){$^\l$}
  \end{DGCpicture}
  \right)=
  \begin{DGCpicture}
  \DGCstrand(0,0)(1,0)/d/
  \DGCdot*<{0.25,2}
  \DGCcoupon*(1,0)(1.25,0.5){$^\l$}
  \end{DGCpicture}.
  \end{array}
   \]

%
\subsection{Derivations on \texorpdfstring{$\UC$}{U}}\label{deriv-on-u}
%
In this section, we define a multiparameter family of $2$-categorical derivations on Lauda's $2$-category $\UC$.

\begin{defn} \label{def-dif-on-U} Fix parameters $x_\l, \ox_\l,y_\l,\oy_\l,a_\l,\oa_\l \in \Bbbk$ for each $\l \in \Z$. These parameters are required to satisfy

\begin{subequations} \label{eqn-parameter-req}
\begin{equation} a_\l - \oa_\l = x_{\l+2}-x_\l - 2y_\l \label{paramA} \end{equation}
\begin{equation} x_\l + \ox_\l = -\l \label{paramX} \end{equation}
\begin{equation} y_\l + \oy_\l = - 1\label{paramY} \end{equation}
\end{subequations}

Let $\dif$ be the 2-categorical derivation on $\UC$, defined on generators as follows.

\begin{subequations} \label{eqn-dif-on-gen}
\begin{align}\label{dotsdif}
\dif \left( \onelineshort{$1$}{$\l$} \right) =  \onelineshort{$2$}{$\l$} \quad \quad \quad \quad
\dif \left( \onelineDshort{$1$}{$\l$} \right) =  \onelineDshort{$2$}{$\l$}
\end{align}
\begin{align}
\dif \left( \crossing{$0$}{$0$}{$0$}{$0$}{$\l$} \right) = a_\l \twolines{$0$}{$0$}{$\l$} + (-1-a_\l) \crossing{$1$}{$0$}{$0$}{$0$}{$\l$} + (-1 + a_\l) \crossing{$0$}{$1$}{$0$}{$0$}{$\l$} \label{crossingdif}
\end{align}
\begin{align}
\dif \left( \crossingD{$0$}{$0$}{$0$}{$0$}{$\l$} \right) = \oa_\l \twolinesD{$0$}{$0$}{$\l$} + (-1-\oa_\l) \crossingD{$1$}{$0$}{$0$}{$0$}{$\l$} +
(-1 + \oa_\l) \crossingD{$0$}{$1$}{$0$}{$0$}{$\l$}
\end{align}
\begin{align}
\dif \left( \cappy{CW}{$0$}{no}{$\l$} \right)  = x_{\l-2} \cappy{CW}{$1$}{no}{$\l$} + y_{\l-2} \cappy{CW}{$0$}{CW}{$\l$}
\end{align}
\begin{align}
\dif \left( \cuppy{CW}{$0$}{no}{$\l$} \right) = \ox_{\l-2} \cuppy{CW}{$1$}{no}{$\l$} + \oy_{\l-2} \cuppy{CW}{$0$}{CW}{$\l$}
\end{align}
\begin{align}
\dif \left( \cappy{CCW}{$0$}{no}{$\l$} \right) = -(\ox_\l + 2 \oy_\l) \cappy{CCW}{$1$}{no}{$\l$} - \oy_\l \cappy{CCW}{$0$}{CW}{$\l$}
\end{align}
\begin{align}
\dif \left( \cuppy{CCW}{$0$}{no}{$\l$} \right) = -(x_\l + 2 y_\l) \cuppy{CCW}{$1$}{no}{$\l$} - y_\l \cuppy{CCW}{$0$}{CW}{$\l$}
\end{align}
\end{subequations}

It is easy to see that the differential can be defined independently for $\UC_{\odd}$ and $\UC_{\even}$.
\end{defn}

\begin{prop}\label{prop-classification-of-dif-on-U} For any choice of parameters in Definition \ref{def-dif-on-U}, $\dif$ is a $2$-categorical derivation on $\mc{U}$. It is
$p$-nilpotent if and only if the parameters lie in the prime field $\F_p$. If $p \ne 2$, any 2-categorical derivation is of the form $b \cdot \dif$ for some constant $b
\in \Bbbk$ and some $\dif$ as above. \end{prop}

The proof of the proposition will be deferred until Appendix \ref{sec-classification-of-dif}, but it is a straightforward computation. That appendix also contains a brief discussion of some more exotic $2$-categorical differentials which can exist in characteristic $2$.

Until Appendix \ref{sec-classification-of-dif}, we assume any derivation $\dif$ on $\UC$ is as in Definition \ref{def-dif-on-U} (without any rescaling).

A quick calculation shows that the differential $\dif$ acts on bubbles in a simple and uniform way, independent of the parameters or the ambient weight!

\begin{cor}\label{cor-dif-action-on-bubbles}The differentials in Definition \ref{def-dif-on-U} act on bubbles by the following formulas.
\begin{subequations} \label{difbubformula-b=1}
\begin{align}
\dif \left( \bigcwbubble{$k$}{} \right) = (k+1) \bigcwbubble{${k+1}$}{} -  \bigcwbubble{$k$}{} \bigcwbubble{$1$}{} \label{difcwbub-b=1}
\end{align}
\begin{align}
\dif \left( \bigccwbubble{$k$}{} \right) = (k+1) \bigccwbubble{${k+1}$}{} -  \bigccwbubble{$k$}{} \bigccwbubble{$1$}{} \label{difccwbub-b=1} \end{align}
\end{subequations}
\end{cor}

Comparing this to Lemma \ref{lemma-dif-on-sym-function}, we see that the map $\Lambda \to \END_{_\l\UC_\l}(\1_\l)$ which sends $h_k \mapsto \cwbubble{$k$}{}$ and $(-1)^k e_k \mapsto
\ccwbubble{$k$}{}$ is an isomorphism of $p$-DG algebras. So is the dual map, which reverses the roles of $h_k$ and $(-1)^k e_k$.

The four symmetries on $\UC$ will each intertwine a $p$-differential as in Definition \ref{def-dif-on-U} with another such $p$-differential, having a different set of parameters.

\begin{cor}\label{cor-symmetries-intertwine-dif}Let $\dif$ be a $2$-categorical differential on $\UC$ parametrized by $a_\l, \oa_\l$, $x_\l, \ox_\l$ and $y_\l, \oy_\l$ as in Definition \ref{def-dif-on-U}. Then,
\begin{itemize}
\item[(i)] under conjugation by $\widetilde{\omega}$, the the $a$, $x$ and $y$-parameters for the new differential $\widetilde{\omega}\dif\widetilde{\omega}^{-1}$ are related to that of $\dif$ by
    \[
    \begin{array}{ccccc}
    a_\l^\omega= \oa_{-\l -4}, && x^\omega_{\l-2}= -(\ox_{-\l}+2\oy_{-\l}), && y^{\omega}_{\l-2}=\oy_{-\l}, \\
    \oa_\l^\omega= a_{-\l -4}, && \ox^\omega_{\l-2}= -(x_{-\l}+2y_{-\l}), && \oy^{\omega}_{\l-2}=y_{-\l};
    \end{array}
    \]
\item[(ii)] under conjugation by $\widetilde{\sigma}$, the $a$, $x$ and $y$-parameters for the new differential $\widetilde{\sigma}\dif\widetilde{\sigma}^{-1}$ are related to that of $\dif$ by
    \[
    \begin{array}{ccccc}
    a_\l^\sigma= -a_{-\l -4}, && x^\sigma_{\l-2}= -(\ox_{-\l}+2\oy_{-\l}), && y^{\sigma}_{\l-2}=\oy_{-\l}, \\
    \oa_\l^\sigma= -\oa_{-\l -4}, && \ox^\sigma_{\l-2}= -(x_{-\l}+2y_{-\l}), && \oy^{\sigma}_{\l-2}=y_{-\l};
    \end{array}
    \]
\item[(iii)] under conjugation by $\widetilde{\psi}$, the $a$, $x$ and $y$-parameters for the new differential $\widetilde{\psi}\dif\widetilde{\psi}^{-1}$ are related to that of $\dif$ by
    \[
    \begin{array}{ccccc}
    a_\l^\psi= -a_{\l}, && x^\psi_{\l-2}= \ox_{\l-2}, && y^{\psi}_{\l-2}=\oy_{\l-2}, \\
    \oa_\l^\psi= -\oa_{\l}, && \ox^\psi_{\l-2}= x_{\l-2}, && \oy^{\psi}_{\l-2}=y_{\l-2};
    \end{array}
    \]
\item[(iv)] under conjugation by $\widetilde{\tau}$, the $A$, $X$ and $Y$-parameters for the new differential $\widetilde{\tau}\dif\widetilde{\tau}^{-1}$ are related to that of $\dif$ by
    \[
    \begin{array}{ccccc}
    a_\l^\tau= \oa_{\l}, && x^\tau_{\l-2}= \ox_{\l-2}, && y^{\tau}_{\l-2}=\oy_{\l-2}, \\
    \oa_\l^\tau= a_{\l}, && \ox^\tau_{\l-2}= x_{\l-2}, && \oy^{\tau}_{\l-2}=y_{\l-2}.
    \end{array}
    \]
\end{itemize}
\end{cor}
\begin{proof} This follows from a simple computation on the generators, which we leave as an exercise.
\end{proof}

\section{Fantastic Filtrations}
\label{sec-standardfiltrations}

\subsection{Idempotents and \texorpdfstring{$p$}{p}-DG filtrations}

The basic technique in additive categorification is (tautologically) the direct sum decomposition. The main computational tool is the following trivial lemma.

\begin{lemma} Let $A$ be an algebra, $M$ an $A$-module, and $I$ a finite set. For $1 \le i \le n$ fix $A$-modules $N_i$ and morphisms $v_i \co M \lra N_i$ and $u_i \co N_i \lra M$ which satisfy: \label{lemma-splitting-A-modules}
\begin{subequations}
\begin{eqnarray}
v_i u_i &=& \1_{N_i}, \label{theygive1N} \\
v_i u_j &=& 0 \textrm{ for } i \ne j, \label{theyareorthogonal} \\
\1_M &=& \sum_i u_i v_i. \label{theydecompose} \end{eqnarray}
\end{subequations}
Then $M \cong \oplus_i N_i$ as $A$-modules. \hfill $\square$
\end{lemma}

One should think that \eqref{theydecompose} implies a decomposition of $M$ into the images of the idempotents $u_i v_i \in \End_A(M)$, and that we have isomorphisms
\begin{equation}\label{eqn-iso-Im-N}
\xymatrix{\mim(u_iv_i)  && \lltwocell_{v_i}^{u_i}{'} N_i \ .}
\end{equation}

Let us restate this in a ring-theoretic way. In the proposition below, one should think that $R = \End_A(M \bigoplus (\bigoplus_i N_i))$, and that $\epsilon = \1_M$.

\begin{prop} Let $R$ be a ring, $I$ a finite set, and suppose that for $i \in I$, elements $u_i,v_i \in R$ satisfy:
\begin{subequations} \label{orthogonalidempotents}
\begin{eqnarray} u_i v_i u_i &=& u_i \\
v_i u_i v_i &=& v_i, \\
v_i u_j &=& 0 \ \textrm{ for } i \ne j.
\end{eqnarray}
\end{subequations}
Then $\epsilon = \sum_i u_i v_i$ is an idempotent, and we have a direct sum decomposition $R\epsilon \cong \oplus_i R v_i u_i$. \hfill$\square$
\end{prop}

That is, clearly $R\epsilon \cong \oplus_i R u_i v_i$ (since $\{u_i v_i|i\in I\}$ is a collection of orthogonal idempotents), and we have isomorphisms
\begin{equation}\label{eqn-iso-Rxy-Ryx}
\xymatrix{Ru_iv_i  & \ltwocell_{\cdot u_i}^{\cdot v_i}{'} Rv_iu_i \ ,}
\end{equation}
In this situation, we will say that the data $\{u_i,v_i|i\in I\}$ constitute a \emph{factorization} of idempotents.

Now suppose that $R$ is a $p$-DG algebra. Direct sum decompositions in $R\dmod$ rarely yield direct sum decompositions in $R_\dif\dmod$. After all, if $\varepsilon \in R$ is an idempotent, then $R\varepsilon$ is an $R_\dif\dmod$ summand if and only if $\dif(\varepsilon)=0$. If $\dif(\varepsilon)=0$ then $R\varepsilon$ is cofibrant, and its image in $\mc{D}(R)$ is compact. This is too much to hope for in general. Even when $R\varepsilon$ is $\dif$-closed, which is true if and only if $\dif(\varepsilon)=\dif(\varepsilon)\varepsilon$, it will still be the case that its complement $R(1-\varepsilon)$ is not $\dif$-closed unless $\dif(\varepsilon)=0$. We must somehow accommodate idempotents whose images are not $\dif$-closed.

The following is a mild generalization of the ideas presented in \cite[Section 4.2]{KQ}.

\begin{lemma}\label{lemma-condition-idempotents-to-filtrations} Suppose that $(R,\dif)$ is a $p$-DG algebra, $I$ is a finite set, and $u_i,v_i \in R$ satisfy \eqref{orthogonalidempotents}. Let $\epsilon = \sum_i u_i v_i$. Let $<$ be a total order on $I$. Define an $I$-indexed $R$-module filtration $F^\bullet$ of $R\epsilon$ by $F^{\le i}:= \sum_{k\le i}Ru_kv_k$ and $F^\emptyset := 0$, so that \[F^{\le i}/F^{< i} \cong R v_i u_i\] as $R$ modules. Then the following conditions are equivalent.
\begin{itemize}
\item[(i)] $F^\bullet$ is a filtration by $p$-DG modules, $Rv_i u_i$ is a $p$-DG module, and the subquotient isomorphism is one of $p$-DG modules.
\item[(ii)] The equations
\begin{subequations} \label{conditionsforfiltration}
\begin{equation} v_i\dif(u_i) = 0, \label{condition1forfiltration} \end{equation}
\begin{equation} u_i \dif(v_i) \in F^{<i} \label{condition2forfiltration} \end{equation}
\end{subequations}
hold for all $i\in I$.
\end{itemize}
\end{lemma}

\begin{proof} We show that $(ii)$ implies $(i)$, and leave the reader to prove the converse. We need to show the following for each $i \in I$:
\begin{itemize}
\item[1)] The module $Rv_iu_i$ is $\dif$-closed. After all, \[\dif(v_iu_i)=\dif(v_i)u_i=\dif(v_i)u_iv_iu_i\in Rv_iu_i.\]
\item[2)] Each layer of the filtration $F^{\le i}$ is $\dif$-closed. This can be shown inductively. Assume that $\dif(u_kv_k) \in F^{\le k}$ for $k<i$. Then we need only check that $\dif(u_iv_i) \in F^{\le i}$. We have \[\dif(u_iv_i)=\dif(u_i)v_i+u_i\dif(v_i)=\dif(u_i)v_iu_iv_i+u_i\dif(v_i)\in Ru_iv_i+F^{< i} \subset F^{\le i}.\]
\item[3)] The diagram
\[\xymatrix{F^{\le i}/F^{<i} & \ltwocell_{\cdot u_i}^{\cdot v_i}{'} Rv_iu_i \ ,}\]
is an isomorphism of $p$-DG modules. For any $r \in R$ we have
    \[\dif(ru_iv_i\cdot u_i)-\dif(ru_iv_i)\cdot u_i=ru_iv_i\dif(u_i)=0,\]
    by the assumption that $v_i\dif(u_i)=0$. Therefore right multiplication by $u_i$ intertwines the differentials. On the other hand,
    \[\dif(rv_iu_i\cdot v_i)-\dif(rv_iu_i)\cdot v_i=rv_iu_i\dif(v_i)\equiv 0~(\mathrm{mod}~F^{<i}),\]
    since $u_i\dif(v_i)\in F^{<i}$. Thus right multiplication by $v_i$ also intertwines the differentials.
\end{itemize}
\end{proof}

\begin{rmk}
One should think that the filtration gives a $p$-DG structure on the subquotient $Ru_i v_i$, but this differential is unusual (since $Ru_i v_i$ is not $\dif$-closed). More
precisely, the subquotient differential $\overline{\dif}$ satisfies $\overline{\dif}(r u_i v_i) = \dif(r u_i v_i) u_i v_i$. Then, multiplication by $u_i$ and $v_i$ give isomorphisms between $(R u_i
v_i,\overline{\dif})$ and $(R v_i u_i,\dif)$.

For any idempotent $\varepsilon \in R$, one can place a $p$-differential on $R\varepsilon$ by the formula $\overline{\dif}(r\varepsilon) = \dif(r\varepsilon)\varepsilon$. However, this need not appear as a subquotient in any filtration on $R$.
\end{rmk}

In general, one can say very little about the $p$-DG modules $R v_i u_i$, or about $(R\varepsilon,\overline{\dif})$ for a general idempotent $\varepsilon$. These modules need not be cofibrant, and their images in $\mc{D}(R)$ need not be compact. In order for Lemma \ref{lemma-condition-idempotents-to-filtrations} to have an impact on the Grothendieck group, we will restrict our attention to
even nicer filtrations. Suppose that $R \epsilon$ and $R v_i u_i$ are compact in $\mc{D}(R)$. Then we have the relation in $K_0(R)$
\[ [R\epsilon]=\sum_{i \in I}[Rv_iu_i]. \]
Let $\varepsilon_i = v_i u_i$. The $p$-DG modules $R \epsilon$ and $R\varepsilon_i$, and the relation between them in $K_0(R)$, do not depend the choices made - on the specific order $<$ chosen on $I$
and the specific factorizations $u_i$ and $v_i$ involved - so long as those choices satisfied the conditions of Lemma \ref{lemma-condition-idempotents-to-filtrations}.

\begin{rmk} \label{lessgoodfiltration} If $v_i \dif(u_i) \in R v_i u_i$ then $R v_i u_i$ is $\dif$-closed, even though condition \eqref{condition1forfiltration} may fail. In this
case there is an isomorphism in $R_\dif\dmod$ between $F^{\le i}/F^{<i}$ and $(R v_i u_i,\dif')$, where $\dif'(r v_i u_i) = \dif(r v_i u_i) + r v_i \dif(u_i).$ It is quite
possible that $(R v_i u_i,\dif)$ is compact or cofibrant, but $(R v_i u_i,\dif')$ has very different properties. This is akin to the situation in Chapter
\ref{section-symmetric-functions}, where the $p$-DG module $\Lambda$ is obviously compact and cofibrant, but after altering the differential, the $p$-DG module $\Lambda(a)$ for $a
\ne 0$ is not. After all, $\Lambda(a)$ is acyclic, but is not contractible in $\mc{C}(\Lambda)$.
\end{rmk}

We do not have a general criterion for when a module is compact in the derived category. However, one sufficient condition is for the module to be a $p$-DG summand of $R$, i.e.
the image of an idempotent $e$ satisfying $\dif(e)=0$. Such a summand is a finite cell module, and its compactness follows from Theorem \ref{thm-compact-mod}.

\begin{defn} \label{def-standard-fil}
In the situation of Lemma \ref{lemma-condition-idempotents-to-filtrations}, and assume the conditions in the lemma hold. If furthermore  \[\dif(\epsilon) = 0, \quad \quad \textrm{and} \quad \quad \dif(v_i u_i)=0\]
for all $i \in I$, we say that the factorization data form a \emph{fantastic filtration}, or a \emph{Fc-filtration} for short\footnote{One can also think of a Fc-filtration as particular kind of filtration by (representable) \emph{F}inite \emph{c}ell modules, arising from a factorization. See Definition \ref{def-nice-p-DG-mod-over-category}. It is quite a splendid splitting.}, on the $p$-DG module $R\epsilon$.
\end{defn}

\begin{eg}
Continue the setup of Proposition \ref{lemma-splitting-A-modules}, and assume that $A$ is now a $p$-DG algebra and $M$, $N_i$'s are $p$-DG modules. Let $R = \END_A(M \bigoplus (\bigoplus_i N_i))$ which is naturally a $p$-DG algebra. Then $\epsilon = \1_M$ and $v_i u_i = \1_{N_i}$, so that $\dif(\epsilon) = 0$ and $\dif(v_i u_i)=0$.
\end{eg}

We summarize the discussion in the following corollary.

\begin{cor}\label{cor-sufficient-condition-idemp-to-filtration} Let $R$ be a $p$-DG algebra, $I$ a finite set, and for $i \in I$ let $u_i,v_i \in R$ satisfy
\eqref{orthogonalidempotents}. Assume furthermore that $\dif(\epsilon)=0$ and $\dif(v_iu_i)=0$ for all $i \in I$, where $\epsilon=\sum_i u_iv_i$. Then there is a Fc-filtration built from this data if and only if there is a total order on $I$ such that \[ v_i \dif(u_j) = 0 \textrm{ for } j \ge i.\]
If so, we have the relation
\[ [R\epsilon]=\sum_{i \in I}[Rv_iu_i]\]
in $K_0(R)$.
\end{cor}

\begin{proof} We show the ``if" direction, and leave the converse to the reader. We need to show the conditions of Lemma \ref{lemma-condition-idempotents-to-filtrations}. The
requirement that $v_i\dif(u_i)=0$ follows by assumption. Let us also note that $\dif(v_i u_j) = 0$ when $j > i$ because $v_i u_j = 0$, and $\dif(v_i u_j)=0$ for $j = i$ by
assumption. Therefore $v_i \dif(u_j)=0$ if and only if $\dif(v_i) u_j=0$, for $j \ge i$.

Because $R\epsilon$ is preserved by $\dif$, we know that $u_i \dif(v_i) \in R \epsilon$. Therefore, $u_i \dif(v_i) \in F^{<i}$ if and only if $u_i\dif(v_i)u_jv_j=0$ for
all $j\ge i$. This is implied by the fact that $\dif(v_i)u_j=0$ for $j \ge i$. \end{proof}

In practice, suppose we are given maps $u_i,v_i$ with $\dif(v_iu_i)=0$ and $\dif(\epsilon)=0$, satisfying \eqref{orthogonalidempotents}. We will draw a graph with vertices $I$, and an edge pointing from $i$ to $j$ labeled by $v_i \dif(u_j)$ ( including $i=j$ which is then a loop at $i$ ), omitting any edges which would be labelled by zero. We call this \emph{factorization graph}. If it is acyclic (in particular, has no loops at vertices), then the
conditions of the corollary will hold. Below is an illustration of the graph near two vertices $i,j\in I$.

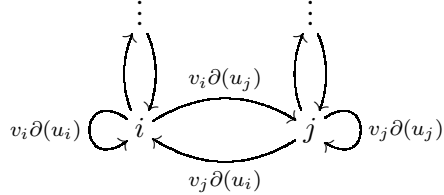
$$
\xymatrix{ & \vdots \ar@/^0.5pc/[d] && \vdots\ar@/^0.5pc/[d] & \\
&i \ar@(ul,dl)[]_{v_i\dif(u_i)} \ar@/^0.5pc/[u] \ar@/^1pc/[rr]^{v_i\dif(u_j)} && j \ar@(ur,dr)[]^{v_j\dif(u_j)} \ar@/^1pc/[ll]^{v_j\dif(u_i)} \ar@/^0.5pc/[u]&
}
$$

This entire section has an analogous version for right $R$-modules and right $p$-DG $R$-modules, which we leave to the reader. Conveniently enough, given the data in Corollary
\ref{cor-sufficient-condition-idemp-to-filtration}, a right Fc-filtration exists if and only if a left Fc-filtration exists (although the order on $I$
will be reversed). The existence of a right Fc-filtration is governed by the graph with vertices $I$ and an edge from $j$ to $i$ labeled by $\dif(v_i) u_j$. Since
$\dif(v_i)u_j = - v_i \dif(u_j)$, this graph is identical to the left module version with arrows reversed and negated. One will be acyclic if and only if the other is.

%
\subsection{Fantastic filtrations on \texorpdfstring{$\mc{EF}$}{EF} and \texorpdfstring{$\mc{FE}$}{FE}}\label{subsec-EF-decomp}
%

There are two important direct sum decompositions used by Lauda which, under certain conditions on the parameters of the differential, will become Fc-filtrations. We
begin with the decomposition coming from \eqref{IdentityDecomp}.

We introduce some notation for \emph{formal curls} analogous to that for bubbles. Again, this notation is independent of the ambient weight.
\begin{equation}\label{curlstuff}
\curl{R}{U}{no}{$k$}{$0$} = - \sum_{a+b=k} \oneline{$b$}{no} ~\bigcwbubble{$a$}{}
\ , \ \ \
\curl{L}{U}{no}{$k$}{$0$} =  \sum_{a+b=k} \bigccwbubble{$a$}{} \oneline{$b$}{no}
\end{equation}
By definition, a curl of negative degree is $0$. The following relations can also be found in \cite{Lau1}.
\begin{equation}
\curl{R}{U}{$\l$}{no}{$r$} = \curl{R}{U}{$\l$}{$_{r-\l}$}{$0$} \ , \ \ \
\curl{L}{U}{$\l$}{no}{$r$} = \curl{L}{U}{$\l$}{$_{r+\l}$}{$0$} \label{curlrelL}
\end{equation}
The case when $r=0$ is precisely the reduction to bubbles relation \eqref{bubblereduction}.

We begin by treating the case $\l > 0$ for \eqref{IdentityDecompPos}. Consider the set of objects
$$\mathbb{X}_\l:=\{\mc{EF}\1_{\l}, \mc{FE}\1_{\l}, \1_{\l}\{1-\l+2c\}|c=0,\dots, \l-1\},$$
and its endomorphism $p$-DG algebra $R:=\END_{\mc{U}}(\mathbb{X}_\l)$ (see Remark \ref{rmk-shrinking-category-to-algebra}).
We will use Lauda's original factorization of idempotents \cite[Theorem 5.10]{Lau1}. Set
\begin{subequations} \label{original-factorization}
\begin{equation}
u_c:=
\cuppy{CW}{$c$}{no}{$\l$}
\quad (0\leq c \leq \l-1), \quad \quad
u_\l:=
\crossingL{no}{no}{no}{$\l$},
\end{equation}
and
\begin{equation}
v_c:= \bottomcurl{CW}{$^{\l-1-c}$}{no}{no}{$\l$}
\quad (0\leq c \leq \l-1), \quad \quad
v_\l:=-~\crossingR{no}{no}{no}{$\l$} .
\end{equation}
\end{subequations}

Lauda shows that $v_s u_s = 1_\l$ is the identity 2-morphism of $\1_\l$, and that $v_s u_t = 0$ for $s \ne t$. Relation \eqref{IdentityDecompPos} shows that $\epsilon =
1_{\mc{EF}} = \sum_{c=0}^\l u_s v_s$. Therefore, $\dif(v_s u_s) = 0$ and $\dif(\epsilon)=0$. To apply Corollary \ref{cor-sufficient-condition-idemp-to-filtration}, it remains to compute the factorization graph described at the end of the previous section, i.e. to compute $v_s \dif(u_t)$ for all $s,t$.

\begin{subequations} \label{eqn-EF-filtration-req}
We begin by computing $v_c \dif(u_c)$. For $0 \le c \le \l-1$ we have
\begin{align*}
\dif(u_c) = \dif \left( \cuppy{CW}{$c$}{no}{$\l$} \right) = (c+\ox_{\l-2}) \cuppy{CW}{$c+1$}{no}{$\l$} + \oy_{\l-2} \cuppy{CW}{$c$}{CW}{$\l$}
\end{align*}
In particular, for $0 \le c \le \l-2$, the first term is a multiple of $u_{c+1}$, and is orthogonal to $v_c$. Therefore
\begin{equation}\label{eqn-EF-filtration-req-1}
v_c\dif(u_c)=\oy_{\l-2} \cwI \textrm{ for } 0 \le c \le \l-2.
\end{equation}
On the other hand, $v_{\l-1}$ is simply a cap, and it is easy to compute that
\begin{equation}\label{eqn-EF-filtration-req-2}
 v_{\l-1}\dif(u_{\l-1})=(\oy_{\l-2}+\l-1+\ox_{\l-2}) \cwI.
\end{equation}
Meanwhile, we have
\begin{align*}
\dif(u_\l) =
\dif \left( \crossingL{no}{no}{no}{$\l$} \right) = (1 - \ox_{\l-2})
\begin{DGCpicture}
\DGCcoupon*(-.3,-.1)(1.3,1.35){}
\DGCstrand(0,0)(1,0)/d/
\DGCdot*<{0,2}
\DGCstrand/d/(0,1.25)(1,1.25)
\DGCdot*<{1.25,2}
\end{DGCpicture}
 -(1+\oa_{\l-2} + 2\oy_\l) \crossingL{yes}{no}{no}{no} \\ +
 (\oa_{\l-2} - 1 + \ox_{\l-2} - \ox_\l) \crossingL{no}{yes}{no}{no} + (\oy_{\l-2} - \oy_\l) \crossingL{no}{no}{yes}{no}
\end{align*}
so that, using \eqref{IdentityDecompNeg} and the vanishing of negative degree curls, we get
\begin{equation}\label{eqn-EF-filtration-req-3}
v_\l\dif(u_\l) = (-1 - \oa_{\l-2} - 2\oy_\l)
\begin{DGCpicture}
\DGCstrand(0,0)(0,1.5)
\DGCdot*<{0}
\DGCdot{.75}
\DGCstrand(.5,0)(.5,1.5)
\DGCdot*>{1.5}
\end{DGCpicture}
+ (\oa_{\l-2} - 1 + \ox_{\l-2} - \ox_\l)
\begin{DGCpicture}
\DGCstrand(0,0)(0,1.5)
\DGCdot*<{0}
\DGCstrand(.5,0)(.5,1.5)
\DGCdot*>{1.5}
\DGCdot{.75}
\end{DGCpicture}
+ (\oy_{\l-2} - \oy_\l)
\cwI
\begin{DGCpicture}
\DGCstrand(0,0)(0,1.5)
\DGCdot*<{0}
\DGCstrand(.5,0)(.5,1.5)
\DGCdot*>{1.5}
\end{DGCpicture}
\ .
\end{equation}
\end{subequations}

In order for the graph to have no loops at any vertex, each of the terms in \eqref{eqn-EF-filtration-req} must vanish. Thus we obtain the following constraints:
\begin{equation}
\begin{array}{lll}
1 + \oa_{\l-2} + 2\oy_\l=0, && \oa_{\l-2} - 1 + \ox_{\l-2} - \ox_\l=0, \\
\oy_{\l-2} = \oy_\l =0, && \ox_{\l-2}=1-\l.
\end{array}
\end{equation}
Solved together under the conditions \eqref{eqn-parameter-req}, we obtain a unique choice of parameters:
\begin{equation} \label{trueparameters}
\begin{array}{lll}
a_{\l-2}=1, && \oa_{\l-2}=-1, \\
x_{\l-2}=1, && \ox_{\l-2}=1-\l,\\
y_{\l-2} = -1, && \oy_{\l-2}=0.
\end{array}
\end{equation}

Let us temporarily ignore this specialization of parameters, and compute the rest of the graph, i.e. $v_t \dif(u_s)$ for $s \ne t$. We have already seen that $\dif(u_\l)$ contains
four terms; the first term factors through $u_0$, and the rest factor through $u_\l$. Therefore, $v_t \dif(u_\l)=0$ for $t \ne 0,\l$, and
\[
v_0 \dif(u_\l) = (\ox_{\l-2} - 1)\bottomcurl{CCW}{$0$}{no}{no}{$\l$}.
\]
We have also already seen that, for $0 \le s \le \l-2$, $\dif(u_s)$ contains two terms; the first factors through $u_{s+1}$ and the second
through $u_s$. Therefore, $v_t \dif(u_s)=0$ for $t \ne s,s+1$, and $v_{s+1} \dif(u_s) = \ox_{\l - 2} + s$. Finally, the first term of $\dif(u_{\l-1})$ has coefficient $\ox_{\l-2}
+ \l - 1$, and consists of a cup with $\l$ dots, which is not orthogonal to any $v_t$.

We encapsulate the data $v_t \dif(u_s)$ for $t \ne s$ in the following graph:
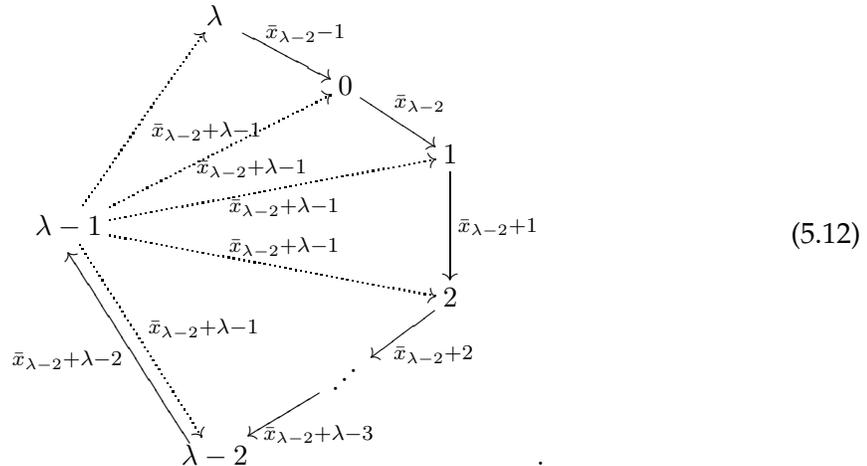
\begin{equation}
\label{fig-break-graph-1}
\begin{gathered}
\xymatrix@R=1.2em{ & {\phantom{2} \l \phantom{2}} \ar[dr]^{\bar{x}_{\l-2}-1} &&\\
&& 0 \ar[dr]^{\bar{x}_{\l-2}} & \\
&&&1 \ar[dd]^{\bar{x}_{\l-2}+1}\\
\l-1 \ar@{.>}[uuur]_-{\bar{x}_{\l-2}+\l -1}\ar@{.>}[uurr]_-{\bar{x}_{\l-2}+\l -1}\ar@{.>}[urrr]_-{\bar{x}_{\l-2}+\l -1}\ar@{.>}[drrr]^-{\bar{x}_{\l-2}+\l -1} \ar@{.>}[dddr]^-{\bar{x}_{\l-2}+\l -1}&&& \\
&&&2\ar[dl]^{\bar{x}_{\l-2}+2}\\
&&\reflectbox{\rotatebox[origin=c]{135}{$\cdots$}}\ar[dl]^{\bar{x}_{\l-2}+\l -3}&\\
&\l-2 \ar@<1ex>[uuul]^{\bar{x}_{\l-2}+\l-2}&&&.
}
\end{gathered}
\end{equation}
Setting $\ox_{\l-2} = 1-\l$, as in \eqref{trueparameters} above, will cause all the dotted arrows to vanish simultaneously, resulting a simple chain
\begin{equation}\label{fig-break-graph-1-specialized}
\xymatrix{
\l \ar[r]^-{-\l} & 0 \ar[r]^-{1-\l} & 1 \ar[r]^-{2-\l} &2\ar[r]^-{3-\l} &\cdots \ar[r]^-{-2}& \l -2 \ar[r]^-{-1} & \l-1.}
\end{equation}

Specializing $\ox_{\l-2} = 2-\l$ will also eliminate the cycles in this graph, but will create loops because $v_s \dif(u_s) \ne 0$. One can only treat this filtration as in Remark \ref{lessgoodfiltration}, rather than being able to use Corollary \ref{cor-sufficient-condition-idemp-to-filtration}.

\begin{defn}\label{def-special-dif}
Let $\dif_1$ be the $2$-categorical derivation on $\UC$ given by the specialization of parameters in \eqref{trueparameters}. It is defined on generators as follows.
\[
\begin{array}{c}
\dif_1 \left( \onelineshort{$1$}{no} \right) =  \onelineshort{$2$}{no} , \quad \quad \quad
\dif_1 \left( \crossing{$0$}{$0$}{$0$}{$0$}{no} \right) =  \twolines{$0$}{$0$}{no} -2 \crossing{$1$}{$0$}{$0$}{$0$}{no} ,\\ \\
\dif_1 \left( \onelineDshort{$1$}{no} \right) =  \onelineDshort{$2$}{no} , \quad \quad \quad
\dif_1 \left( \crossingD{$0$}{$0$}{$0$}{$0$}{no} \right) = - \twolinesD{$0$}{$0$}{no} -2 \crossingD{$0$}{$1$}{$0$}{$0$}{no},
\end{array}
\]
\[
\begin{array}{c}
\dif_1 \left( \cappy{CW}{$0$}{no}{$\l$} \right)  =  \cappy{CW}{$1$}{no}{$\l$}- \cappy{CW}{$0$}{CW}{$\l$},\quad \quad \quad
\dif_1 \left( \cuppy{CW}{$0$}{no}{$\l$} \right) = (1-\l) \cuppy{CW}{$1$}{no}{$\l$},\\ \\
\dif_1 \left( \cuppy{CCW}{$0$}{no}{$\l$} \right) =  \cuppy{CCW}{$1$}{no}{$\l$} + \cuppy{CCW}{$0$}{CW}{$\l$},\quad \quad \quad
\dif_1 \left( \cappy{CCW}{$0$}{no}{$\l$} \right) = (\l+1 ) \cappy{CCW}{$1$}{no}{$\l$}.
\end{array}
\]
\end{defn}

Thus we have proven the upcoming proposition.

\begin{prop} \label{prop-filtration-EF}
Let $\l > 0$. For the differential $\dif_1$ from Definition \ref{def-special-dif}, the data of $\{u_c,v_c\}$ above yield a Fc-filtrations on $\EC \FC\1_\l$ corresponding to \eqref{IdentityDecompPos}. Moreover, $\dif_1$ is the only specialization of the parameters for which these data yield a Fc-filtration. \hfill$\square$
\end{prop}

We have now treated \eqref{IdentityDecompPos} for $\l > 0$. When $\l \le 0$, we need only consider $u_\l$ and $v_\l$ in \eqref{IdentityDecompPos}. We have already computed what
conditions are required for $v_\l \dif(u_\l)=0$, and it is clear that $\dif_1$ suffices to produce a Fc-filtration (with a single term). We leave the reader to check the
analogous statements for \eqref{IdentityDecompNeg}. Thankfully, $\dif_1$ also provides the unique specialization yielding a Fc-filtration (for Lauda's chosen factorization of \eqref{IdentityDecompNeg}).

\vspace{0.06in}

The Fc-filtration we have presented here depends on the choice of factorization of idempotents, which is by no means canonical. It is natural to expect that $\dif_1$ is
not the only specialization which will yield the correct Grothendieck group.

Applying the equivalence $\widetilde{\tau}$ (rotation by 180 degrees) one obtains a different factorization with projections \begin{subequations} \label{dual-factorization}
\begin{equation}
u^\prime_c:=
\cappy{CW}{$c$}{no}{$\l$}
\quad (0\leq c \leq \l-1), \quad \quad
u^\prime_\l:=
\crossingR{no}{no}{no}{$\l$},
\end{equation}
and inclusions
\begin{equation}
v^\prime_c:=
\begin{DGCpicture}
\DGCstrand(0,2)/d/(1,1)/d/(0,1)/u/(1,2)/u/
\DGCdot*<{0.8,2}
\DGCcoupon*(0.1,0.6)(0.9,1.4){$_{_{\l-1-c}}$}
\DGCcoupon*(1.1,1)(1.3,1.6){$\l$}
\end{DGCpicture}
\quad (0\leq c \leq \l-1), \quad \quad
v^\prime_\l:=-~\crossingL{no}{no}{no}{$\l$}.
\end{equation}
\end{subequations}
The same exact calculation
will suffice for these idempotents (computing using right modules rather than left modules), except that the parameters should be substituted as in Corollary \ref{cor-symmetries-intertwine-dif}. This yields the dual specialization \begin{equation} \label{trueparametersdual} \begin{array}{lll} a_{\l-2}=-1, &&
\oa_{\l-2}=1, \\ x_{\l-2}=1-\l, && \ox_{\l-2}=1,\\ y_{\l-2} = 0, && \oy_{\l-2}=-1. \end{array} \end{equation} We write this specialization of parameters as $\dif_{-1}$.

\begin{defn}\label{def-special-dif-dual}
Let $\dif_{-1}$ be the $2$-categorical derivation on $\UC$ given by the specialization of parameters in \eqref{trueparametersdual}. It is defined on generators as follows.
\[
\begin{array}{c}
\dif_{-1} \left( \onelineshort{$1$}{no} \right) =  \onelineshort{$2$}{no} , \quad \quad \quad
\dif_{-1} \left( \crossing{$0$}{$0$}{$0$}{$0$}{no} \right) =  - \twolines{$0$}{$0$}{no} - 2 \crossing{$0$}{$1$}{$0$}{$0$}{no} ,\\ \\
\dif_{-1} \left( \onelineDshort{$1$}{no} \right) =  \onelineDshort{$2$}{no} , \quad \quad \quad
\dif_{-1} \left( \crossingD{$0$}{$0$}{$0$}{$0$}{no} \right) = \twolinesD{$0$}{$0$}{no} -2 \crossingD{$1$}{$0$}{$0$}{$0$}{no},
\end{array}
\]
\[
\begin{array}{c}
\dif_{-1} \left( \cappy{CW}{$0$}{no}{$\l$} \right)  =  (1-\l) \cappy{CW}{$1$}{no}{$\l$},\quad \quad \quad
\dif_{-1} \left( \cuppy{CW}{$0$}{no}{$\l$} \right) = \cuppy{CW}{$1$}{no}{$\l$} - \cuppy{CW}{$0$}{CW}{$\l$},\\ \\
\dif_{-1} \left( \cuppy{CCW}{$0$}{no}{$\l$} \right) =  (\l + 1) \cuppy{CCW}{$1$}{no}{$\l$},\quad \quad \quad
\dif_{-1} \left( \cappy{CCW}{$0$}{no}{$\l$} \right) = \cappy{CCW}{$1$}{no}{$\l$} + \cappy{CCW}{$0$}{CW}{$\l$}.
\end{array}
\]
\end{defn}

\begin{prop} \label{prop-filtration-EF-dual} Let $\l > 0$. For the differential $\dif_{-1}$ from Definition \ref{def-special-dif-dual}, the data of $\{u_c', v_c'\}$ above yield a
Fc-filtration on $\EC \FC\1_\l$ corresponding to \eqref{IdentityDecompPos}. Moreover, $\dif_{-1}$ is the only specialization of the parameters for which these data
yield a Fc-filtration. \hfill$\square$ \end{prop}

\begin{rmk} We have now seen two possibilities for the data of a differential and a factorization on $\EC \FC \1_\l$ which give rise to a Fc-filtration. In fact, we have a very strong uniqueness result. It turns out that Lauda's prescient choice of factorization in \cite{Lau1} is not just Lauda-given, but God-given.

Outside of characteristic $2$ there is a ``canonical" choice of a rotation-invariant factorization, though it has not yet appeared
in the literature. In this factorization, $u_\l$ and $v_\l$ are as before. For $0 \le c \le \l-1$ we let
\[u_c = \sum_{j=0}^{\l-1} \rho_j \cuppy{CW}{$c-j$}{no}{no}
\ccwbubble{$j$}{} \]
where for $c \le \frac{\l-1}{2}$ we have
\[
\begin{cases} \rho_0=1 & \\ \rho_k = 0 & k \ne 0 \end{cases}
\]
and for $c > \frac{\l-1}{2}$ we have
\[
\begin{cases} \rho_k = 1 & k < 2c-\l+1 \\ \rho_k = \frac{1}{2} & k = 2c - \l + 1 \\ \rho_k = 0 & k > 2c-\l + 1 \end{cases}
\]
and $v_c$ is the 180 degree rotation of $u_c$. (For
any given $\l$ there may be other rotation-invariant factorizations, but this is the unique formula which stabilizes as $\l$ goes to $\infty$.) Though this may seem at first glance as a more preferable factorization because of the extra symmetry, it will be clear from the next proposition that this is not the case.
\end{rmk}

\begin{prop} \label{prop-only-way-to-do-it} Consider either $\UC_{\even}$ or $\UC_{\odd}$. Suppose that $\dif$ is a differential as in Definition \ref{def-dif-on-U}, and $\{u_c,
v_c\}$ is a factorization of $\EC \FC \1_\l$ which yields a Fc-filtration. Then either $\dif=\dif_1$ and $\{u_c,v_c\}$ are as in \eqref{original-factorization}, or
$\dif=\dif_{-1}$ and $\{u_c,v_c\}$ are as in \eqref{dual-factorization}. \end{prop}

\begin{proof} In any factorization, the maps $u_\l$ and $v_\l$ are canonical (up to scalar), because they are the only maps of the appropriate degree. The composition $v_\l
\dif(u_\l)$ was computed in \eqref{eqn-EF-filtration-req-3}, and must be zero in a Fc-filtration. Combining this with \eqref{eqn-parameter-req} we obtain the
following conditions: \begin{itemize} \item[(1)] The parameters $a_\l$, $\oa_\l$, $y_\l$ and $\oy_\l$ do not depend on the ambient weight (within either $\UC_{\odd}$ or
$\UC_{\even}$). \item[(2)] They satisfy $a = - \oa = 1 + 2 \oy = -1 - 2 y$. \item[(3)] The parameters $x_\l$ and $\ox_\l$ form an arithmetic series, with $\ox_\l - \ox_{\l-2} =
a-1$. Also, $\ox_\l + x_\l = -\l$. \end{itemize}

In particular, if one can calculate the parameters $\oa_\mu$, $\ox_\mu$ and $\oy_\mu$ for a single value of $\mu$, then one determines all parameters for all $\l$ with $\l-\mu$
even. We now prove that $\dif = \dif_{\pm 1}$, by computing with small values of $\l$.

When $\l > 0$ the map $u_0$ and the map $v_{\l-1}$ are also canonical (up to scalar), for degree reasons. When $\l=1$, the entire factorization is canonical:
\begin{subequations}
\begin{equation}
u_0=
\cuppy{CW}{$0$}{no}{$1$}
\quad  \quad \quad
u_1=
\crossingL{no}{no}{no}{$1$},
\end{equation}
\begin{equation}
v_0= \cappy{CCW}{$0$}{no}{$1$} \quad \quad \quad
v_1=-~\crossingR{no}{no}{no}{$1$} .
\end{equation}
\end{subequations}
We have already computed $v_0 \dif(u_0)$ in \eqref{eqn-EF-filtration-req-2}, which will be zero precisely when $\oy_{-1} = - \ox_{-1}$. Meanwhile, a quick computation shows that
\begin{equation}
v_0\dif(u_1)=(\ox_{-1}-1)\cappy{CW}{$0$}{no}{$1$}, \ \ \ \
v_1\dif(u_0)=-\ox_{-1}\cuppy{CCW}{$0$}{no}{$1$}.
\end{equation}
Therefore, the graph for this filtration will be acyclic when either $\ox_{-1}=0$ or $\ox_{-1}=1$. These two choices correspond precisely to $\dif_1$ and $\dif_{-1}$. Thus we have shown that the differential on $\UC_{\odd}$ is equal to $\dif_{\pm 1}$.

For $\l=2$ there is a one-parameter family of factorizations. For any $t\in \Bbbk$, let
\begin{subequations}
\begin{equation}
u_0(t)=\cuppy{CW}{$0$}{no}{$2$}, \quad \quad \quad u_1(t)=\cuppy{CW}{$1$}{no}{$2$}+t\cuppy{CW}{$0$}{CCW}{$2$},
\end{equation}
\begin{equation}
v_0(t)=\cappy{CW}{$1$}{no}{$2$}+(1+t)\cappy{CW}{$0$}{CCW}{$2$}, \quad \quad \quad v_1(t)=\cappy{CW}{$0$}{no}{$2$},
\end{equation}
\end{subequations}
and let $u_2(t)$ and $v_2(t)$ be sideways crossings as before. We have normalized our choice of projections and inclusions such that the \emph{leading coefficient}, i.e.~the coefficient of the unique term without bubbles, is $1$.

One can easily check, using leading coefficients, that $v_{s+1} \dif(u_s) = \ox_{0} + s$ for $0 \le s \le 1$, and $v_0 \dif(u_2) = \ox_0 - 1$. In order for this cycle to be
broken, one must specialize $\ox_0$ to an element of $\{-1,0,1\}$. We leave the reader to check that $v_0 \dif(u_0)$ has coefficient $\oy_0 + t \ox_0$, and $v_1 \dif(u_1)$ has
coefficient $\ox_0 -t \ox_0 + \oy_0 + 1$; these coefficients must vanish in a Fc-filtration. The only possibilities are $t=0$ and $\ox_0=-1$, or $t=1$ and $\ox_0 =
1$. Once again, these choices correspond precisely to $\dif_1$ (with the factorization of \eqref{original-factorization}) and $\dif_{-1}$ (with the factorization of
\eqref{dual-factorization}). Thus the differential on $\UC_{\even}$ is equal to $\dif_{\pm 1}$.

To show that the factorizations of \eqref{original-factorization} and \eqref{dual-factorization} are the only ones possible, one could generalize the computations above to
arbitrary $\l>0$. We normalize our factorizations so that the leading coefficient of any $u_c$ or $v_c$ is $1$. The number of parameters required to define the complete family of
factorizations gets quite large, so it is impractical to write out the factorization and compute the entire graph at once. Let us outline the calculation when $\dif=\dif_1$, which
we leave as an exercise to the reader. The same style of argument will work for $\dif = \dif_{-1}$ as well.

Using a calculation with leading coefficients, one can check that $v_{s+1} \dif(u_s) = \ox_{\l-2} + s$ for $0 \le s \le \l-1$, and $v_0 \dif(u_{\l}) = \ox_{\l-2} - 1$. Fix $\dif =
\dif_1$. Then the simple chain of \eqref{fig-break-graph-1-specialized} is certainly a subgraph of our factorization graph. Therefore one requires that $v_t \dif(u_s) = 0$ for $0
\le t \le s \le \l-1$, in order for the graph to be acyclic. We know that $u_0$ agrees with \eqref{original-factorization}, being canonical. One need only compute $v_0
\dif(u_0)=0$ to deduce that $u_1$ agrees with \eqref{original-factorization}. Similarly, once we know that $u_k$ agrees with \eqref{original-factorization} for all $k \le s$, one
need only compute $v_t \dif(u_s)=0$ for $t \le s$ to deduce that $u_{s+1}$ agrees as well. \end{proof}

\begin{rmk} The reader who attempts the exercise above may find this technique useful. Let the parameters in this family of factorizations be the coefficients in $u_c$ for
various $0 \le c \le \l-1$; these will determine the coefficients of $v_c$. However, one should not attempt to compute the coefficients in $v_c$. Instead, to compute $v_t
\dif(u_s)$, one should write $\dif(u_s)$ as a linear combination of various $u_t$ for $t \le s+1$, and use orthogonality. Beginning with $t=s+1$ and moving downwards, one can eliminate one parameter at a time. \end{rmk}

%
\subsection{Divided power modules} \label{dividedpowers}
%

By Remark \ref{rmk-morphisms-in-u}, we have $\END_{\mc{U}}(\1_{\l}\mc{E}^n\1_{\lambda-2n})\cong \NH_n\o \Lambda$. Since any $2$-categorical differential preserves local relations of $\UC$, this map is an isomorphism of $p$-DG algebras. We have already studied the $p$-DG algebra $\Lambda$ in Chapter \ref{section-symmetric-functions}, and thus we need only examine the $p$-DG structure on the nilHecke algebra to understand the derived category of the endomorphism ring. This was done in \cite[Chapter 3]{KQ}, and we summarize the results here.

The specialization \eqref{trueparameters} of the previous section implies that the differential is ``$a$-local'', i.~e.~$a_\l = a = 1$ is independent of $\l$. Let us assume until
\textbf{boldly stated otherwise} that the differential is local with $a=1$. The local differential with $a = -1$ (corresponding to \eqref{trueparametersdual}) behaves in a similar
way, and everything stated here will hold with minor modifications.

With the assumption $a=1$, the differential acts on the local generators of the nilHecke algebra via
\begin{equation}\label{eqn-dif-on-gen-of-NH}
\dif\left(~
\begin{DGCpicture}
\DGCstrand(0,0)(0,1)
\DGCdot*>{1}
\DGCdot{0.5}
\end{DGCpicture}
~\right)=
\begin{DGCpicture}
\DGCstrand(0,0)(0,1)
\DGCdot*>{1}
\DGCdot{0.5}[r]{$^2$}
\end{DGCpicture},
\ \ \ \
\dif\left(~
\begin{DGCpicture}
\DGCstrand(0,0)(1,1)
\DGCdot*>{1}
\DGCstrand(1,0)(0,1)
\DGCdot*>{1}
\end{DGCpicture}
~\right)=
\begin{DGCpicture}
\DGCstrand(0,0)(0,1)
\DGCdot*>{1}
\DGCstrand(1,0)(1,1)
\DGCdot*>{1}
\end{DGCpicture}
-2
\begin{DGCpicture}
\DGCstrand(0,0)(1,1)
\DGCdot*>{1}
\DGCstrand(1,0)(0,1)
\DGCdot*>{1}
\DGCdot{0.7}
\end{DGCpicture}.
\end{equation}
We review the explanation in \cite{KQ} for why the quantum divided power relation
\begin{equation}\label{eqn-divided-power}
E^n=[n]!E^{(n)}
\end{equation}
is categorified using this $p$-DG algebra structure on $\NH_n$.

The obvious inclusion of $p$-DG algebras $\sym_n\subset \pol_n$ allows us to regard any $p$-DG module over $\pol_n$ as, via restriction, a module over $\sym_n$. In particular, the module\footnote{This module was called $\mc{P}_n^+$ in \cite{KQ}.}
\begin{equation}
\mc{E}^{(n)}_+:=\pol_n\cdot v_+, \ \ \ \ \dif(v_+)=-\sum_{i=0}^n(n-i)x_iv_+,
\end{equation}
is a $p$-DG module over $\sym_n$, graded so that $\mathrm{deg}(v_+)=\frac{(1-n)n}{2}$. It has a $\sym_n$-basis which spans a $\dif$-stable $\Bbbk$-vector space $U_n^+$:
\begin{equation}
U_n^+:=\Bbbk\langle x_1^{a_1}\cdots x_n^{a_n}v_{+}| 0 \leq a_i \leq n-i,~i=1,\dots, n\rangle,
\end{equation}
so that $\mc{E}^{(n)}_+\cong \sym_n\o U_n^+.$ It follows that $\mc{E}^{(n)}_+$ is a finite cell module over $\sym_n$. Therefore, the $p$-DG structure on its endomorphism ring over $\sym_n$ will correctly describe its endomorphism ring in the derived category (see equation \eqref{eqn-mor-space-homotopy-category}). We have
\begin{equation}
\NH_n\cong \END_{\sym_n}(\mc{E}^{(n)}_+),
\end{equation}
as $p$-DG algebras, i.e.~intertwining the induced differential on endomorphisms with the differential in \eqref{eqn-dif-on-gen-of-NH} (see equations (65) and (66) of \cite{KQ}).

The upshot of the above discussion is that the left regular representation of $\NH_n$ is isomorphic to
\begin{equation}\label{eqn-decomp-of-left-regular-rep}
\NH_n\cong \sym_n \o U_n^+\o (U_n^+)^* \cong \mc{E}^{(n)}_+ \o (U_n^+)^*,
\end{equation} where $(U_n^+)^*$ stands for the dual $p$-complex of $U_n^+$. Starting from this, one can show that the facts below hold. We regard equation \eqref{eqn-decomp-of-left-regular-rep} as a categorical interpretation of the relation \eqref{eqn-divided-power}.
\begin{itemize}
\item[(1)] By introducing a $\dif$-stable filtration on $(U_n^+)$, the left regular representation $\NH_n$ inherits a filtration whose subquotients are isomorphic, as $p$-DG modules, to grading-shifted copies of $\mc{E}^{(n)}_+$.
\item[(2)] The complex $U_n^+$ is acyclic whenever $n\geq p$, so that $\mc{E}_+^{(n)}$, and hence $\NH_n$, is acyclic whenever $n\geq p$. Moreover, $\mc{D}(\NH_n)\cong 0$ for $n \ge p$. See \cite[Proposition 3.15]{KQ}.
\item[(3)] The module $\mc{E}^{(n)}_+$ is cofibrant over $\NH_n$ and compact in the derived category so long as $0\leq n \leq p-1$. See \cite[Proposition 3.26]{KQ}.
\item[(4)] When $0\leq n \leq p-1$, any indecomposable compact module in the derived category of $\NH_n$ is isomorphic to a module of the form $\mc{E}^{(n)}_+\o V$, where $V$ is some non-contractible indecomposable $p$-complex. See \cite[Corollary 3.27]{KQ}.
\end{itemize}
Because of our choice of grading, the symbol of $U_n^+$ and its dual in $K_0(\Bbbk)=\mathbb{O}_p$ satisfy
\begin{equation}
{[U_n^+]}={[(U_n^+)^*]}=[n]!.
\end{equation}
Thus in the Grothendieck group $K_0(\NH_n)$, we have
\begin{equation} \label{desiredgrotheqn}
{[\NH_n]}=[\mc{E}^{(n)}_+\o (U_n^+)^*]=[n]!{[\mc{E}^{(n)}_+]}
\end{equation}
where both sides are zero for $n \ge p$.

Likewise, one has the left $p$-DG module $\mc{E}^{(n)}_{-}$, which can be identified with the $p$-DG $\NH_n$-module with a conjugate action
$$\mc{E}^{(n)}_{-}:=\left(\mc{E}^{(n)}_+\right)^{{\widetilde{\omega}}}.$$
Since $\widetilde{\omega}$ is an automorphism of $\mc{U}$, $\mc{E}^{(n)}_{-}$ shares all the above properties of $\mc{E}^{(n)}_{+}$

\begin{defn} \label{def-divided-power-mod}Fix any $\lambda\in \Z$, and let ${}_\l \UC$ denote the direct sum of all ${}_\l \UC_\mu$ for $\mu \in \Z$, and similarly for $\UC_\l$. Fix $n\in \N$.
\begin{itemize}
\item[(i)] The left $p$-DG module $\1_\l\mc{E}^{(n)}$ over $(\UC,\dif_1)$ is the induced module
\[
\1_\l\mc{E}^{(n)}:=\Ind_{\NH_n}^{_\l \mc{U}}(\mc{E}^{(n)}_+),
\]
where the induction comes from the composition of inclusions
$$\NH_n \lra \NH_n\o \Lambda\cong \END_{\1_\l\mc{U}}(\1_{\l}\mc{E}^{n}) \lra {_{\l}\mc{U}_{\l-2n}}.$$
\item[(ii)] The left $p$-DG module $\mc{F}^{(n)}\1_{\l}$ is the induced module
\[
\mc{F}^{(n)}\1_\l:=\Ind_{\NH_n}^{\mc{U}_\l}(\mc{E}^{(n)}_{-}),
\]
where the induction comes from the composition of inclusions
$$\NH_n \lra \NH_n\o \Lambda\cong \END_{\mc{U}_{\l}}(\1_{\l-2n}\mc{F}^{n}\1_\l) \lra {_{\l-2n}\mc{U}_{\l}}.$$
\end{itemize}
Both modules in the definition will be referred to as the \emph{divided power modules}.
\end{defn}

\begin{cor}\label{cor-compactness-of-divided-power-mod} Fix a weight $\l \in \Z$, and let $n\in \N$.
\begin{itemize}
\item[(i)] The representable module $\1_\l\mc{E}^{n}$ (resp.$\mc{F}^{n}\1_\l$) admits an $n!$-step filtration whose subquotients are isomorphic to grading shifts of the divided power module $\1_\l\mc{E}^{(n)}$ (resp. $\mc{F}^{(n)}\1_\l$).
\item[(ii)]The divided power modules are acyclic whenever $n\geq p$.
\item[(iii)]The $p$-DG module $\1_\l\mc{E}^{(n)}$ (resp. $\mc{F}^{(n)}\1_{\l}$ ) is cofibrant over the $p$-DG category $({_{\l}\UC},\dif_1)$ (resp. $({\UC_\l},\dif_1)$ whenever $0\leq n\leq p-1$, and its image in the derived category $\mc{D}({_\l\UC},\dif_1)$ (resp. $\mc{D}({\UC_\l},\dif_1)$ ) is compact.

\end{itemize}
\end{cor}
\begin{proof}Follows from the corresponding properties for $\mc{E}^{(n)}_+$ (resp. $\mc{E}^{(n)}_{-}$).
\end{proof}

\paragraph{Locality of the differential.}
Now we revoke the assumption that $a_\l$ is necessarily independent of $\l$, or equal to $\pm 1$. That is, we equip $\NH_n$ with some other differential $\widetilde{\dif}$ which
disagrees with both $\dif_1$ and $\dif_{-1}$. We expect $K_0(\NH_n,\widetilde{\dif}) \ne K_0(\NH_n,\dif_{\pm 1})$. This would give another reason to assume that the differential for the $a$-parameters on $\UC$ should be independent of regions.

When $n=2$, the $p$-DG structure on $\NH_2$ depends on a single parameter $a$. It is shown in \cite{KQ} that the Grothendieck group is in error unless $a = \pm 1$. So let us
assume that $a_\l = \pm 1$ for all $\l$. In order to show that $a_\l$ is independent of $\l$ (with a given parity) we need only show that $a_\l = 1$ and $a_{\l-2}=-1$ gives the
incorrect $p$-DG Grothendieck group for $\NH_3$.

Here we give a simple example of what may go wrong in characteristic $p=3$. Until the end of this section we omit the orientations on strands, with the understanding that all
strands point up.

Suppose the algebra $\NH_3$ has the non-local differential $\widetilde{\dif}$ defined by
\begin{subequations}
\begin{equation}\label{eqn-non-local-dif-1}
\widetilde{\dif}\left(~
\begin{DGCpicture}
\DGCstrand(0,0)(0.5,1)
\DGCstrand(0.5,0)(0,1)
\DGCstrand(1,0)(1,1)
\end{DGCpicture}
~\right)=
\begin{DGCpicture}
\DGCstrand(0,0)(0.5,1)
\DGCstrand(0.5,0)(0,1)
\DGCdot{0.75}
\DGCstrand(1,0)(1,1)
\end{DGCpicture}
+
\begin{DGCpicture}
\DGCstrand(0,0)(0,1)
\DGCstrand(0.5,0)(0.5,1)
\DGCstrand(1,0)(1,1)
\end{DGCpicture}
\ , \ \ \ \
\widetilde{\dif}~\left(
\begin{DGCpicture}
\DGCstrand(0,0)(0,1)
\DGCstrand(0.5,0)(1,1)
\DGCstrand(1,0)(0.5,1)
\end{DGCpicture}
~\right)=
\begin{DGCpicture}
\DGCstrand(0,0)(0,1)
\DGCstrand(0.5,0)(1,1)
\DGCdot{0.75}
\DGCstrand(1,0)(0.5,1)
\end{DGCpicture}
-
\begin{DGCpicture}
\DGCstrand(0,0)(0,1)
\DGCstrand(0.5,0)(0.5,1)
\DGCstrand(1,0)(1,1)
\end{DGCpicture}
\ .
\end{equation}

As a module over $\pol_3$, $\NH_3$ has a basis
\[
\left\{~
\begin{DGCpicture}
\DGCstrand(0,0)(0,1)
\DGCstrand(0.5,0)(0.5,1)
\DGCstrand(1,0)(1,1)
\end{DGCpicture}\ , \ \
\begin{DGCpicture}
\DGCstrand(0,0)(0.5,1)
\DGCstrand(0.5,0)(0,1)
\DGCstrand(1,0)(1,1)
\end{DGCpicture}\ , \ \
\begin{DGCpicture}
\DGCstrand(0,0)(0,1)
\DGCstrand(0.5,0)(1,1)
\DGCstrand(1,0)(0.5,1)
\end{DGCpicture}\ , \ \
\begin{DGCpicture}
\DGCstrand(0,0)(0.5,1)
\DGCstrand(0.5,0)(1,1)
\DGCstrand(1,0)(0,1)
\end{DGCpicture}\ , \ \
\begin{DGCpicture}
\DGCstrand(0,0)(1,1)
\DGCstrand(0.5,0)(0,1)
\DGCstrand(1,0)(0.5,1)
\end{DGCpicture}\ , \ \
\begin{DGCpicture}[scale=0.5]
\DGCstrand(0,0)(2,2)
\DGCstrand(1,0)(0,1)(1,2)
\DGCstrand(2,0)(0,2)
\end{DGCpicture}
~\right\}.
\]
Since
\begin{equation}\label{eqn-non-local-dif-2}
\widetilde{\dif}\left(~
\begin{DGCpicture}
\DGCstrand(0,0)(0,1)
\DGCstrand(0.5,0)(0.5,1)
\DGCstrand(1,0)(1,1)
\end{DGCpicture}
~\right) = 0,
\end{equation}
the $\pol_3$-module spanned by the the identity forms a $p$-DG subalgebra of $\NH_3$, which we still call $\pol_3$. We compute $\widetilde{\dif}$ for the rest of basis elements above.
\begin{equation}\label{eqn-non-local-dif-3}
\widetilde{\dif}\left(~
\begin{DGCpicture}
\DGCstrand(0,0)(0.5,1)
\DGCstrand(0.5,0)(1,1)
\DGCstrand(1,0)(0,1)
\end{DGCpicture}
~\right)  =
\begin{DGCpicture}
\DGCstrand(0,0)(0.5,1)
\DGCstrand(0.5,0)(1,1)
\DGCstrand(1,0)(0,1)
\DGCdot{0.85}
\end{DGCpicture}
+
\begin{DGCpicture}
\DGCstrand(0,0)(0.5,1)
\DGCstrand(0.5,0)(1,1)
\DGCdot{0.85}
\DGCstrand(1,0)(0,1)
\end{DGCpicture}
+
\begin{DGCpicture}
\DGCstrand(0,0)(0,1)
\DGCstrand(0.5,0)(1,1)
\DGCstrand(1,0)(0.5,1)
\end{DGCpicture}
-
\begin{DGCpicture}
\DGCstrand(0,0)(0.5,1)
\DGCstrand(0.5,0)(0,1)
\DGCstrand(1,0)(1,1)
\end{DGCpicture}
\ ,
\end{equation}
\begin{equation}\label{eqn-non-local-dif-4}
\widetilde{\dif}\left(~
\begin{DGCpicture}
\DGCstrand(0,0)(1,1)
\DGCstrand(0.5,0)(0,1)
\DGCstrand(1,0)(0.5,1)
\end{DGCpicture}
~\right)  =
\begin{DGCpicture}
\DGCstrand(0,0)(1,1)
\DGCstrand(0.5,0)(0,1)
\DGCdot{0.85}
\DGCstrand(1,0)(0.5,1)
\end{DGCpicture}
+
\begin{DGCpicture}
\DGCstrand(0,0)(1,1)
\DGCstrand(0.5,0)(0,1)
\DGCstrand(1,0)(0.5,1)
\DGCdot{0.85}
\end{DGCpicture}
+
\begin{DGCpicture}
\DGCstrand(0,0)(0,1)
\DGCstrand(0.5,0)(1,1)
\DGCstrand(1,0)(0.5,1)
\end{DGCpicture}
-
\begin{DGCpicture}
\DGCstrand(0,0)(0.5,1)
\DGCstrand(0.5,0)(0,1)
\DGCstrand(1,0)(1,1)
\end{DGCpicture}
\ ,
\end{equation}
\begin{equation}\label{eqn-non-local-dif-5}
\widetilde{\dif}\left(~
\begin{DGCpicture}[scale=0.5]
\DGCstrand(0,0)(2,2)
\DGCstrand(1,0)(0,1)(1,2)
\DGCstrand(2,0)(0,2)
\end{DGCpicture}
~\right)
 =
\begin{DGCpicture}[scale=0.5]
\DGCstrand(0,0)(2,2)
\DGCstrand(1,0)(0,1)(1,2)
\DGCstrand(2,0)(0,2)
\DGCdot{1.7}
\end{DGCpicture}
+
\begin{DGCpicture}[scale=0.5]
\DGCstrand(0,0)(2,2)
\DGCstrand(1,0)(0,1)(1,2)
\DGCdot{1.7}
\DGCstrand(2,0)(0,2)
\end{DGCpicture}
+
\begin{DGCpicture}[scale=0.5]
\DGCstrand(0,0)(2,2)
\DGCdot{1.7}
\DGCstrand(1,0)(0,1)(1,2)
\DGCstrand(2,0)(0,2)
\end{DGCpicture}
-
\begin{DGCpicture}[scale=0.5]
\DGCstrand(0,0)(2,2)
\DGCstrand(1,0)(0,2)
\DGCstrand(2,0)(1,2)
\end{DGCpicture}
-
\begin{DGCpicture}[scale=0.5]
\DGCstrand(0,0)(1,2)
\DGCstrand(1,0)(2,2)
\DGCstrand(2,0)(0,2)
\end{DGCpicture}
\ .
\end{equation}
\end{subequations}
It follows from equations \eqref{eqn-non-local-dif-1} through \eqref{eqn-non-local-dif-5} that the $\pol_3$-splitting of $\NH_3$ (in terms of the basis above) is actually a filtration in the category of $p$-DG modules for $\pol_3$. The associated graded $p$-DG modules are contractible, with the exception of $\pol_3$ itself. It follows that the inclusions $\Bbbk \lra \pol_3\lra \NH_3$ are quasi-isomorphisms of $p$-DG algebras. From Corollary \ref{cor-qis-algebra-equivalence-derived-categories}, one finds that $K_0(\NH_3, \widetilde{\dif}) \cong \mathbb{O}_3$, and the symbol $[\NH_3]$ is a generator of the Grothendieck group as an $\mathbb{O}_3$-module.

This clearly disagrees with the desired relation $[\NH_3] = 0$, which is ultimately needed to categorify the equality $E^3 = 0$ in $u^+_q(\mf{sl}_2)$ at a sixth root of unity.

\section{Decategorification} \label{sec-decategorfication}

\subsection{An idempotented form of the small quantum \texorpdfstring{$\mathfrak{sl}(2)$}{sl(2)}}
Following \cite[Chapter 36]{Lus4} we define an idempotented form of the small quantum $\mf{sl}_2$ over the ring $\mathbb{O}_p$, and some of its finite dimensional representations. We first recall the generic version due to Beilinson-Lusztig-MacPherson \cite{BLM}.

\begin{defn}\label{def-big-sl2}
\begin{itemize}
\item[(I)] The quantum algebra $U_q(\mathfrak{sl}_2)$ is a unital associative algebra over $\Z[q^{\pm}]$ generated by $ E, F, K, K^{-1} $ and subject to the relations:
\begin{enumerate}
\item[(i)] $ KK^{-1} = 1 = K^{-1}K$,
\item[(ii)] $ KE = q^{2} E K $, \quad $ KF = q^{-2} F K $,
\item[(iii)] $ EF-FE = \frac{K-K^{-1}}{q-q^{-1}}$.
\end{enumerate}
\item[(II)] The non-unital associative quantum algebra $\dot{U}_q(\mathfrak{sl}_2)$ is obtained from $U_q(\mathfrak{sl}_2)$ by removing $K, K^{-1}$ and adjoining a family of orthogonal idempotents $\{1_\l|\l\in \Z\}$, such that
\begin{enumerate}
\item[(i)] $ 1_\l \cdot 1_{\mu}=\delta_{\l,\mu}1_{\l}$ for any $\l, \mu \in \Z$,
\item[(ii)] $ E 1_{\l} = 1_{\l+2} E $, \quad $ F 1_{\l} = 1_{\l-2} F $,
\item[(iii)] $ EF 1_{\l}-FE1_{\l} =[\l]1_{\l}$,
\end{enumerate}
where in the last equation, $[\l]=\frac{q^{\l}-q^{-\l}}{q-q^{-1}}=\sum_{i=0}^{\l-1} q^{\l -1-2i}$.
\end{itemize}
\end{defn}
Both forms of the quantum algebra above are known as the \emph{generic} quantum $\mf{sl}_2$, as opposed to the quantum group at a root of unity to be discussed below. The second form, abbreviated as $\dot{U}$, is the most convenient for the purpose of categorification (see \cite{Lau1}). It has a collection of $\Z[q^{\pm}]$-integral algebra generators formed by the \emph{divided power elements}
\begin{equation}\label{eqn-div-element}
E^{(n)}1_\l:= \frac{E^n1_\l}{[n]!}, \quad \quad F^{(n)}1_\l:= \frac{F^n1_\l}{[n]!},
\end{equation}
where $\l\in \Z$ and $n\in \N$. Lusztig's \emph{canonical basis} $\dot{\mathbb{B}}$ is an additive $\Z[q^{\pm}]$-basis for $\dot{U}$, which consists of
\begin{itemize}
\item $E^{(a)}F^{(b)}1_\l$, where $a,b \in \N$, $\l \in \Z$ with $\l \leq b-a$;
\item $F^{(b)}E^{(a)}1_\l$, where $a,b \in \N$, $\l \in \Z$ with $\l \geq b-a$,
\end{itemize}
and with the identifications $E^{(a)}F^{(b)}1_{b-a}=F^{(b)}E^{(a)}1_{b-a}$.

Write ${_\l\dot{\mathbb{B}}_\mu}$ for the collection of all canonical basis elements in the set $1_\l \cdot \dot{\mathbb{B}} \cdot 1_\mu$. We will use the following.

\begin{lemma}\label{lemma-canonical-ordered-basis}For fixed weights $\l,\mu\in \Z$, the set of canonical basis elements in ${_\l\dot{\mathbb{B}}_\mu}$ coincides with its intersection with either of the following collections
\begin{itemize}
\item $\dot{\mathbb{B}}^+:=\{E^{(a)}F^{(b)}1_\l|a, b \in \N, \l\leq b-a\}$;
\item $\dot{\mathbb{B}}^-:=\{F^{(b)}E^{(a)}1_\l|a, b \in \N, \l\geq b-a\}$.
\end{itemize}
\end{lemma}
\begin{proof}
See \cite[Lemma 2.6]{Lau1}.
\end{proof}

\paragraph{The small quantum group.} Let $ l \geq 1 $ be $ 2 $ or an odd integer.  Let $ \zeta_{2l} $ be a fixed primitive $2l$-th root of unity in $\C$, and set $\mc{O}_{2l}:=\Z[\zeta_{2l}]$. Let $\dot{U}_{\zeta_{2l}}$ be the base change of $\dot{U}$ by reduction of coefficients via the surjective ring map
\begin{equation}
\Z[q^{\pm}] \twoheadrightarrow \Z[q^{\pm}]/(\Psi_{2l}(q))\cong \mc{O}_{2l}.
\end{equation}

\begin{defn}\label{def-small-sl2-2l-root}
The non-unital associative quantum algebra $\dot{u}_{\zeta_{2l}}(\mathfrak{sl}_2)$ is the $\mc{O}_{2l}$-subalgebra in $\dot{U}_{\zeta_{2l}}$ generated by the collection of elements
$$\{E1_\l,~F1_\l|\l\in \Z\}.$$
\end{defn}

Set $[k]_{\zeta_{2l}}$ to be the quantum integer $[k]$ evaluated at $\zeta_{2l}$. In $\dot{U}_{\zeta_{2l}}(\mathfrak{sl}_2)$, equation \eqref{eqn-div-element} implies that
\begin{equation}\label{eqn-nilpotent-in-O2l}
E^{k}1_\l=[k]_{\zeta_{2l}}!E^{(k)}1_\l, \ \ \ \ \ F^{k}1_\l=[k]_{\zeta_{2l}}!F^{(k)}1_\l,
\end{equation}
which are non-zero if and only if $0\leq k\leq l-1$. Therefore $E^l 1_\l=0$ in $\dot{u}_{\zeta_{2l}}(\mathfrak{sl}_2)$.

We will refer to $\dot{u}_{\zeta_{2l}}(\mathfrak{sl}_2)$ as the \emph{Lusztig idempotented small quantum group} and denote it by $\dot{u}_{\mc{O}_{2l}}$.

When $l=p$ is a prime number, we also introduce an $\mathbb{O}_p$-integral version of this idempotented algebra. Denote by $\dot{U}_{\mathbb{O}_p}$ the idempotented $\mathbb{O}_p$-algebra obtained from $\dot{U}$ by reduction of coefficients along the ring map
$$\Z[q^{\pm 1}] \twoheadrightarrow \Z[q^{\pm}]/(1+q^2+\cdots+q^{2(p-1)})\cong \mathbb{O}_p.$$
It inherits the canonical basis of $\dot{U}$.

\begin{defn}\label{def-small-sl2-Op-form}
The non-unital associative quantum algebra $\dot{u}_{\mathbb{O}_p}(\mathfrak{sl}_2)$ is the subalgebra of $\dot{U}_{\mathbb{O}_p}$ generated by $\{E1_\l, F1_\l|\l\in \Z\}$. It has a presentation over $\mathbb{O}_p$ as follows
\begin{enumerate}
\item[(i)] $ 1_\l \cdot 1_{\mu}=\delta_{\l,\mu}1_{\l}$ for any $\l, \mu \in \Z$,
\item[(ii)] $ E 1_{\l} = 1_{\l+2} E $, \quad $ F 1_{\l} = 1_{\l-2} F $,
\item[(iii)] $ EF 1_{\l}-FE1_{\l} =[\l]_{\mathbb{O}_p}1_{\l}$,
\item[(iv)] $E^p=0$, \quad $F^p=0$.
\end{enumerate}
where in the third condition, $[\l]_{\mathbb{O}_p}=\sum_{i=0}^{\l-1} q^{\l -1-2i}\in \mathbb{O}_p$.
\end{defn}

Notice that $1+q^2+\cdots+q^{2(p-1)}=\Psi_p(q^2)=\Psi_p(q)\Psi_{2p}(q)$, where $\Psi_N$ stands for the $N$-th cyclotomic polynomial with integer coefficients. It follows that there is a surjective map of rings obtained from the compositions
$$\mathbb{O}_p\cong \Z[q^{\pm}]/(\Psi_p(q^2))\twoheadrightarrow \Z[q^{\pm}]/(\Psi_{2p}(q))\cong \mc{O}_{2p},$$
so that tensor product with $\mc{O}_{2p}$ via this map gives rise to an isomorphism of algebras
$$\dot{u}_{\mathbb{O}_p}\o_{\mathbb{O}_{p}}\mc{O}_{2p}\cong \dot{u}_{\mc{O}_{2p}}.$$

The algebra $\dot{u}_{\mathbb{O}_p}$ acquires an $\mathbb{O}_p$-integral basis by reduction of Lusztig's canonical basis. We will denote it by $\dot{\mathbb{B}}(\mathbb{O}_p)$, which now consists of
\begin{itemize}
\item $E^{(a)}F^{(b)}1_\l$, where $a,b \in \{0,1,\ldots, p-1\}$, $\l \in \Z$ with $\l \leq b-a$;
\item $F^{(b)}E^{(a)}1_\l$, where $a,b \in \{0,1,\ldots, p-1\}$, $\l \in \Z$ with $\l \geq b-a$,
\end{itemize}
with it understood that $E^{(a)}F^{(b)}1_{b-a}=F^{(b)}E^{(a)}1_{b-a}$. As for the generic case, set ${_\l\dot{\mathbb{B}}(\mathbb{O}_p)_\mu}$ to be the collection of elements $1_\l\cdot\dot{\mathbb{B}}(\mathbb{O}_p)\cdot 1_\mu$.

\begin{cor}For fixed weights $\l,\mu\in \Z$, the set of canonical basis elements in ${_\mu\dot{\mathbb{B}}(\mathbb{O}_p)_\l}$ coincides with its intersection with either of the following sets
\begin{itemize}
\item $\dot{\mathbb{B}}(\mathbb{O}_p)^+:=\{E^{(a)}F^{(b)}1_\l|0\leq a, b \leq p-1, \l\leq b-a\}$;
\item $\dot{\mathbb{B}}(\mathbb{O}_p)^-:=\{F^{(b)}E^{(a)}1_\l|0 \leq a, b \leq p-1, \l\geq b-a\}$.
\end{itemize}
\end{cor}
\begin{proof}This is a direct consequence of Lemma \ref{lemma-canonical-ordered-basis}.
\end{proof}

We will categorify the algebra $\dot{u}_{\mathbb{O}_p}$, together with the basis $\dot{\mathbb{B}}(\mathbb{O}_p)$, in the next section.

\begin{rmk}\label{rmk-restricted-u-dot}
The ring $\mathbb{O}_p$ affords a residue map onto $\F_p$, by sending $q\mapsto 1$. Base changing $\dot{u}_{\mathbb{O}_p}$ along this ring map gives us an idempotented form of the restricted universal enveloping algebra of $\mf{sl}_2$:
\[
\dot{u}_{\F_p}(\mf{sl}_2):=\dot{u}_{\mathbb{O}_p}\o_{\mathbb{O}_p}\F_p.
\]
This form of the restricted Lie algebra was utilized in \cite{GK} to construct modular fusion categories.
\end{rmk}

\paragraph{Some simple representations.} For each weight $\mu \in \{0,1,\ldots, p-1\}$, we define the highest weight module $V^\mu$ of $\dot{u}_{\mathbb{O}_p}$, as follows. As an $\mathbb{O}_p$-module, it is free of rank $\mu+1$:
$$V^\mu\cong \bigoplus_{i=0}^{\mu}\mathbb{O}_p\hat{1}_{\mu-2i},$$
where $\{\hat{1}_{\mu-2i}|i=0,\dots, \mu\}$ is a basis. The operators $E1_\l, F1_\l \in \dot{u}_{\mathbb{O}_p}$ act on $V^\mu$ by
\begin{equation}\label{eqn-module-basis}
F1_\l \cdot \hat{1}_{\mu-2i}=\delta_{\l,\mu-2i}[i+1]\hat{1}_{\mu-2i-2}, \ \ \ \ E1_\l \cdot \hat{1}_{\mu-2i}=\delta_{\l,\mu-2i}[\mu+1-i]\hat{1}_{\mu-2i+2},
\end{equation}
except that we define $[0]:=0$ and $\hat{1}_{\mu+2}=0$. One readily shows the relations in Definition \ref{def-small-sl2-Op-form} are satisfied. The rank $p$ representation $V^{p-1}$ is usually known as the \emph{Steinberg module}.

%
\subsection{Grothendieck ring as small quantum \texorpdfstring{$\mathfrak{sl}(2)$}{sl(2)}}
\label{grothring}
%
For any of the differentials $\dif$ in Definition \ref{def-dif-on-U}, we denote the abelian category of $p$-DG modules over $\UC$ by $\UC_\dif\dmod$. It decomposes into a direct sum of abelian categories
\[
\UC_\dif\dmod=\bigoplus_{\mu, \l \in \Z} ({_{\mu}\UC_{\l}})_\dif\dmod.
\]
There is a natural induction functor, coming from composition of $1$-morphisms: for any weights $\l_1, \l_2, \l_3, \l_4 \in \Z$,
\begin{eqnarray}\label{eqn-induction-abelian}
({_{\l_1}\UC_{\l_2}}\otimes {_{\l_3}\UC_{\l_3}})_\dif\dmod & \lra & \delta_{\l_2, \l_3} ({_{\l_1}\UC_{\l_4}})_\dif\dmod\\
\mc{M} \boxtimes \mc{N} & \mapsto & \mathrm{Ind}(\mc{M}\boxtimes \mc{N}). \nonumber
\end{eqnarray}

Passing to the derived category, the induction functor descends to an exact functor
\begin{equation}\label{eqn-induction-derived}
\mathrm{Ind}: \mc{D}(\UC\o\UC,\dif) \lra \mc{D}(\UC,\dif),
\end{equation}
and hence a map of $\mathbb{O}_p$-modules
\begin{equation}\label{eqn-induction-K0}
{[\mathrm{Ind}]}: K_0(\UC\o\UC,\dif) \lra K_0(\UC,\dif).
\end{equation}

We will not be able to say much about a general differential. Henceforth, we specialize to the case $\dif =\dif_1$. What follows will also apply to $\dif_{-1}$ by applying the automorphism $\widetilde{\tau}$.

We start by generalizing Definition \ref{def-divided-power-mod}.
\begin{defn}\label{def-canonical-mod}
For any $a,b\in \N$ and $\l \in \Z$, let $\mc{E}^{(a)}\mc{F}^{(b)}\1_\l$ be the induced $p$-DG module
\[
\mc{E}^{(a)}\mc{F}^{(b)}\1_\l:=\mathrm{Ind}_{\UC_{\l-2b}\o {\UC_{\l}}}^{\ \UC_\l}\left(\mc{E}^{(a)}\1_{\l-2b}\boxtimes \mc{F}^{(b)}\1_{\l}\right),
\]
where the induction is along the composition
\[\UC_{\l-2b}\o {\UC_{\l}}\lra \UC_{\l}, \ \ \ \  \xi_1\1_{\l-2b} \o {\1_{\mu}\xi_2\1_{\l}} \mapsto \delta_{\l-2b, \mu}\xi_1 \xi_2 \1_\l. \]
One defines $\mc{F}^{(b)}\mc{E}^{(a)}\1_\l$ similarly. We will refer to these modules as the \emph{canonical modules} over $\UC_\l$. Notice that the statements of Corollary \ref{cor-compactness-of-divided-power-mod} hold for these modules as well.
\end{defn}

Now fix a weight $\l \in \Z$, and consider $\UC_\l$.

Assume that $\l \leq p-1$. The Fc-filtrations on $\EC\FC\1_\l$ and $\FC\EC\1_\l$ established in Section \ref{subsec-EF-decomp} can be regarded as an algorithm to filter an
arbitrary representable module of the form $\EC^{r_k}\FC^{s_k}\cdots\EC^{r_1}\FC^{s_1}\1_\l\{k\} \in \UC_\l$ by $p$-DG modules in the set
$\widetilde{\mathbb{X}}_\l:=\{\EC^r\FC^s\1_\l|r,s\in \N, t\in \Z\}$ with possible $\Z$-grading shifts. Furthermore, the modules in $\widetilde{\mathbb{X}}_\l$ can be further
filtered by canonical modules $\EC^{(r)}\FC^{(s)}\1_\l$ with grading shifts. Hence in the derived category, any representable module is presented by a convolution of canonical
modules with grading shifts. Now we may further shrink the size of canonical modules needed, since those of the form $\EC^{(r)}\FC^{(s)}\1_{\l}$ for which either $r\geq p$ or
$s\geq p$ are contractible, using Corollary \ref{cor-compactness-of-divided-power-mod}. It follows that the collection \begin{equation}\label{eqn-compact-cofib-generator}
\mathbb{X}^+_\l:=\{\EC^{(r)}\FC^{(s)}\1_\l|0\leq r,s \leq p-1\} \end{equation} forms another generating set of $\mc{D}(\mc{U}_\l)$. By Corollary
\ref{cor-compactness-of-divided-power-mod} again, $\mathbb{X}^+_\l$ consists of compact cofibrant modules. Using Proposition \ref{prop-criterion-derived-equivalence}, we have a
derived equivalence \begin{equation} \mc{D}(\UC_\l) \cong \mc{D}(\END_{\UC_\l}(\mathbb{X}^+_\l)), \end{equation} where we use the notation introduced in Remark
\ref{rmk-shrinking-category-to-algebra}. Furthermore, the cofibrance of the modules in $\mathbb{X}^+_\l$ allows us to compute the endomorphism algebra as usual. Now we use the
following lemma, which is a direct consequence of \cite[Proposition 9.8]{Lau1}.\footnote{Lauda shows this over a field, and this result is strengthened in \cite[Proposition
5.15]{KLMS} over $\Z$.}

\begin{lemma}\label{lemma-positive-end-algebra}
The endomorphism algebra $\END_{\UC_\l}(\mathbb{X}^+_\l)$ is a strongly positive $p$-DG algebra. \hfill$\square$
\end{lemma}

Now suppose $\l >1-p$.

The entire argument above goes through, except with an alternative choice of compact
cofibrant generators $$\mathbb{X}^-_\l:=\{\FC^{(s)}\EC^{(r)}\1_\l| 0 \leq r,s \leq p-1 \}.$$

Combined with Corollary \ref{cor-K-group-positive}, we have established the following.

\begin{cor}\label{cor-K0-of-Ulambda}For any weight $\l \in \Z$, the Grothendieck group of the $p$-DG category $\UC_\l$ is isomorphic to
\[
K_0(\UC_\l) \cong \mathbb{O}_p\langle \ \dot{\mathbb{B}}(\mathbb{O}_p)_\l \rangle,
\]
the free $\mathbb{O}_p$-module spanned by $\dot{\mathbb{B}}(\mathbb{O}_p)_\l$. \hfill$\square$
\end{cor}

The strong positivity also establishes the following.

\begin{cor}\label{cor-Kunneth-for-UC}
For any weights $\l_1, \l_2, \l_3, \l_4 \in \Z$, the $p$-DG categories ${_{\l_1}\UC_{\l_2}}$, ${_{\l_3}\UC_{\l_4}}$ enjoy the K\"{u}nneth property
\[
K_0({_{\l_1}\UC_{\l_2}}) \o_{\mathbb{O}_p} K_0({_{\l_3}\UC_{\l_4}})\cong K_0({_{\l_1}\UC_{\l_2}}\o {_{\l_3}\UC_{\l_4}}).
\]
\end{cor}
\begin{proof} This follows from Lemma \ref{lemma-positive-end-algebra} and Corollary \ref{cor-Moritaly-positive-Kunneth}.
\end{proof}

It follows that taking $K_0$ commutes with tensor products for $\UC$ in equation \eqref{eqn-induction-K0},
\[
K_0(\UC\o \UC) \cong K_0(\UC) \o_{\mathbb{O}_p} K_0(\UC),
\]
and the symbol of the induction functor equips $K_0(\UC)=\bigoplus_{\mu,\l\in \Z}K_0({_\mu\UC_\l})$ with an idempotented $\mathbb{O}_p$-algebra structure, whose multiplication is given by the induction functor:
\begin{align}
[\mathrm{Ind}]: K_0(\UC) \otimes_{\mathbb{O}_p} K_0(\UC)  \lra   K_0(\UC).
\end{align}

Now we have our main theorem.

\begin{thm}\label{thm-itsanalghom}
There is an isomorphism of $\mathbb{O}_p$-algebras
$$\dot{u}_{\mathbb{O}_p}(\mathfrak{sl}(2)) \lra K_0(\UC,\dif_{\pm 1})$$
sending $E1_\l \mapsto [\mc{E\1_\l}]$ and $1_\l F \mapsto [\1_\l \mc{F}]$ for any weight $\l \in \Z$.
\end{thm}

\begin{proof} We first need to show that the defining relations for $\dot{u}_{\mathbb{O}_p}$ as in Definition \ref{def-small-sl2-2l-root} hold in $K_0(\UC)$, and the non-trivial relations to check are (iii) and (iv).

By Proposition \ref{prop-filtration-EF}, in the derived category $\mc{D}(\UC,\dif_1)$, the representable modules $\mc{\EC\FC\1_\l}$ fits into a convolution diagram (see Remark \ref{rmk-convolution-finite-cell}):
\begin{equation*}
\begin{gathered}
\xymatrix@C=0.75em{0=F_{0} \ar[rr] && F_{1} \ar[rr] \ar[dl] && F_{2} \ar@{-}[r]\ar[dl] &\cdots\ar[r] &F_{\l-1} \ar[rr] && F_\l=\EC\FC\1_\l, \ar[dl]\\
& \1_{\l}\ar[ul]_{[1]}\{\l-1\} && \1_{\l}\{\l-3\} \ar[ul]_{[1]} && && \FC\EC\1_\l\ar[ul]_{[1]} &
}
\end{gathered}
\end{equation*}
Therefore in the Grothendieck group the desired relation holds. Relation (iv) follows from Corollary \ref{cor-compactness-of-divided-power-mod} of the previous chapter.

Moreover, the map in the statement sends $\dot{\mathbb{B}}(\mathbb{O}_p)_\l$ to the symbols of modules in $\mathbb{X}_\l^{\pm}$, which form a basis for $K_0(\UC)$ by Corollary \ref{cor-K0-of-Ulambda}. Therefore, it is an isomorphism.
\end{proof}

\begin{rmk}\label{rmk-catefy-restricted-udot}
Forgetting about the gradings, the ungraded $p$-differential $2$-category $(\UC,\dif_1)$ categorifies the idempotented restricted universal enveloping algebra $\dot{u}_{\F_p}(\mf{sl}_2)$ of \cite{GK}. We regard this as a categorical analogue of the base change construction in Remark \ref{rmk-restricted-u-dot}. See \cite[Remark 3.36]{KQ} for more comments.
\end{rmk}

%
\subsection{Cyclotomic quotients}
%

In this section we prove that $p$-DG cyclotomic quotients categorify simple $u_q(\mf{sl}_2)$-modules. The argument will be very similar to that of the previous section.

For any weight $\mu \in \N$, consider the category
$$\UC_\mu\cong \bigoplus_{\l\in \Z}{_\l\UC_\mu}.$$
It carries a natural module-category structure over $\UC$:
$$\UC\o \UC_\mu \lra \UC_\mu , \quad \quad \quad \xi_1\1_{\l^\prime} \o {\1_{\l}\xi_2\1_{\mu}} \mapsto \delta_{\l^\prime, \l}\xi_1 \xi_2 \1_\mu.$$

Following ideas of Khovanov-Lauda \cite{KL1} and Rouquier \cite{Rou2} (see also \cite[Section 1.4]{Web}), we define the \emph{cyclotomic quotient category} $\VC^{\mu}$ to be the quotient category of $\UC_\mu$ by morphisms in the two-sided ideal which is left monoidally generated by
\begin{itemize}
\item[(i)] Any morphism that contains the following subdiagram on the far right:
$$\begin{DGCpicture}
\DGCstrand(0,0)(0,1)
\DGCdot*>{1}
\DGCcoupon*(0.4,0.4)(0.7,0.8){$\mu$}
\end{DGCpicture}~.$$
\item[(ii)] All positive degree bubbles on the far right region $\mu$.
\end{itemize}
Here by ``two-sided'' we mean concatenating diagrams vertically from top and bottom to those in the relations, while by ``left monoidally generated'' we mean adding pictures from $\UC$ to the left of those generators. Schematically we depict elements in the ideal as follows.

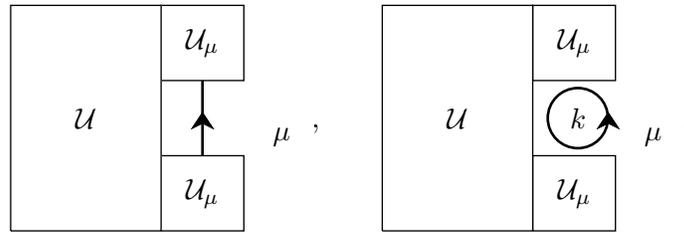
\[
\begin{DGCpicture}
\DGCstrand(0.55,1)(0.55,2)
\DGCdot*>{1.5}
\DGCcoupon(0,0)(1.1,1){$\UC_\mu$}
\DGCcoupon(0,2)(1.1,3){$\UC_\mu$}
\DGCcoupon(-2,0)(0,3){$\UC$}
\DGCcoupon*(1.2,1)(2,1.5){$\mu$}
\end{DGCpicture}, \quad \quad
\begin{DGCpicture}
\DGCbubble(0.6,1.5){0.4}
\DGCdot*>{1.5,R}
\DGCcoupon*(0.4,1.3)(0.8,1.7){$k$}
\DGCcoupon(0,0)(1.1,1){$\UC_\mu$}
\DGCcoupon(0,2)(1.1,3){$\UC_\mu$}
\DGCcoupon(-2,0)(0,3){$\UC$}
\DGCcoupon*(1.2,1)(2,1.5){$\mu$}
\end{DGCpicture}.
\]

One implication of these relations is that
\begin{equation}\label{cyclotomiccurl}
0 = \curl{R}{D}{$\mu$}{no}{$0$} = \sum_{a+b=\mu}
\onelineD{$b$}{no} ~\bigccwbubble{$a$}{$\mu$}.
\end{equation}
Moreover, every term with $a>0$ is also in the ideal, so that
\begin{equation}\label{eqn-cyclotomic-rel}
\begin{DGCpicture}
\DGCstrand(0,0)(0,1)
\DGCdot{.4}[r]{\small{$\mu$}}
\DGCdot*<{0}
\DGCcoupon*(0.7,0.4)(1.0,0.8){$\mu$}
\end{DGCpicture}~=0.
\end{equation}
The quotient of the nilHecke algebra $\NH_r$ by the two-sided ideal generated by $\mu$ dots on the rightmost strand $(x_r^\mu)$ is called the \emph{cyclotomic nilHecke algebra}.

\begin{rmk}\label{rmk-universal-quotient} This version of the cyclotomic quotient category $\mc{V}^\mu$ was due to Rouquier using his $2$-category in \cite{Rou2}, which is related to Khovanov-Lauda's construction using the negative half of $\UC$ in \cite{KL1} by equation \eqref{eqn-cyclotomic-rel}. Technically, Rouquier's $2$-category is slightly different from that of Khovanov-Lauda \cite{KL3,Lau1}, but a recent work of Brundan \cite{Brundan2KM}, building on the earlier work of Cautis-Lauda \cite{CauLau}, has established the equivalence of the two seemingly different definitions. 

There is also a \emph{universal cyclotomic quotient} category $\widetilde{\VC}^\mu$ introduced by Rouquier \cite{Rou2}.  It is constructed from $\UC_\mu$ as the quotient by relation (i), but not relation (ii).  Equivalently, any morphism factoring through an object $\l$ with $\l>\mu$ is killed. The parallel construction on Khovanov-Lauda's $2$-category is considered by  Webster in \cite{Web}. Note that counterclockwise bubbles of degree $\ge \mu+1$ are real bubbles, and lie inside this kernel, but counterclockwise bubbles of degree $\le \mu$ are not in the kernel. We will study the universal cyclotomic quotient in more detail in upcoming works.
\end{rmk}

Now let $\dif$ be a differential as in Definition \ref{def-dif-on-U}. By Corollary \ref{cor-dif-action-on-bubbles}, $\dif$ preserves the ideal above, so that it
induces a quotient differential on $\VC^\mu$.

\begin{defn}\label{def-cyclo-quo}
The \emph{cyclotomic quotient $p$-DG category} $(\VC^\mu,\dif)$ is the category $\VC^\mu$ equipped with the induced differential.
It is equipped with the obvious left action of $(\UC,\dif)$.
\end{defn}

To avoid potential confusion, we denote the regions in $\VC^\mu$ by $\hat{\1}_\l$ for $\l\in \Z$.

From now on, we specialize $\dif=\dif_1$, and we restrict to the case when $\mu \in \{0,1,\ldots,p-1\}$.

As a quotient category, $(\VC^\mu,\dif_1)$ inherits the $\mc{EF}$ and $\mc{FE}$ Fc-filtrations. It follows that any representable module $p$-DG module $\EC^{r_k}\FC^{s_k}\cdots\EC^{r_1}\FC^{s_1}\hat{\1}_\mu \{k\}$ over $\VC^\mu$ can be filtered by representable modules of the form $\FC^r\EC^s\hat{\1}_\mu$, where $r,s\in \N$. If $s>0$, $\FC^r\EC^s\hat{\1}_\mu\cong 0$; while if $r\geq p$, $\FC^r\EC^s\hat{\1}_\mu\cong 0$ is acyclic. Therefore, as for $\UC_\mu$, one may choose a set of compact cofibrant generators of $\VC^\mu$ to consist of
\begin{equation}
\mathbb{X}^\mu:=\{\hat{\1}_{\mu-2r}\FC^{(r)}\hat{\1}_\mu|0\leq r \leq p-1\}
\end{equation}

Now we compute the endomorphism ring $\END_{\VC^\mu}(\mathbb{X}^\mu)$. Let $H_{r,\mu}:=H^*(\mathbb{G}(r,\mu),\Bbbk)$ be the cohomology of the Grassmannian of $r$-planes in $\mu$-space. It has a presentation
$$H_{r,\mu}\cong\sym_r/(h_k | k \ge \mu+1-r).$$
With the induced quotient differential coming from $\sym_r$, it is a positive $p$-DG algebra. We claim that it is isomorphic to $\END_{\VC^\mu}(\FC^{(r)} \hat{\1}_\mu)$. In particular, $H_{\mu+1,\mu}=0$ so that $\FC^{(\mu + 1)} \hat{\1}_\mu = 0$ in $\VC^\mu$.

Let $x_r^\mu \in \NH_r$ denote the $\mu$-th power of a dot on the rightmost strand. One has an explicit presentation of the nilHecke algebra $\NH_r$ as a matrix algebra over $\sym_r$ \cite[Proposition 2.16]{KLMS}, which after tensoring with $\Lambda$ coincides with the endomorphism ring of $\FC^r \1_\mu$ in $\mc{U}$. On the one hand, it follows from the cyclotomic relation \eqref{eqn-cyclotomic-rel} that there is a surjection of $p$-DG algebras\footnote{Technically, we need the fact that the matrix presentation is preserved under $\dif_1$, see \cite[3.24]{KQ}.}
\begin{equation}\label{eqn-surj-cyclo-to-endo}
\NH_r/(x_r^\mu)\twoheadrightarrow \END_{\VC^\mu}(\FC^r\hat{\1}_\mu).
\end{equation}
On the other hand, the $2$-representation of $\UC$ on the flag category of highest weight $\mu$ constructed by Lauda \cite[Section 7]{Lau1}, when restricted to $\UC_\mu$, factors through the ideal in the definition of $\VC^\mu$. Furthermore, it is shown there that $\FC^{r}\1_\mu$ is sent under the representation to the cohomology of an $r$-step flag variety in $\mu$-space. It follows that the above surjection \eqref{eqn-surj-cyclo-to-endo} is an isomorphism of graded vector spaces, and hence actually an isomorphism of $p$-DG algebras. The two-sided ideal $(x_r^\mu)$ is identified, under the matrix presentation, with the ideal $(h_k | k \ge \mu+1-r) \subset \sym_r \cong Z(\NH_r)$ generated by diagonal matrices. Therefore, the endomorphism ring of the column module $\FC^{(r)} \hat{\1}_\mu$ is precisely $H_{r,\mu}$. We also refer the reader to \cite[Section 5]{Lau3} for a more explicit identification of the algebra $\END_{\VC^\mu}(\FC^r \hat{\1}_\mu)$ with an $r!\times r!$-matrix algebra with coefficients in $H_{r,\mu}$. 

Granted this, the endomorphism algebra of the direct sum of the compact generators can be identified with
\begin{equation}
\END_{\VC^\mu}(\mathbb{X}^\mu)\cong\prod_{r=0}^{p-1} \END_{\VC^\mu}(\FC^{(r)}\hat{\1}^\mu)\cong \prod_{r=0}^{p-1}H_{r,\mu},
\end{equation}
and therefore it is strongly positive. It follows that
$$\mc{D}(\VC^\mu)\cong \mc{D}(\END_{\VC^\mu}(\mathbb{X}^\mu))\cong \prod_{r=0}^{p-1}\mc{D}(H_{r,\mu}),$$
so that we may compute the Grothendieck group of $\mc{V}^\mu$ from Corollary \ref{cor-K-group-positive}.

Now the natural left action
$$\UC \o \VC^\mu \lra \VC^\mu$$
gives rise to an induction functor on the corresponding derived categories,
$$\mc{D}(\UC \o \VC^\mu) \lra \mc{D}(\VC^\mu),$$
which in turn descends to a map of $\mathbb{O}_p$-modules on the Grothendieck group level:
\[
K_0(\UC\o\VC^\mu) \lra K_0(\VC^\mu).
\]
Since both $\UC$ and $\VC^\mu$ are locally $p$-DG Morita equivalent to positive $p$-DG algebras, the K\"{u}nneth formula implies that the action on Grothendieck groups is $\mathbb{O}_p$-linear, and it equips $K_0(\VC^\mu)$ with the structure of $\dot{u}_{\mathbb{O}_p}(\mathfrak{sl}(2))$-module.

\begin{thm}\label{thm-cyclo-quot}
Let $\mu\in \{0,1\dots, p-1\}$. There is an isomorphism
\[
 K_0(\VC^\mu) \cong V^\mu
\]
of $u_q(\mf{sl}_2)$-modules, where $V^\mu$ is the irreducible module of highest weight $\mu$ defined over $\mathbb{O}_p$.
\end{thm}

\begin{proof} The computation of the Grothendieck group as an $\mathbb{O}_p$-module follows from Corollary \ref{cor-K-group-positive}. In fact we can identify a basis of $K_0(\VC^\mu)$ with that of the module $V^\mu$ by identifying $[\hat{\1}_{\mu-2r}\FC^{(r)}\hat{\1}_\mu]$ with the basis element of $\hat{1}_{\mu-2r}$ defined in equation (\ref{eqn-module-basis}).
\end{proof}


\appendix

\section{Classification of derivations on \texorpdfstring{$\mc{U}$}{U}}\label{sec-classification-of-dif}

In this chapter we outline the tedious calculation required to prove Proposition \ref{prop-classification-of-dif-on-U}. We also give a cursory discussion of some degenerate
differentials which can appear in characteristic $2$. Most of the calculations will be left to the reader as exercises, but we will describe what can be gleaned from each
computation.

Before starting, we provide the reader with some useful \emph{bubble slide} relations, which describe how to move bubbles across a strand. The general bubble slide relation can be
found in \cite[Proposition 5.6]{Lau1}. We will only need to slide bubbles of degree $2k$ for $k=1,2$.

\begin{subequations} \label{bubbleslidesonetwo}
\begin{align}
\cwbubble{$1$}{} \onelineshort{$0$}{no} - \onelineshort{$0$}{no} \cwbubble{$1$}{} =2 \onelineshort{$1$}{no} \label{bubslidedeg1}
\end{align}
\begin{align}
\ccwbubble{$1$}{}
\onelineshort{$0$}{no} - \onelineshort{$0$}{no} \ccwbubble{$1$}{} = -2 \onelineshort{$1$}{no}
\end{align}
\begin{align}
\cwbubble{$2$}{} \onelineshort{$0$}{no} - \onelineshort{$0$}{no} \cwbubble{$2$}{} =2 \onelineshort{$1$}{no} \cwbubble{$1$}{} + 3 \onelineshort{$2$}{no}
\label{bubslidedeg2}
\end{align}
\begin{align}
\ccwbubble{$2$}{} \onelineshort{$0$}{no} - \onelineshort{$0$}{no} \ccwbubble{$2$}{} = -2 \onelineshort{$1$}{no} \ccwbubble{$1$}{} + \onelineshort{$2$}{no}
\end{align}
\end{subequations}

We begin the classification of differentials on $\UC$. We restrict our attention to either $\UC_{\even}$ or $\UC_{\odd}$. For each relation, we describe what properties a degree
$2$ derivation must have in order to preserve that relation.

\paragraph{NilHecke relations and preliminaries.}

For an arbitrary map $\dif$ of degree $2$ on $\UC$, we must have
\begin{align*} \dif \left( \onelineshort{$1$}{$\l$} \right) = b_{\l} \onelineshort{$2$}{$\l$} + \a_{\l} \cwbubble{$1$}{} \onelineshort{$1$}{$\l$} + \b_{\l} \cwbubble{$1$}{}
\cwbubble{$1$}{} \onelineshort{$0$}{$\l$} + \g_{\l} \cwbubble{$2$}{} \onelineshort{$0$}{$\l$}, \end{align*}
\begin{align*} \dif \left( \crossing{$0$}{$0$}{$0$}{$0$}{$\l$} \right) = c_\l \twolines{$0$}{$0$}{$\l$} + \mu_\l \crossing{$1$}{$0$}{$0$}{$0$}{$\l$} + \rho_\l
\crossing{$0$}{$1$}{$0$}{$0$}{$\l$} + \delta_\l \cwbubble{$1$}{} \crossing{$0$}{$0$}{$0$}{$0$}{$\l$}. \end{align*}
This is because, by Lauda's classification of morphisms, the terms on the right hand side form a basis for morphisms in the appropriate degree.

We leave the reader to compute $\dif$ of both sides of \eqref{NHreldotforce}, under the assumption that $\dif$ is a 2-categorical derivation. The bubble slide relations will be
useful. If $\dif$ preserves this relation, then one can immediately deduce a number of facts: \begin{itemize} \item $b_\l = b_{\l+2}$, $\a_\l = \a_{\l+2}$, $\b_{\l} =
\b_{\l+2}$, $\g_{\l} = \g_{\l+2}$ \item $2\a_\l=0$, $4\b_\l=0$, $\g_\l=0$ \item $\delta_\l = -\a_\l$. \item $\mu_\l = -b_\l-c_\l$ and $\rho_\l = -b_\l+c_\l$. \end{itemize} We write $b=b_\l$ because it does not depend on the weight $\l$, and similarly for the variables $\a, \b, \g, \delta$.

\begin{rmk} 
Outside of characteristic $2$ we must have $\a=\b=\delta=\g=0$, which implies that $\dif$ preserves the nilHecke algebra $\NH_n \subset \END_{\UC}(\mc{E}^n)$. In
characteristic $2$, however, it is possible for $\a = \delta$ and $\b$ to be nonzero. Note that when $2 \a = 4 \b = 0$, the bubbles created by the differential will slide freely
through any strand (using \eqref{bubslidedeg1}), making these bubbles relatively unobtrusive. Lest one think that, over $\Z$ instead of $\Bbbk$, there may be differentials where
$4 \b = 0$ but $2 \b \ne 0$, the reduction to bubbles relation will imply that $2 \b = 0$. There will be additional restrictions on the values of $\a$ and $\b$ imposed by the
requirement that $\dif^2 = 0$, but nonetheless there are degenerate $2$-differentials with a nonzero parameter $\a$ or $\b$.

We have not studied these degenerate $2$-differentials in any detail. We do not expect them to give rise to Fc-filtrations, or to induce a differential on $\NH_n \o
\Lambda$ which allows for a nice study of the Grothendieck group.

We will discuss some of the additional properties of these degenerate $2$-differentials within remarks throughout this chapter, all notated with \textbf{(Char $p=2$)}. Outside of
these remarks, we will assume henceforth that $\a = \b = 0$. \end{rmk}

We will show soon enough that if $b = 0$ and $\dif^p=0$ then $c_\l = 0$ as well, for all $\l$ (see Lemma \ref{pnilpotentNH}). We may as well write $a_\l = \frac{c_\l}{b}$ when $b \ne 0$, and $a_\l
=0$ when $b = 0$. Thus $\dif$ has the following form:
\begin{align*} \dif \left( \onelineshort{$1$}{$\l$} \right) = b \onelineshort{$2$}{$\l$}, \end{align*}
\begin{align*} \dif \left( \crossing{$0$}{$0$}{$0$}{$0$}{$\l$} \right) = b a_\l \twolines{$0$}{$0$}{$\l$} + b (-1-a_\l) \crossing{$1$}{$0$}{$0$}{$0$}{$\l$} + b(-1 + a_\l)
\crossing{$0$}{$1$}{$0$}{$0$}{$\l$}, \end{align*}
depending on the parameters $b$ and $a_\l$. In particular, it agrees with the differential of Definition \ref{def-dif-on-U} up to multiplication by $b$.

Relations \eqref{NHrelR2} and \eqref{NHrelR3} live within $\NH_3$. That these relations are preserved by any differential of the above form was shown in \cite{KQ}. Technically, the calculation for \eqref{NHrelR3} in \cite{KQ} was done under the assumption that $a_\l = a_{\l+2}$, but it is not hard to redo the computation in the general case.

The same discussion applies to the downwards pointing nilHecke algebra inside $\END_{\UC}(\mc{F}^n)$. As usual, we denote by $\overline{b}$, $\oa_\l$ the corresponding parameters.

\paragraph{Biadjointness relations.}
Let us note that the cup with a dot and the cup with a bubble are a basis for the degree $2$ morphisms from $\1$ to $\mc{E}\mc{F}$ (resp. $\mc{F}\mc{E}$). Therefore, when $\dif$ acts on a cup or cap, the result has the same form as in Definition \ref{def-dif-on-U}, though with unknown coefficients.

\begin{itemize}
\item[\eqref{biadjoint1}] Because $\dif$ of the identity is always zero, we need to show that $\dif$ of each other diagram is zero. This check will relate the coefficients for $\dif$ of a clockwise cap with $\dif$ of a counterclockwise cup, and vice versa. We leave it as an exercise to confirm that the coefficients for counterclockwise cups and caps are determined from the coefficients of clockwise cups and caps precisely as in Definition \ref{def-dif-on-U}. (This exercise will use the bubble sliding relation \eqref{bubslidedeg1}.) Henceforth we will name the coefficients $x_\l$, $y_\l$, $\ox_\l$, and $\oy_\l$ as in Definition \ref{def-dif-on-U}.

\item[\eqref{biadjointdot}] This check gives the relation between $\dif$ of an upward-pointing and downward-pointing dot. It is easy, and implies that $b=\overline{b}$.

\textbf{(Char $p=2$):} This also implies that $\a = \overline{\a}$ and $\b = \overline{\b}$.

\item[\eqref{biadjointcrossing}] This is a simple but annoying computation, using the dot forcing relation \eqref{NHreldotforce} and the bubble sliding relation
\eqref{bubslidedeg1}. It is an excellent warmup exercise for the reader. The equality of both sides here amounts to \begin{equation} b(a_\l - \oa_\l) = x_{\l+2} - x_\l - 2 y_\l.
\label{paramAalt} \end{equation} The reader should compare this to \eqref{paramA}. Similarly, the upside-down version gives \begin{equation} b(\oa_\l - a_\l) = \ox_{\l+2} - \ox_\l - 2 \oy_\l. \label{paramAalt2} \end{equation} \end{itemize}

\paragraph{Positivity and Normalization of bubbles.} Let us consider what happens when $\dif$ is applied to a clockwise bubble of degree $2k$. Ignore, for the moment, the distinction between real and fake bubbles.
\begin{align*}
& \qquad \dif \left( \bigcwbubble{$k$}{$\l$} \right) = \dif \left(
\begin{DGCpicture}
\DGCbubble(0,0){0.5}
\DGCdot*<{0.25,L}
\DGCdot{-0.25,R}[r]{$_{k+\l-1}$}
\DGCdot*.{0.25,R}[r]{$\l$}
\end{DGCpicture}
~\right)
= \\ & (b(k+\l-1)+x_{\l-2} + \ox_{\l-2})~
\begin{DGCpicture}
\DGCbubble(0,0){0.5}
\DGCdot*<{0.25,L}
\DGCdot{-0.25,R}[r]{$_{k+\l}$}
\DGCdot*.{0.25,R}[r]{$\l$}
\end{DGCpicture}
~+ (y_{\l-2} + \oy_{\l-2})
\begin{DGCpicture}
\DGCbubble(0,0){0.5}
\DGCdot*<{0.25,L}
\DGCdot{-0.25,R}[r]{$_{k+\l-1}$}
\DGCdot*.{0.25,R}[r]{$\l$}
\end{DGCpicture}
\bigcwbubble{$1$}{} = \\ &
(b(k+\l-1)+x_{\l-2} + \ox_{\l-2}) \bigcwbubble{$k+1$}{$\l$} + (y_{\l-2} + \oy_{\l-2}) \bigcwbubble{$k$}{$\l$} \bigcwbubble{$1$}{}
\end{align*}

The positivity and normalization relations imply that the LHS is zero when $k \le 0$, since the bubble of degree $2(k)$ is a scalar (either $0$ or $1$). When $k < -1$, the RHS is
zero for the same reason. When $k=-1$, the RHS is equal to $b(\l-2) + x_{\l-2} + \ox_{\l-2}$, which is zero if and only if \begin{equation} \label{paramXalt} x_{\l-2} + \ox_{\l-2}
= -b (\l-2).\end{equation} When $k=0$, the RHS is equal to a bubble of degree $2(1)$ with coefficient $b(\l-1) + x_{\l-2} + \ox_{\l-2} + y_{\l-2} + \oy_{\l-2}$. Combined with
\eqref{paramXalt}, this is zero if and only if \begin{equation} y_{\l-2} + \oy_{\l-2} = -b. \label{paramYalt} \end{equation} The reader should compare these equations with
\eqref{paramX} and \eqref{paramY}.

With these equalities in place, our formula is quite simple, and no longer depends on the ambient weight.
\begin{subequations} \label{difbubformula}
\begin{align}
\dif \left( \bigcwbubble{$k$}{} \right) = b(k+1) \bigcwbubble{$k+1$}{} - b \bigcwbubble{$k$}{} \bigcwbubble{$1$}{} \label{difcwbub}
\end{align}

Technically, this discussion only applies to real bubbles. The bubble of degree $2(-1)$ is real when $\l \ge 2$, so that we have a relation between $x_\l$ and $\ox_\l$ (resp.
$y_\l$ and $\oy_\l$) when $\l \ge 0$. When $\l = 1$, the bubble of degree $0$ is real, so that \[x_{-1} + \ox_{-1} + y_{-1} + \oy_{-1} = 0.\] However, combining \eqref{paramAalt}
with \eqref{paramAalt2} one can deduce that \[x_{-1} + \ox_{-1} + 2y_{-1} + 2y_{-1} = x_{1} + \ox_{1} = -b.\] Combining these two equations, we deduce \eqref{paramXalt} and
\eqref{paramYalt} for $\l=1$ as well.

We leave the reader to perform the analogous computation for counterclockwise bubbles. The degree $2(-1)$ bubble is real for $\l \le -2$, and the coefficients have subscript $\l$ instead of $\l-2$. One can deduce \eqref{paramXalt} and \eqref{paramYalt} for $\l \le 0$, so that these formula hold for all $\l \in \Z$. Again, we deduce that
\begin{align}
\dif \left( \bigccwbubble{$k$}{} \right) = b(k+1) \bigccwbubble{$k+1$}{} - b \bigccwbubble{$k$}{} \bigccwbubble{$1$}{} \label{difccwbub}
\end{align}
\end{subequations}
At this point we have only shown that the formulas \eqref{difcwbub} and \eqref{difccwbub} hold for real bubbles.

\textbf{(Char $p=2$):} The inclusion of non-zero $\a$ or $\b$ terms will throw a wrench into these computations. For instance, \eqref{paramYalt} will be replaced with
\[
y_{\l-2} +\oy_{\l-2} = -b + (1 - \l) \a
\]
and \eqref{difcwbub} will be replaced by
\[
\begin{array}{ll}
\dif \left( \bigcwbubble{$k$}{} \right) & = b(k+1) \bigcwbubble{$k+1$}{} + (k \a-b) \bigcwbubble{$k$}{}\bigcwbubble{$1$}{} \\ &
+ (k+\l-1)\b \bigcwbubble{$k-1$}{} \bigcwbubble{$1$}{} \bigcwbubble{$1$}{}.
\end{array}
\]
On the surface it appears that these formula depend on the ambient weight,
although the fact that we work within $\UC_{\even}$ or $\UC_{\odd}$ and the fact that $2 \b = 2\a =0$ implies that $\l \b$ and $\l \a$ are constants. We have not studied this
differential on $\Lambda$ in any detail.  Surprisingly enough, the parameters $\a$ and $\b$ make no trouble in the remaining relations.

\paragraph{Infinite Grassmannian relations.} We will check the infinite Grassmannian relations \eqref{infgrass} under the assumption that \eqref{difbubformula} holds for all
bubbles, real and fake. The first several relations serve as the definition of fake bubbles, so that this check will confirm \eqref{difbubformula} for fake bubbles.

Suppose that $m \in \Z, m \ge 1$.
\begin{align*}
 \dif \left( \sum_{k+l=m} \bigcwbubble{$k$}{} \hspace{-0.1in} \bigccwbubble{$l$}{} \right) = &
 \sum_{k+l=m} b(k+1) \bigcwbubble{${k+1}$}{}\hspace{-0.1in} \bigccwbubble{$l$}{} + b(l+1)\bigcwbubble{$k$}{} \hspace{-0.1in} \bigccwbubble{${l+1}$}{}
 \\ &
 - b \bigcwbubble{$k$}{} \hspace{-0.1in} \bigccwbubble{$l$}{} \left( \bigcwbubble{$1$}{} + \bigccwbubble{$1$}{} \right)
\end{align*}
The parenthetical in the final term is equal to zero. A diagram $\cwbubble{$_{k'}$}{} \ccwbubble{$_{l'}$}{}$ with $k'+l' = m+1$ appears in the RHS with coefficient exactly $bk'$ from the first term and $bl'$ from the second term (even when either $k'$ or $l'$ is zero). Therefore, the final result is equal to
\begin{align*}
b(m+1) \sum_{k' + l' = m+1} \bigcwbubble{${k'}$}{} \hspace{-0.1in} \bigccwbubble{${l'}$}{}
\end{align*}
This is zero by the infinite Grassmannian relation \eqref{infgrass}.

\paragraph{Reduction to bubbles.} Let us compute $\dif$ of a formal curl.
\begin{align*}
& \dif \left( \curl{R}{U}{no}{$k$}{$0$} \hspace{-0.1in} \right) = \dif \left( - \sum_{i+j=k} \oneline{$i$}{no} \bigcwbubble{$j$}{} \right) = \\
& - \sum_{i+j=k} bi \oneline{$i+1$}{no} \bigcwbubble{$j$}{} + b(j+1) \oneline{$i$}{no} \bigcwbubble{${j+1}$}{} - b \oneline{$i$}{no} \bigcwbubble{$j$}{} \hspace{-0.1in} \bigcwbubble{$1$}{} \\
& = bk \curl{R}{U}{no}{$_{k+1}$}{$0$} - b \oneline{$0$}{no} \bigcwbubble{${k+1}$}{} - b \curl{R}{U}{no}{$k$}{$0$} \hspace{-0.1in} \bigcwbubble{$1$}{}
\end{align*}

The first two terms of the middle line combine to equal the first two terms of the bottom line, by an argument similar to that used in the check of the infinite Grassmannian relation. Similarly, we have:

\begin{align*}
\dif \left( \hspace{-0.1in} \curl{L}{U}{no}{$k$}{$0$} \right) = bk \curl{L}{U}{no}{$_{k+1}$}{$0$} + b \bigccwbubble{${k+1}$}{} \oneline{$0$}{no} - b \bigccwbubble{$1$}{}\hspace{-0.1in}
\curl{L}{U}{no}{$k$}{$0$}
 \end{align*}

On the other hand, consider what happens to a normal curl.

\begin{align*}
\dif \left( \curl{R}{U}{$\l$}{no}{$0$} \hspace{-0.1in} \right) = &(-2b+x_{\l-2} + \ox_{\l-2}) \curl{R}{U}{$\l$}{no}{$1$} - b \oneline{$0$}{$\l$} \bigcwbubble{}{} \\ & +
(y_{\l-2} + \oy_{\l-2}) \curl{R}{U}{$\l$}{no}{$0$} \hspace{-0.1in} \bigcwbubble{$1$}{}
\end{align*}

Note the following three observations. All the terms with $a_{\l}$ will cancel out, when we force all dots into the curl. The bubble with no dots has degree $2(-\l+1)$. Also,
$-2b+x_{\l-2}+\ox_{\l-2} = -b\l$ by \eqref{paramXalt} and $y_{\l-2}+\oy_{\l-2}=-b$ by \eqref{paramYalt}. Combining these observations, we apply \eqref{curlstuff} to obtain:

\begin{align*}
\dif \left( \curl{R}{U}{${\l}$}{no}{$0$} \hspace{-0.1in} \right) = b(-\l) \curl{R}{U}{${\l}$}{$_{1-\l}$}{$0$} -b \curl{R}{U}{no}{$_{-\l}$}{$0$}\hspace{-0.1in} \bigcwbubble{$1$}{} -b
\oneline{$0$}{$\l$} \bigcwbubble{${1-\l}$}{}
\end{align*}

This agrees precisely with the computation for a formal curl of degree $-\l$. The check for a left curl is similar.

\paragraph{Identity decomposition.} The differential of the LHS is zero, so we need to check the same for the RHS. Consider the first diagram on the RHS. We have not
written down a formula for the differential of a sideways crossing, because it is unenlightening. On the other hand, when two sideways crossings meet as below, there is a miraculous cancellation. We leave this as an exercise.
\begin{subequations} \label{difIdDecomp}
\begin{align} \dif \left(
\begin{DGCpicture}
\DGCstrand(0,0)(1,1.25)(0,2.5)
\DGCdot*>{0.25}
\DGCdot*>{1.25}
\DGCdot*>{2.25}
\DGCstrand(1,0)(0,1.25)(1,2.5)
\DGCdot*.{1.66}[l]{$\l$}
\DGCdot*<{0.25}
\DGCdot*<{1.25}
\DGCdot*<{2.25}
\end{DGCpicture}
\right) = (b - \ox_{\l-2})~
\begin{DGCpicture}
\DGCstrand(0,0)(1,1)/u/(0,1)/d/(1,0)/d/
\DGCdot*>{1.5}
\DGCstrand(0,2.5)/d/(1,2.5)
\DGCdot*<{2}
\end{DGCpicture}~ + (b - x_{\l-2})~
\begin{DGCpicture}
\DGCstrand(0,0)(1,0)/d/
\DGCdot*>{.5}
\DGCstrand(0,2.5)/d/(1,1.5)/d/(0,1.5)/u/(1,2.5)
\DGCdot*<{1}
\end{DGCpicture}
\label{difIdDecompPos} \end{align}
\begin{align} \dif \left(
\begin{DGCpicture}
\DGCstrand(0,0)(1,1.25)(0,2.5)
\DGCdot*<{0.25}
\DGCdot*<{1.25}
\DGCdot*<{2.25}
\DGCstrand(1,0)(0,1.25)(1,2.5)
\DGCdot*.{1.66}[l]{$\l$}
\DGCdot*>{0.25}
\DGCdot*>{1.25}
\DGCdot*>{2.25}
\end{DGCpicture}
\right) = (-b - x_{\l}-2y_{\l})~
\begin{DGCpicture}
\DGCstrand(0,0)(1,1)/u/(0,1)/d/(1,0)/d/
\DGCdot*<{1.5}
\DGCstrand(0,2.5)/d/(1,2.5)
\DGCdot*>{2}
\end{DGCpicture}~ + (-b - \ox_{\l}-2\oy_{\l})~
\begin{DGCpicture}
\DGCstrand(0,0)(1,0)/d/
\DGCdot*<{.5}
\DGCstrand(0,2.5)/d/(1,1.5)/d/(0,1.5)/u/(1,2.5)
\DGCdot*>{1}
\end{DGCpicture}
\label{difIdDecompNeg} \end{align}
\end{subequations}

Let us now check \eqref{IdentityDecompPos}. When $\l < 0$, the sum in \eqref{IdentityDecompPos} vanishes, and the curls in \eqref{difIdDecompPos} also vanish,
so all is well. When $\l=0$, the sum in \eqref{IdentityDecompPos} vanishes, and each curl in \eqref{difIdDecompPos} is equal to a cup/cap with no dots.
Therefore the two terms in \eqref{difIdDecompPos} combine with coefficient $2b - x_{-2} - \ox_{-2} = 2b + b(-2) = 0$. So assume that $\l > 0$.

\begin{align} & \dif \left( \sum_{k+l+m=\l-1} \cwcapbubcup{$k$}{$l$}{$m$}{$\l$} \right) = \nonumber \\ & \sum_{k+l+m=\l-1} (bk+x_{\l-2})
\cwcapbubcup{$k+1$}{$l$}{$m$}{$\l$} + (bl+b) \cwcapbubcup{$k$}{$_{l+1}$}{$m$}{$\l$} + (bm+\ox_{\l-2}) \cwcapbubcup{$k$}{$l$}{$m+1$}{$\l$} \label{foobar}
\end{align}

The additional terms with $\cwbubble{$1$}{}$ all cancel out, having coefficient $b + y_{\l-2} + \oy_{\l-2} = 0$. Now look in \eqref{foobar} at all the terms with at least one dot
on both the cup and the cap. Suppose that $k+l+m = \l-2$. The coefficient of \begin{align*} \cwcapbubcup{$k+1$}{$l$}{$m+1$}{$\l$}\end{align*} is precisely $bk + x_{\l-2} + bm +
\ox_{\l-2} + bl = b(k+l+m) - b(\l-2) = 0$. Therefore the only terms which survive have either no dots on top or
no dots on bottom. This should begin to smack of an equation like \eqref{difIdDecompPos}.

Let us compare coefficients for \begin{align*} \cwcapbubcup{$k+1$}{$l$}{$0$}{$\l$} \end{align*} when $k+l=\l-1$. The coefficient in \eqref{foobar} is $bk + x_{\l-2} + bl = b(\l-1)
+ x_{\l-2}$. Using \eqref{paramXalt}, this equals $b-\ox_{\l-2}$, which is precisely the coefficient of the same
term in \eqref{difIdDecompPos}, since the curl on bottom contributes this term once. A similar calculation suffices for terms with a dot on top but no dots on bottom. Finally,
let us compare coefficients for \begin{align*} \cwcapbubcup{$0$}{$\l$}{$0$}{$\l$} \end{align*} The coefficient in \eqref{foobar} is just $b\l$, while both curls contribute to the
coefficient in \eqref{difIdDecompPos}, which is $2b - x_{\l-2} - \ox_{\l-2} = b\l$ as desired.

The computation for \eqref{IdentityDecompNeg} is similar. This concludes the proof that $\dif$ is a derivation.

\paragraph{$p$-Nilpotence.} Now we ask when a derivation $\dif$ of the above form is actually a $p$-differential, when $\Bbbk$ has characteristic $p$. This amounts to showing that
$\dif^p=0$ on every generator.

Since $\dif$ preserves the nilHecke algebra, we can quote the following result from \cite[Lemma 3.6]{KQ}.

\begin{lemma} On a dot or crossing (with neighboring region $\l$), $\dif^p=0$ if and only if $b,a_\l \in \F_p \subset \Bbbk$, and either $b \ne 0$ or
$a_\l=b=0$. \hfill$\square$ \label{pnilpotentNH}\end{lemma}

Now let us iteratively apply $\dif$ to a clockwise cap. It is easy to show inductively that the coefficient of \begin{align*} \cappy{CW}{$k$}{no}{$\l$} \qquad
\textrm{inside} \qquad \dif^k \left( \cappy{CW}{$0$}{no}{$\l$} \right) \end{align*} is $x_{\l-2}(x_{\l-2}+b)(x_{\l-2}+2b) \cdots (x_{\l-2}+b(k-1))$. This coefficient
will vanish in $\dif^p$ if and only if $x_{\l-2} = -lb$ for some $l \in \F_p$. This holds if and only if $x_{\l-2} \in \F_p$, and $x_{\l-2}=0$ when $b=0$. Similarly, the
coefficient of a cap with $k$ degree $2(1)$ clockwise bubbles in $\dif^k$ is $y_{\l-2}(y_{\l-2}-b)(y_{\l-2}-2b) \cdots (y_{\l-2}-b(k-1))$. This coefficient
vanishes in $\dif^p$ iff $y_{\l-2} \in \F_p$, and $y_{\l-2}=0$ when $b=0$. The same computation for a clockwise cup gives the same results for $\ox_\l$ and
$\oy_\l$.

Therefore, if $\dif$ is a $p$-differential and $b=0$ then $\dif = 0$. We assume henceforth that $b \ne 0$. We may divide $\dif$ by $b$ and rename $\frac{x_\l}{b}$ by $x_\l$ (and
so forth). This is identical to assuming that $b=1$. In this case, we obtain a differential of the form of Definition \ref{def-dif-on-U} with requirements
\eqref{eqn-parameter-req}.

\begin{lemma}
Assume that $b=1$. If $x_{\l-2}$ and $y_{\l-2}$ are both in $\F_p$, then
\begin{align*}
\dif^p \left( \cappy{CW}{$0$}{no}{$\l$} \right) = 0.
\end{align*}
\end{lemma}

\begin{proof} A general term in $\dif^p$ has $\a$ dots, $\b_1$ bubbles of degree $2(1)$, $\b_2$ bubbles of degree $2(2)$, and so forth, such that $\a + \sum_i
i \b_i = p$. We have already shown that the coefficient is zero when $\a=p$ and the rest are zero, or when $\b_1=p$ and the rest are zero. Similarly, if
$\b_p=1$ and the rest are zero, the coefficient is also zero. This is because the degree $2p$ bubble appeared from \eqref{difcwbub}, applying $\dif$ to a degree $2(p-1)$
bubble, and in this equation it appears with coefficient $p$.

Now consider every other term. Each term can be achieved within $\dif^p$ in a number of ways: we must choose $\a$ places where $\dif$ adds a dot, $\b_1$ places
to add a bubble which is never promoted to a higher bubble, $\b_2$ size $2$ subsets where a bubble is added and then promoted to a degree $2(2)$ bubble, $\b_3$
size $3$ subsets, and so forth. The number of choices is a non-trivial multinomial coefficient, and therefore is a multiple of $p$. \end{proof}

\begin{cor}
Assume that $b=1$. The derivation $\dif$ is a $p$-differential if and only if every parameter lies in $\F_p$.  \hfill $\square$
\end{cor}

The concludes the proof of Proposition \ref{prop-classification-of-dif-on-U}.

\textbf{(Char $p=2$):} When $\a$ or $\b$ is nonzero, the nilHecke algebra is not preserved by $\dif$, and the computations above are not correct. Nonetheless, it is quick to
compute $\dif^2$ of all the generators. Assuming that $\dif$ is $2$-nilpotent, we have two possibilities: 
\begin{enumerate} 
\item If $b=0$, then $a_\l = x_\l = \ox_\l = y_\l =
\oy_\l = 0$ for all $\l$, and $\a \b = 0$. However, there are nonzero $2$-differentials even when $b=0$! 
\item If $b=1$ then all the usual parameters live in $\F_2$. In addition,
$\a=0$ and $\b(x_\l + y_\l \l)=0$. 
\end{enumerate}


\bibliographystyle{alpha}
\bibliography{qy-bib}

%

\vspace{0.1in}

\noindent B.~E.: {\sl \small Department of Mathematics, University of Oregon, Eugene, OR 97403, USA} 
\newline\noindent  {\tt \small email: belias@uoregon.edu}

\vspace{0.1in}

\noindent Y.~Q.: { \sl \small Department of Mathematics, Yale University, New
Haven, CT 06511, USA} \newline \noindent {\tt \small email: you.qi@yale.edu}

%
\end{document}